\documentclass{article}

\PassOptionsToPackage{numbers,sort&compress}{natbib}
\usepackage[preprint]{KSBTS}
\usepackage[utf8]{inputenc}
\usepackage[T1]{fontenc}

\usepackage{microtype}
\usepackage{graphicx}
\usepackage{subfigure}
\usepackage{booktabs}
\usepackage{hyperref}
\usepackage{url}
\usepackage{xcolor}
\usepackage{textcomp}
\usepackage{amsthm}
\usepackage{aliascnt}
\usepackage{amsmath}
\usepackage{amssymb}
\usepackage{wrapfig}
\usepackage{amsfonts}
\usepackage{mathtools}
\mathtoolsset{showonlyrefs}
\usepackage{nicefrac}
\usepackage{bm}
\usepackage[textsize=tiny]{todonotes}
\usepackage[capitalize,noabbrev]{cleveref}

\usepackage{float}              
\usepackage{placeins}  
\usepackage{needspace}          

\theoremstyle{plain}
\newtheorem{theorem}{Theorem}[section]

\newaliascnt{proposition}{theorem}
\newtheorem{proposition}[proposition]{Proposition}
\aliascntresetthe{proposition}

\newaliascnt{lemma}{theorem}
\newtheorem{lemma}[lemma]{Lemma}
\aliascntresetthe{lemma}

\newaliascnt{corollary}{theorem}
\newtheorem{corollary}[corollary]{Corollary}
\aliascntresetthe{corollary}

\theoremstyle{definition}

\newaliascnt{definition}{theorem}

\aliascntresetthe{definition}

\newaliascnt{assumption}{theorem}
\newtheorem{assumption}[assumption]{Assumption}
\aliascntresetthe{assumption}

\theoremstyle{remark}

\newaliascnt{remark}{theorem}
\newtheorem{remark}[remark]{Remark}
\aliascntresetthe{remark}

\crefname{theorem}{theorem}{theorems}
\Crefname{theorem}{Theorem}{Theorems}

\crefname{proposition}{proposition}{propositions}
\Crefname{proposition}{Proposition}{Propositions}

\crefname{lemma}{lemma}{lemmas}
\Crefname{lemma}{Lemma}{Lemmas}

\crefname{corollary}{corollary}{corollaries}
\Crefname{corollary}{Corollary}{Corollaries}

\crefname{definition}{definition}{definitions}
\Crefname{definition}{Definition}{Definitions}

\crefname{assumption}{assumption}{assumptions}
\Crefname{assumption}{Assumption}{Assumptions}

\crefname{remark}{remark}{remarks}
\Crefname{remark}{Remark}{Remarks}

\newcommand{\R}{\mathbb{R}}
\newcommand{\E}{\mathbb{E}}
\renewcommand{\P}{\mathbb{P}}

\newcommand{\avg}{\text{avg}}
\newcommand{\Var}{\text{Var}}
\newcommand{\TruncNormal}{\text{TruncNormal}}
\newcommand{\ISE}{\text{ISE}}

\title{Direct Estimation of Schrödinger Bridge Time-Series Drifts:
Finite-Sample, Asymptotic, and Adaptive Guarantees}

\author{%
Othmane MAZHAR\thanks{This author was supported by the BNP-Paribas Chair ``Futures of Quantitative Finance''.}\\
LPSM, Sorbonne Universit\'e and Universit\'e Paris Cit\'e\\
\texttt{othmane.xx90@gmail.com}
\And
Huy\^en PHAM\thanks{This author is supported by the Chair ``Financial Risks'', by FiME (Laboratory of Finance and Energy Markets), and the EDF--CACIB Chair ``Finance and Sustainable Development''.}\\
\'{E}cole Polytechnique, CMAP\\
\texttt{huyen.pham@polytechnique.edu}
}

\begin{document}

\maketitle
\begin{abstract}
We study nonparametric estimation of Schr\"odinger bridge (SB) drifts from i.i.d.\ data observed on a single time interval. Starting from the conditional-ratio form of the Schr\"odinger bridge time-series (SBTS) drift formula, we analyze a direct Nadaraya--Watson plug-in estimator built from kernelized numerator and denominator terms. Unlike recent SB analyses based on entropic-OT potentials, Sinkhorn iterations, or iterative bridge solvers, our approach works directly at the drift level and isolates \emph{statistical error} from optimization, approximation, and discretization error.

Under H\"older regularity, a marginal-density floor, and bounded support, we prove a uniform non-asymptotic bound for admissible bandwidth pairs, a pointwise CLT under genuine undersmoothing, and an adaptive bandwidth selector satisfying an oracle inequality. We also prove a pivot-local minimax lower bound which, through an explicit uniform pivot, yields a global minimax lower bound under transparent compatibility conditions; hence the adaptive selector is minimax-rate optimal up to logarithmic factors. Synthetic experiments provide theorem-targeted diagnostics for finite-sample scaling, Gaussian approximation, and adaptive behavior.
\end{abstract}

\section{Introduction}

Schr\"odinger bridges (SBs) interpolate between probability distributions by minimizing relative
entropy on path space and connect stochastic control, entropic optimal transport, and diffusion-based
generative modeling
\cite{schrodinger1931umkehrung,jamison1975markov,follmer1988randomfields,
leonard2014survey,chen2021siamreview}. They have recently been used for generative modeling,
sampling, and time-series generation
\cite{debortoli2021sb,strome2023sampling,hamdouche2023sbts,alouadi2025robust}.
Yet the \emph{statistical theory} of data-driven SB estimators remains limited, especially for
\emph{direct nonparametric drift estimation}.

Most existing analyses do not treat the drift as the primary statistical object. They instead proceed
through entropic-OT potentials, Sinkhorn iterations, or iterative bridge solvers, so the resulting
guarantees are intertwined with optimization, approximation, or discretization error. We take a
different route: starting from the closed-form conditional-ratio representation of the SB drift, we
analyze a direct kernel plug-in estimator and study the \emph{statistical error of the drift itself}.
This drift-level viewpoint is also essential for sharpness: our lower bound perturbs the drift directly
while preserving normalization, denominator separation, and the global smoothness class.

\paragraph{Problem and scope.}
Fix one interval $[s,u]$ with $0<s<u$. In the first-order SBTS notation of
\cite{hamdouche2023sbts}, the drift on $[s,u)$ depends on the observed past through
$\xi=X_s$. Given $M$ i.i.d. samples
$\{(X_s^{(m)},X_u^{(m)})\}_{m=1}^M\sim\nu$, our goal is to estimate
$a^*(t,x;\xi)$ and quantify its accuracy. This first-order notation is only for readability:
fixed-memory extensions replace $\xi\in\mathbb R^d$ by a lag block in $(\mathbb R^d)^L$,
with $d$ replaced by $Ld$, and finite-grid bounds follow intervalwise by a union bound.

\paragraph{Why this is nontrivial.}
The drift is a \emph{ratio} of conditional expectations, so numerator and denominator errors must
be transferred stably and empirical denominators must inherit deterministic population floors.
Moreover, the kernelized terms are weighted by an SB exponential with both \emph{contracting}
and \emph{expanding} factors, and the drift is singular as $t\uparrow u$. Bounded support is therefore
central: it controls the expanding factor and makes the weighted empirical class analyzable.

\paragraph{Contributions.}
Our contributions are fourfold.
\begin{enumerate}
    \item \textbf{Direct drift-level analysis.}
    We analyze a kernel plug-in estimator derived from the SBTS conditional-ratio drift formula,
    decoupling statistical error from Sinkhorn stability, iterative solver error, and SDE discretization.

    \item \textbf{Uniform finite-sample guarantees.}
    Under H\"older smoothness, a marginal-density floor, and bounded support, we prove a uniform
    non-asymptotic bound over space and time away from the terminal singularity, with rate
    $(\log M/M)^{\beta/(2\beta+d)}$ up to constants, the terminal factor, and logarithmic terms.

    \item \textbf{Pointwise inference.}
    We prove
    $\sqrt{Mh^d}\,(\hat a(t,x;\xi)-a^*(t,x;\xi))
    \Rightarrow \mathcal N(0,\Sigma(t,x;\xi))$
    under genuine undersmoothing and give a plug-in variance formula for confidence intervals.

    \item \textbf{Adaptive minimax optimality.}
    We construct a data-driven bandwidth selector satisfying an oracle inequality. A pivot-local
    Le Cam lower bound shows that the rate cannot be improved around interior laws, and an explicit
    uniform pivot yields a global minimax lower bound under transparent compatibility conditions.
    Thus the adaptive selector is globally minimax-rate optimal up to logarithmic factors.
\end{enumerate}

The empirical section is theorem-aligned rather than benchmark-driven: synthetic experiments with
known drifts assess finite-sample scaling, pointwise Gaussian approximation, confidence-interval
coverage, and oracle-competitiveness of the adaptive selector.

\section{Related Work}

Recent SB theory and algorithms typically proceed through entropic optimal transport, Sinkhorn-type
iterations, or iterative bridge solvers; see
\cite{schrodinger1931umkehrung,jamison1975markov,follmer1988randomfields,
leonard2014survey,chen2021siamreview} for background and
\cite{cuturi2013sinkhorn,debortoli2021sb,liu_etal_2024_gsbm,strome2023sampling}
for modern computational, generative, sampling, and bridge-matching perspectives. In that line of work, the drift is usually obtained only after
estimating potentials or solving an auxiliary optimization problem, so the final guarantees are coupled
with computational stability, iteration error, or discretization error. Our paper takes a different route:
starting from the closed-form SBTS drift representation of \cite{hamdouche2023sbts}, we study the
resulting conditional-ratio drift formula and estimate the drift directly.

Recent statistical and nonparametric analyses in the SB/entropic-OT literature include sample-complexity
and CLT results for entropic OT and Sinkhorn divergences \cite{genevay2019sinkhorn,mena2019entropic},
plug-in estimation of Schr\"odinger bridges via entropic-OT potentials \cite{pooladian2024plugin}, and
forward--reverse kernel methods for Schr\"odinger-type quantities \cite{belomestny2025forwardreverse}.
In the time-series setting, \cite{hamdouche2023sbts} provides the SBTS path-space formalism,
optimal-drift characterization, and conditional-ratio representation used here, while
\cite{alouadi2025robust} provides extensive empirical benchmarks for SBTS generation. Our repeated-sample
formulation matches this SBTS data model, is standard in sample-based OT/SB studies
\cite{liu_etal_2024_gsbm,pooladian2024plugin}, and parallels copies-based drift estimation
\cite{marie_rosier_2023_nw_iid_paths,panloup_tindel_varvenne_2020_drift}. These works are complementary
to ours: our target is the drift estimator itself, our main setting is the SBTS single-interval formulation,
and our guarantees are stated directly at the drift level.

\paragraph{Technical distinction from standard kernel theory.} Our analysis uses classical kernel and
pointwise-asymptotic ideas \cite{gyorfi2002,tsybakov2009nonparametric}, together with adaptive
bandwidth-selection ideas \cite{lepski1997,goldenshluger2011}, but these tools do not apply off the shelf.
The estimand is a drift-level ratio of weighted conditional moments, and the SB weight contains both a
contracting Gaussian term and an expanding exponential term. The proof therefore needs a deterministic
ratio-transfer step, empirical denominator localization, and a weighted-class entropy argument showing
that the SB-weighted kernel class retains the entropy order of translated kernels. The lower bound also
differs from standard regression lower bounds: the perturbation must change the drift while preserving
normalization, the H\"older ball, the marginal-density floor, and the automatic SB denominator floor.
These are the SB-specific ingredients that allow the finite-sample, CLT, adaptive, and minimax results
to hold directly at the drift level.

\section{Problem Setup and Estimator}
\label{sec:setup}

Recall the SBTS construction of \cite{hamdouche2023sbts}. For an observation grid
$\mathcal T=\{t_1,\ldots,t_N\}$ and a target time-series law $\mu$ on $(\R^d)^N$, the bridge is the
path-space entropy projection
\begin{equation}
P^*\in\arg\min_P
\left\{
H(P\mid W):(X_{t_1},\ldots,X_{t_N})_\#P=\mu
\right\},
\end{equation}
where $W$ is the reference Wiener law. The constraint is therefore imposed on the full joint law of
the observed time series. Writing $\mu^W_{\mathcal T}$ for the grid law under $W$ and
$\varrho:=d\mu/d\mu^W_{\mathcal T}$, the optimizer is the tilt
$dP^*/dW=\varrho(X_{t_1},\ldots,X_{t_N})$. The stochastic-control formulation gives an adapted
log-gradient drift, and \cite{hamdouche2023sbts} derive from it the conditional-ratio formula used
below. We study this formula on a fixed interval $[s,u]$, with $0<s<u$. Write $\Delta:=u-s$ and
$\Delta(t):=u-t$ for $t\in[s,u)$. We observe $M$ i.i.d.\ copies
$\{(X_s^{(m)},X_u^{(m)})\}_{m=1}^M$ of an $\R^d\times\R^d$-valued random pair $(X_s,X_u)$
with joint law $\nu$. Throughout, $B\subset\R^d$ denotes a bounded set containing the support of
both $X_s$ and $X_u$.

For readability, we use the first-order notation of the SBTS drift formula: on the observation grid,
the conditional law of the next state depends on the past only through the current endpoint. Thus,
with $s=t_i$, $u=t_{i+1}$, and $\xi=X_s$, the drift on $[s,u)$ is
\begin{equation}
\label{eq:sb-drift}
a^*(t,x;\xi)
=
\frac{1}{\Delta(t)}
\left(
\frac{\E\!\left[X_u\ F(t,\xi,x,X_u)\mid X_s=\xi\right]}
{\E\!\left[F(t,\xi,x,X_u)\mid X_s=\xi\right]}
-x
\right),
\ 
F(t,\xi,x,y):=e^{-\frac{|y-x|^2}{2\Delta(t)}+\frac{|y-\xi|^2}{2\Delta}}.
\end{equation}
This conditional-ratio formula is the starting point of the paper: we estimate the drift directly from
\eqref{eq:sb-drift}, rather than through potentials, Sinkhorn iterations, or an auxiliary bridge solver.

Let $p$ denote the joint density of $(X_s,X_u)$ on $B\times B$, let $f(\xi):=\int_B p(\xi,y)\ dy$ be the marginal density of $X_s$, and define
\begin{align}
\label{eq:g1-pop}
g_1(t,x;\xi)
&:=\int_B F(t,\xi,x,y)p(\xi,y)\ dy
= f(\xi)\E\!\left[F(t,\xi,x,X_u)\mid X_s=\xi\right],\\
\label{eq:g2-pop}
g_2(t,x;\xi)
&:=\int_B y\ F(t,\xi,x,y)p(\xi,y)\ dy
= f(\xi)\E\!\left[X_uF(t,\xi,x,X_u)\mid X_s=\xi\right].
\end{align}
Set $D^*(t,x;\xi):=g_1(t,x;\xi)/f(\xi)$, $N^*(t,x;\xi):=g_2(t,x;\xi)/f(\xi)$, and $Q^*(t,x;\xi):=N^*(t,x;\xi)/D^*(t,x;\xi)$. Then
\begin{equation}
\label{eq:adrift-pop}
a^*(t,x;\xi)
=
\frac{1}{\Delta(t)}\left(Q^*(t,x;\xi)-x\right)
=
\frac{1}{\Delta(t)}
\left(\frac{N^*(t,x;\xi)}{D^*(t,x;\xi)}-x\right).
\end{equation}

\paragraph{Kernel plug-in estimator.}
Fix a possibly signed kernel $K : \mathbb R^d \to \mathbb R$ and bandwidths
$h_1,h_2>0$, with $K_h(z):=h^{-d}K(z/h)$. For $(t,x,\xi)\in[s,u)\times\R^d\times B$, define
\begin{align}
\label{eq:fhat}
\hat f_j(\xi)&:=\frac{1}{M}\sum_{m=1}^M K_{h_j}(X_s^{(m)}-\xi),\qquad j\in\{1,2\},\\
\label{eq:ghat1}
\hat g_1(t,x;\xi)&:=\frac{1}{M}\sum_{m=1}^M F(t,\xi,x,X_u^{(m)})\ K_{h_1}(X_s^{(m)}-\xi),\\
\label{eq:ghat2}
\hat g_2(t,x;\xi)&:=\frac{1}{M}\sum_{m=1}^M X_u^{(m)}\ F(t,\xi,x,X_u^{(m)})\ K_{h_2}(X_s^{(m)}-\xi).
\end{align}
Let $\hat D(t,x;\xi):=\hat g_1(t,x;\xi)/\hat f_1(\xi)$ and $\hat N(t,x;\xi):=\hat g_2(t,x;\xi)/\hat f_2(\xi)$. The direct Nadaraya--Watson plug-in estimator
\cite{nadaraya1964estimating,watson1964smooth} is
\begin{equation}
\label{eq:adrift-hat}
\hat a(t,x;\xi):=
\frac{1}{\Delta(t)}
\left(\frac{\hat N(t,x;\xi)}{\hat D(t,x;\xi)}-x\right).
\end{equation}
For risk statements we use the denominator-floor convention stated in Appendix~\ref{app:adaptive}; it is inactive on the high-probability events used below.
We allow $h_1\neq h_2$ because the two weighted functionals may have different effective envelopes and variance levels. In the main-text CLT and adaptive results, however, we specialize to the diagonal choice $h_1=h_2=h$ to keep the bandwidth parameter one-dimensional. The unequal-bandwidth CLT, and adaptive versions are given in Appendix~\ref{app:clt} to~\ref{app:lower}.

\subsection{Assumptions}

We state assumptions for the \emph{SBTS single-interval model} above.

\begin{assumption}[Kernel]
\label{ass:kernel}
Let $\ell_\beta:=\max\{m\in\mathbb N_0:m<\beta\}$ and $\theta_\beta:=\beta-\ell_\beta\in(0,1]$.
The kernel $K:\mathbb R^d\to\mathbb R$ is bounded, compactly supported,
Lipschitz with constant $L_K$, integrates to one, and is a multivariate
kernel of order $\ell_\beta$:
$\int_{\mathbb R^d}z^\alpha K(z)\,dz=0,
$ $ 1\le |\alpha|\le \ell_\beta$ .
Moreover,
$\mu_\beta(K):=\int_{\mathbb R^d}|z|^\beta |K(z)|\,dz<\infty$ .
\end{assumption}

\begin{assumption}[Pair-law regularity]
\label{ass:pair}
The pair $(X_s,X_u)$ admits a joint density $p$ supported on $B\times B$, where
$B\subset\mathbb{R}^d$ is bounded. There exist constants $\beta>0$, $L_p>0$, and
$f_{\min}>0$ such that:
\begin{enumerate}
\item for every $y\in B$, the map $\xi\mapsto p(\xi,y)$ on $B$ is the restriction
of a function $\bar p(\cdot,y)\in \Sigma(\beta,L_p)$, with
$\Sigma(\beta,L_p)$ defined in Appendix~\ref{Appendix A}, on a neighborhood of $B$,
uniformly in $y\in B$;
\item the marginal density $f(\xi):=\int_B p(\xi,y)\,dy$
satisfies
$\inf_{\xi\in B} f(\xi)\ge f_{\min}>0$.
\end{enumerate}
Consequently, denominator separation holds automatically on compact query sets. Indeed, for every
$R>0$ and $\eta\in(0,\Delta)$, writing $B_R:=\{x\in\mathbb R^d: |x|\le R\}$, the explicit form of $F$ and the boundedness of $B$ imply that
there exists a deterministic constant $D_{\min}=D_{\min}(B,R,\eta,\Delta)>0$ such that
$\inf_{\xi\in B,\ x\in B_R,\ t\in[s,u-\eta]}
D^*(t,x;\xi)
= \inf_{\xi\in B,\ x\in B_R,\ t\in[s,u-\eta]}
\E\!\left[F(t,\xi,x,X_u)\mid X_s=\xi\right]
\ge 2D_{\min}$.
Thus $D_{\min}$ denotes this deterministic constant and is not an additional model parameter.
\end{assumption}

\begin{remark}[Why bounded support is needed]
\label{rem:bounded-support}
The bounded-support assumption is essential for our proof strategy. The SB weight in \eqref{eq:sb-drift} contains both the contracting factor $e^{-|y-x|^2/(2\Delta(t))}$ and the expanding factor $e^{|y-\xi|^2/(2\Delta)}$. Generic sub-Gaussian tails do not provide the uniform control needed for the weighted empirical-process bounds in our argument. Bounded support makes the expansion term uniformly bounded and allows us to control the weighted class at the same entropy order as the classical translated-kernel class. In particular,
\begin{equation}
C_F:=\sup_{\xi,y\in B}e^{\frac{|y-\xi|^2}{2\Delta}}<\infty,
\qquad
C_Y:=\sup_{y\in B}|y|<\infty.
\end{equation}
\end{remark}

\begin{remark}[Fixed finite-memory extensions]
The first-order display is used only to keep notation readable. For any fixed
memory length $L$, replace the conditioning variable $\xi\in\mathbb R^d$ by the
lag block
\begin{equation}
h_i^{(L)}=(X_{t_{i-L+1}},\ldots,X_{t_i})\in(\mathbb R^d)^L
\end{equation}
and replace the pair law of $(X_{t_i},X_{t_{i+1}})$ by the joint law of
$(h_i^{(L)},X_{t_{i+1}})$. The ratio algebra is unchanged, and the kernel bias,
entropy, concentration, variance, CLT, and adaptive arguments carry over with
the conditioning dimension $d$ replaced by $Ld$. For a fixed interval $i$, the same applies to
the full finite history, with $d$ replaced by $id$. What is not covered here is estimation of the SB
drift from a single discretely observed trajectory; that problem has dependent design points and
requires different empirical-process and time-series arguments.
\end{remark}

\paragraph{Bandwidth scale.}
The oracle bandwidth scale is $h_1\asymp h_2\asymp (\log M/M)^{1/(2\beta+d)}$, up to model-dependent constants and the precise logarithmic term in the finite-sample theorem. For the pointwise CLT, however, the bandwidth must be genuinely undersmoothed and therefore must decay strictly faster than this rate-optimal scale.

\section{Main Results}
\label{sec:main}

We now present finite-sample, asymptotic, and adaptive guarantees for the estimator \eqref{eq:adrift-hat} in the
SBTS single-interval setting of Section~\ref{sec:setup}. Fix $R>0$, $\eta\in(0,\Delta)$, and $\rho>0$, and write
$B_R:=\{x\in\mathbb R^d: |x|\le R\}$. If
$\operatorname{supp}(K)\subset B(0,R_K)$, set
$B_\rho:=\{\xi\in B:\operatorname{dist}(\xi,\partial B)\ge \rho\},
\ Q_{R,\eta,\rho}:=[s,u-\eta]\times B_R\times B_\rho$ .
All uniform suprema below are taken over $Q_{R,\eta,\rho}$, and bandwidths satisfy
$h_j\le \rho/R_K$. Constants may also depend on $\rho$, but not on $(M,h_1,h_2)$.

\subsection{Deterministic ratio transfer}

The following proposition isolates the deterministic step that turns estimation error for the kernelized building blocks into estimation error for the drift.

\begin{proposition}[Deterministic ratio-stability transfer]
\label{prop:ratio-transfer}
Assume \Cref{ass:pair}. Let
\begin{equation}
\mathcal E
:= \left\{
\inf_{\xi\in B_\rho}\hat f_1(\xi)\ge \frac{f_{\min}}2,\quad
\inf_{\xi\in B_\rho}\hat f_2(\xi)\ge \frac{f_{\min}}2,\quad
\inf_{(t,x,\xi)\in Q_{R,\eta,\rho}}\hat D(t,x;\xi)\ge \frac{D_{\min}}2
\right\}.
\end{equation}
Then on $\mathcal E$, for every $(t,x,\xi)\in Q_{R,\eta,\rho}$,
\begin{align*}
|\hat a(t,x;\xi)-a^*(t,x;\xi)|
\le
&\frac{4}{\Delta(t)D_{\min}}
\Bigg[\frac{f(\xi)|\hat g_2(t,x;\xi)-g_2(t,x;\xi)|+|g_2(t,x;\xi)|\ |\hat f_2(\xi)-f(\xi)|}{f_{\min}^2}
\\
&+ |Q^*(t,x;\xi)|
\frac{f(\xi)|\hat g_1(t,x;\xi)-g_1(t,x;\xi)|+|g_1(t,x;\xi)|\ |\hat f_1(\xi)-f(\xi)|}{f_{\min}^2}
\Bigg].
\end{align*}
\end{proposition}

Under \Cref{ass:pair}, the quantities $f$, $g_1$, $g_2$, and $Q^*$ are uniformly bounded on $Q_{R,\eta,\rho}$, so \Cref{prop:ratio-transfer} yields the simpler upper bound for $|\hat a(t,x;\xi)-a^*(t,x;\xi)|$ on $\mathcal E$,
\begin{equation}
\frac{C}{\Delta(t)}
\Big(
|\hat g_2(t,x;\xi)-g_2(t,x;\xi)|+|\hat f_2(\xi)-f(\xi)|
+|\hat g_1(t,x;\xi)-g_1(t,x;\xi)|+|\hat f_1(\xi)-f(\xi)|
\Big).
\end{equation}

\subsection{Finite-sample uniform bound}

Let $C_K>0$ be the entropy constant from Appendix~\ref{Appendix A} and define
\begin{equation}
\Lambda(h,\delta):=d\log\!\Big(\frac{C_K}{h}\Big)+\log\!\Big(\frac{2}{\delta}\Big),
\ 
\widetilde\Lambda(h,\delta):=d\log\!\Big(\frac{C_K}{h}\Big)+\log\!\Big(\frac{8}{\delta}\Big), \  r_{e,j}(\delta)
:= h_j^\beta
+ \sqrt{\frac{\widetilde\Lambda(h_j,\delta)}{M h_j^d}},
\end{equation}
for $j \in \{1,2\}$.
We call $(h_1,h_2)$ $\delta$-admissible if
$r_{e,1}(\delta)+r_{e,2}(\delta)\le c_0$.
\begin{theorem}[Uniform finite-sample bound]
\label{thm:finite}
Under \Cref{ass:kernel,ass:pair}, there exist constants $c_0,C>0$ such that, for every $\delta\in(0,1)$ and every $\delta$-admissible bandwidth pair $(h_1,h_2)$, with probability at least $1-\delta$,
\begin{equation}
\sup_{(t,x,\xi)\in Q_{R,\eta,\rho}}
|\hat a(t,x;\xi)-a^*(t,x;\xi)|
\le
\frac{C}{\eta}
\left[
h_1^\beta+h_2^\beta
+\sqrt{\frac{\Lambda(h_1,\delta)}{M h_1^d}}
+\sqrt{\frac{\Lambda(h_2,\delta)}{M h_2^d}}
\right].
\end{equation}
\end{theorem}

The key difficulty is that the SB weight is not a standard bounded regression weight: it contains both a contracting Gaussian and an expanding exponential. Our proof handles this by factorizing $F(t,\xi,x,y)=\chi(t,x,y)G(\xi,y)$ with $\chi(t,x,y):=e^{-|y-x|^2/(2\Delta(t))}\in[0,1]$ and $G(\xi,y):=e^{|y-\xi|^2/(2\Delta)}\le C_F$ on $B\times B$. Bounded support controls the expanding term, while the contraction term yields both Lipschitz control in the parameters and decay in the unbounded state variable $x$. As a result, the weighted empirical-process class has the same entropy order as the classical translated-kernel class $\{z\mapsto K_h(z-\xi):\xi\in B_\rho\}$.

\begin{corollary}[Oracle bandwidth]
\label{cor:rate}
If $h_1=h_2=h^\star:=\bigl((d\log(C_K M)+\log(2/\delta))/M\bigr)^{1/(2\beta+d)}$, then for each fixed $\delta\in(0,1)$ there exists $M_0(\delta)\ge 1$, depending on $\delta$ and the model constants, such that for all $M\ge M_0(\delta)$ the pair $(h^\star,h^\star)$ is $\delta$-admissible and, with probability at least $1-\delta$,
\begin{equation}
\sup_{(t,x,\xi)\in Q_{R,\eta,\rho}}
| \hat a(t,x;\xi)-a^*(t,x;\xi) |
\le
\frac{C}{\eta}
\left(
\frac{d\log(C_K M)+\log(2/\delta)}{M}
\right)^{\beta/(2\beta+d)}.
\end{equation}
In particular, for $\delta=M^{-\gamma}$ with fixed $\gamma>0$,
\begin{equation}
\sup_{(t,x,\xi)\in Q_{R,\eta,\rho}} |\hat a-a^*|
=
O_{\P}\!\left(\Big(\tfrac{\log M}{M}\Big)^{\beta/(2\beta+d)}\right).
\end{equation}
\end{corollary}

\begin{remark}[Finite observation grids]
For a grid $0<t_1<\cdots<t_N=T$ with first-order Markov observations, construct
the interval-specific estimator $\hat a_i$ from the pairs
$(X_{t_i},X_{t_{i+1}})$. If the interval laws satisfy \Cref{ass:pair} with constants and support set uniform in $i$, Theorem~\ref{thm:finite} applies on each interval with the same
constants. Applying it with confidence level $\delta/(N-1)$ and taking a union
bound gives the simultaneous intervalwise bound with
\begin{equation}
\Lambda_N(h,\delta)
=
d\log\!\left(\frac{C_K}{h}\right)
+
\log\!\left(\frac{2(N-1)}{\delta}\right).
\end{equation}
Thus finite-grid extensions cost only the logarithmic factor $\log(N-1)$.
\end{remark}

\begin{remark}[Time-edge singularity]
\label{rem:eta}
The factor $1/\eta$ reflects the intrinsic singularity of the SB drift as $t\uparrow u$. The statistical rate in $M$ remains the classical kernel rate up to logarithmic factors.
\end{remark}

\begin{remark}[Diagonal specialization]
\label{rem:diagonal-specialization}
The estimator \eqref{eq:adrift-hat} and \Cref{thm:finite} allow different bandwidths $h_1$ and $h_2$, but in \Cref{sec:clt,sec:adaptive} we specialize to the diagonal choice $h_1=h_2=h$ so that inference and adaptation are one-dimensional and match the experimental protocol. The unequal-bandwidth versions are given in Appendix~\ref{app:clt} to~\ref{app:lower}.
\end{remark}

\subsection{Pointwise asymptotic normality}
\label{sec:clt}

In the main text we specialize to $h_1=h_2=h$; the unequal-bandwidth statement is given in Appendix~\ref{app:clt}. Let $R(K):=\int_{\R^d}K(z)^2\ dz$. For fixed $(t,x,\xi)\in[s,u-\eta]\times B_R\times\operatorname{int}(B)$, define
\begin{equation}
\psi_{t,x,\xi}(y):=\bigl(y-x-\Delta(t)a^*(t,x;\xi)\bigr)F(t,\xi,x,y),
\quad 
\Sigma(t,x;\xi):=
\frac{R(K)}{f(\xi)}
\frac{\Var(\psi_{t,x,\xi}(X_u)\mid X_s=\xi)}{\Delta(t)^2D^*(t,x;\xi)^2}.
\end{equation}

\begin{theorem}[Pointwise CLT]
\label{thm:clt}
Assume \Cref{ass:kernel,ass:pair}, and let $h_1=h_2=h\to0$ satisfy $Mh^d\to\infty$ and $h^\beta\sqrt{Mh^d}\to0$. Then for every fixed $(t,x,\xi)\in[s,u-\eta]\times B_R\times\operatorname{int}(B)$,
\begin{equation}
\sqrt{Mh^d}\ \bigl(\hat a(t,x;\xi)-a^*(t,x;\xi)\bigr)\Rightarrow \mathcal N\!\bigl(0,\Sigma(t,x;\xi)\bigr).
\end{equation}
\end{theorem}

\begin{remark}[Undersmoothing is genuinely required]
\label{rem:undersmoothing}
The condition $h^\beta\sqrt{Mh^d}\to0$ is the scaled-bias requirement for the CLT. Taking only a smaller constant multiple of the rate-optimal bandwidth does not suffice as it preserves the same exponent and therefore does not make the scaled bias vanish.
\end{remark}

\subsection{Adaptive oracle inequality and minimax optimality}
\label{sec:adaptive}

Again following \Cref{rem:diagonal-specialization}, we specialize in the main text to $h_1=h_2=h$. Let $\mathcal H=\{h_0 2^{-k}\}_{k=0}^{K_M}\subset(0,1]$ be a finite geometric bandwidth grid with largest element $h_0\asymp1$ and smallest element of order $M^{-1/d}$. For $h\in\mathcal H$, let $\hat a_h$ denote \eqref{eq:adrift-hat} with $h_1=h_2=h$, and define
\begin{equation}
\mathcal R_\nu(\hat a_h):=
\sup_{(t,x,\xi)\in Q_{R,\eta,\rho}}
\E_\nu\!\bigl[|\hat a_h(t,x;\xi)-a^*(t,x;\xi)|\bigr].
\end{equation}
Let $\mathfrak P(\beta,L_p,f_{\min},B)$ denote the class of pair laws satisfying \Cref{ass:pair} with the displayed constants.

\begin{theorem}[Oracle inequality]
\label{thm:oracle}
Under \Cref{ass:kernel,ass:pair}, there exist constants $M_0\ge 1$ and $C_1,C_2>0$ such that, for
all $M\ge M_0$, the diagonal GL/Lepski selector
$\widehat h\in\mathcal H$ described in Appendix \ref{app:adaptive} satisfies
\begin{equation}
\sup_{\nu\in \mathfrak P(\beta,L_p,f_{\min},B)}
\mathcal R_\nu(\hat a_{\widehat h})
\le
C_1 \inf_{h\in \mathcal H}
\left(
h^\beta+\sqrt{\frac{\log M}{M h^d}}+\frac{\log M}{M h^d}
\right)
+ C_2 M^{-1}.
\end{equation}
\end{theorem}

\begin{corollary}[Adaptive minimax optimality]
\label{cor:adaptive}
If $\mathcal H=\{h_0 2^{-k}: k=0,1,\ldots,K_M\}$, with $h_0\asymp 1$ and smallest element of order
$M^{-1/d}$, then for all sufficiently large $M$,
\begin{equation}
\sup_{\nu\in \mathfrak P(\beta,L_p,f_{\min},B)}
\mathcal R_\nu(\hat a_{\widehat h})
\lesssim
\left(\frac{\log M}{M}\right)^{\beta/(2\beta+d)} .
\end{equation}

Moreover, for every interior pivot law $\nu_0\in \mathfrak P(\beta,L_p,f_{\min},B)$ satisfying the slack conditions of \Cref{prop:lower}, there
exist constants $Q_0>0$, $c_{\rm loc}>0$, and $c>0$ such that, for all large $M$,
\begin{equation}
\inf_{\widetilde a} \sup_{\nu\in \mathfrak P_{\mathrm{loc},M}^{\cap} (\nu_0;Q_0,c_{\rm loc}M^{-\beta/(2\beta+d)})} \mathcal R_\nu(\widetilde a)
\ge c M^{-\beta/(2\beta+d)} .
\end{equation}
Thus the adaptive rate is locally sharp around every such interior pivot.

If, in addition, the compatibility conditions of \Cref{cor:global-minimax-explicit-pivot} hold, then the same local construction,
instantiated at the explicit uniform pivot, yields the global minimax lower bound
\begin{equation}
\inf_{\widetilde a}
\sup_{\nu\in \mathfrak P(\beta,L_p,f_{\min},B)}
\mathcal R_\nu(\widetilde a)
\gtrsim
M^{-\beta/(2\beta+d)} .
\end{equation}
Thus,
$\hat a_{\widehat h}$
is globally minimax-rate optimal over
$\mathfrak P(\beta,L_p,f_{\min},B)$, up to logarithmic factors.
\end{corollary}

\begin{remark}[Oracle interpretation and minimax interpretation]
The adaptive theorem controls the estimator up to multiplicative constants and lower-order terms.
Thus it supports rate-level competitiveness with oracle tuning, but not finite-sample equality with the
oracle bandwidth.

The lower-bound argument has two layers. \Cref{prop:lower} is a pivot-local Le Cam construction:
around any interior law satisfying the slack conditions, no estimator can improve on the rate
\(M^{-\beta/(2\beta+d)}\) for the uniform risk. \Cref{cor:global-minimax-explicit-pivot} then instantiates this construction at an
explicit uniform pivot contained in the global class, under transparent compatibility conditions. Since
that local neighborhood is a subset of the global model, the same construction yields a global
minimax lower bound over
\(\mathfrak P(\beta,L_p,f_{\min},B)\).

Thus the adaptive selector is not only locally sharp around interior pivots; under the explicit
compatibility conditions of \Cref{cor:global-minimax-explicit-pivot}, it is globally minimax-rate optimal up to logarithmic
factors.
\end{remark}

\begin{remark}[Experiment-facing proxy]
\label{rem:sup-grid-proxy}
The adaptive theorem controls uniform risk over $(t,x,\xi)$. Exact evaluation of
that risk is not available in experiments, so
we use the fixed-grid sup error
\begin{equation}\label{eq:exp_prox}
E_{\infty,\mathcal G}(t,\xi;h)
:= \max_{x\in\mathcal G_x}
|\hat a_h(t,x;\xi)-a^*(t,x;\xi)| \qquad \text{for fixed prescribed interior pairs $(t,\xi)$}.
\end{equation}
It tests the same one-sided GL behavior on a theorem-facing proxy, but should
not be read as a direct finite-sample verification of the full uniform oracle inequality.
\end{remark}

\section{Experiments}
\label{sec:experiments}
We use controlled synthetic experiments in the SBTS single-interval setting to test the theorems predictions: finite-sample scaling,
pointwise Gaussian approximation under undersmoothing, adaptive bandwidth
behavior, and terminal-edge deterioration. The goal is not to re-benchmark SBTS
as an end-to-end generator, which has already been studied in
\cite{hamdouche2023sbts,alouadi2025robust}, but to isolate the
statistical behavior of the direct drift estimator from optimization,
discretization, simulation, and potential-to-drift recovery error. We therefore
use synthetic bridge families with cached ground-truth drifts.

\paragraph{Setups and metrics.}
We consider two synthetic families with numerically cached ground-truth drifts: a Gaussian-to-Gaussian bridge (GG) and a Mixture-to-Mixture bridge (MM). We report results for $d\in\{1,2\}$ on fixed bounded evaluation boxes $\mathcal X_{\rm eval}\subset\R^d$, using the product Epanechnikov kernel and sample sizes $M\in\{10^3,2\times10^3,4\times10^3,8\times10^3\}$; exact model parameters, grids, repetition counts, and interior query points are given in Appendix~\ref{app:exp-details}.

For rate and adaptivity, the primary loss is the sup-grid error
$E_{\infty,\mathcal G}(t_0,\xi_0;h)$ defined in \eqref{eq:exp_prox},
which is structurally closer to the theorem-side uniform risk than the integrated squared error. For rate verification we fit the slope of $\log \overline E_{\infty,\mathcal G}(M)$ versus $\log M$ and compare it to the finite-range slope induced by $(\log M/M)^{\beta/(2\beta+d)}$; since the implementation uses the product Epanechnikov kernel, the benchmark is the $\beta=2$ absolute-error rate.

For the CLT experiment we use
$Z_M:=\sqrt{M h_M^d}\ (\hat a_{h_M}(t_0,x_0;\xi_0)-a^\ast(t_0,x_0;\xi_0))/\sqrt{\hat\Sigma(t_0,x_0;\xi_0)}$ as the standardized statistic,
with $\hat\Sigma$ the plug-in variance estimator from Appendix~\ref{app:exp-details}. The bandwidth sequence $h_M$ is chosen from a genuinely undersmoothed regime satisfying $M h_M^d\to\infty$ and $h_M^\beta\sqrt{M h_M^d}\to0$. In the final reported one-dimensional runs we fix $h_M=M^{-\alpha}$ with $\alpha=0.22$ for GG1 and $\alpha=0.28$ for MM1 after a short pilot screen on the larger sample sizes. The final CLT summaries use $N_{\rm rep}=300$ Monte Carlo repetitions.

For adaptivity we compare the empirical oracle bandwidth $h_{\rm or}$, obtained by minimizing the empirical error against the cached truth on a fixed geometric grid, with a single canonical theorem-aligned selector $\hat h$ on the same grid. We use the raw-max, one-sided GL-type rule with $\kappa_{\mathrm{pair}}=\kappa_{\mathrm{final}}=2$ and restrict the search to the stable regime $Mh^d\ge81$; selector calibration and nearby variants are summarized in Appendix~\ref{app:exp-bandwidths}. The main diagnostic is the oracle ratio
$\Gamma_{\nu,M}^{(r)}:=E_{\infty,\mathcal G}^{(r)}(t_0,\xi_0;\hat h)/E_{\infty,\mathcal G}^{(r)}(t_0,\xi_0;h_{\rm or})\ge1$,
together with its repetition average $\bar\Gamma_\nu(M)$ and the four-point summaries
$\widehat C_\nu^{\max}:=\max_{M\in\mathcal M}\bar\Gamma_\nu(M)$ and
$\widehat C_\nu^{\avg}:=|\mathcal M|^{-1}\sum_{M\in\mathcal M}\bar\Gamma_\nu(M)$ for
$\mathcal M=\{10^3,2\times10^3,4\times10^3,8\times10^3\}$.

\FloatBarrier
\Needspace{16\baselineskip}
\begin{figure}[H]
    \centering
    \includegraphics[width=.48\linewidth]{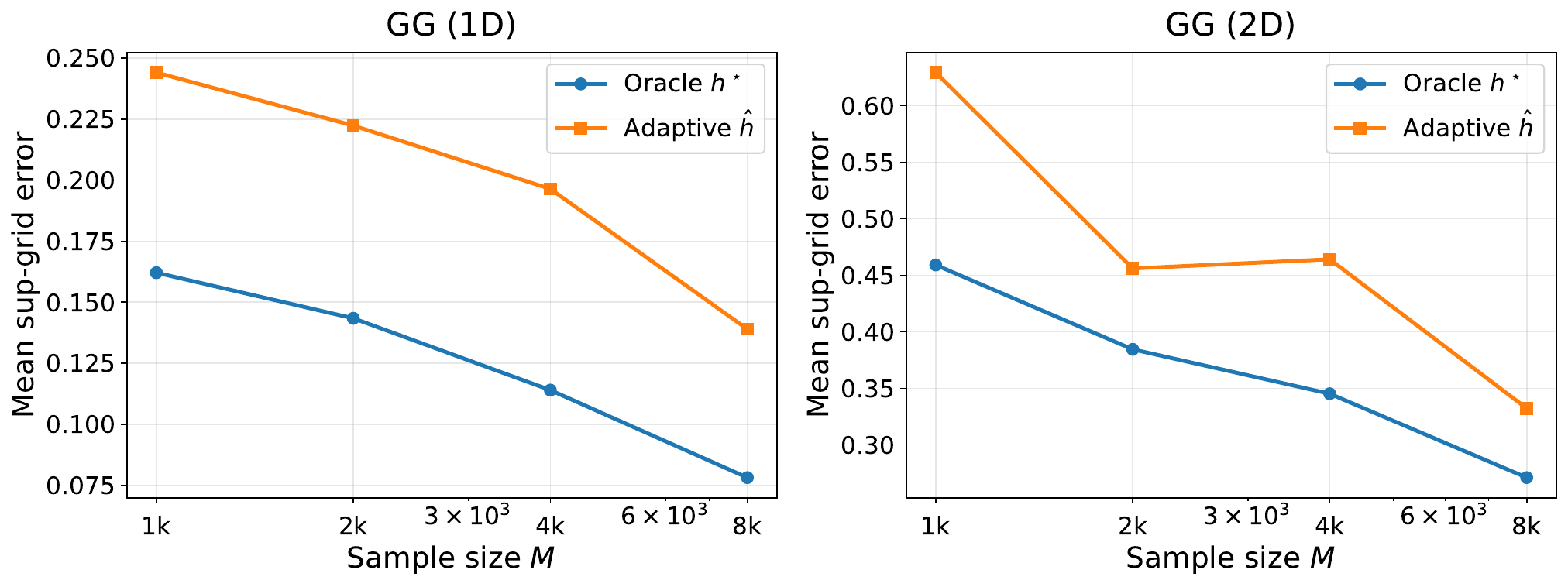}
    \hfill
    \includegraphics[width=.48\linewidth]{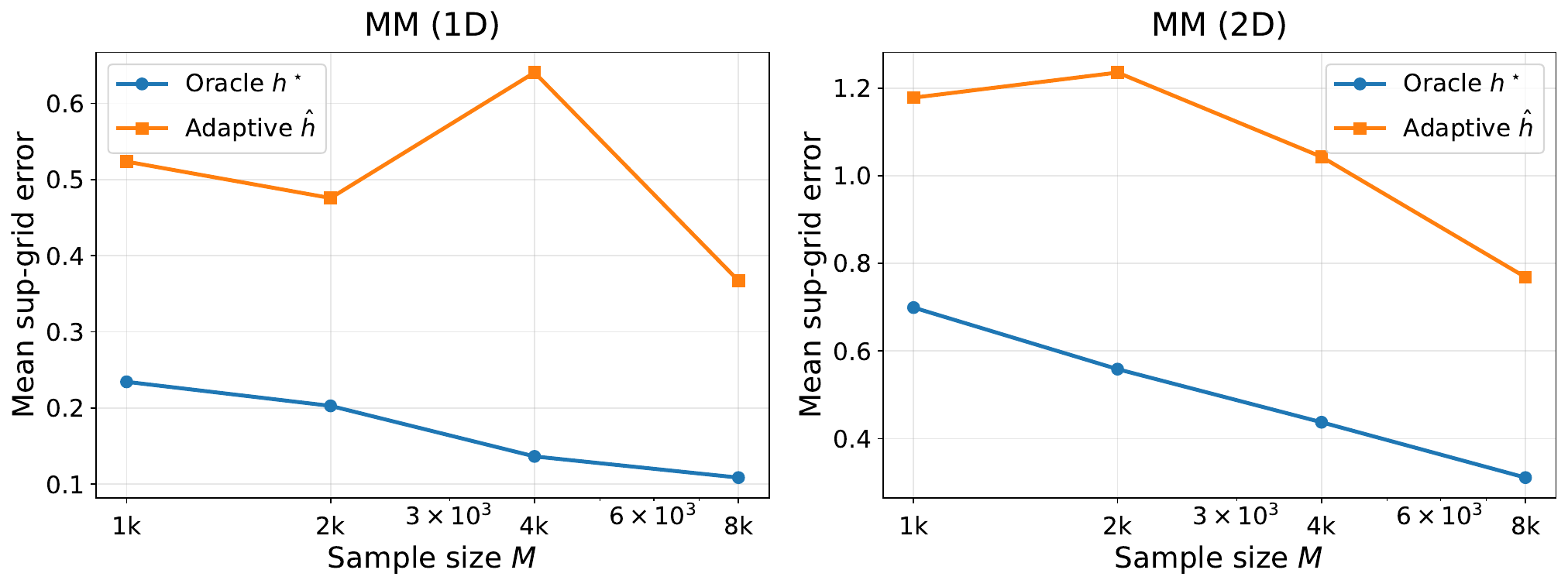}
    \caption{Log--log sup-grid error curves for the GG and MM testbeds. The oracle slopes are compared to the finite-range theoretical slope induced by $(\log M/M)^{\beta/(2\beta+d)}$; adaptive performance itself is assessed below through oracle-ratio diagnostics.}
    \label{fig:rates}
\end{figure}
\FloatBarrier

\begin{table}[H]
\centering
\small
\setlength{\tabcolsep}{4pt}
\renewcommand{\arraystretch}{1.02}
\caption{Estimated oracle slopes of $\log \overline{E}_{\infty,\mathcal G}(M)$ versus $\log M$.}
\label{tab:slopes}
\begin{tabular}{lccc}
    \toprule
    Testbed & $d$ & Theory & Oracle-$h_{\rm or}$ \\
    \midrule
    GG & 1 & $-0.349379$ & $-0.348899$ \\
    GG & 2 & $-0.291150$ & $-0.243435$ \\
    MM & 1 & $-0.349379$ & $-0.390229$ \\
    MM & 2 & $-0.291150$ & $-0.385828$ \\
    \bottomrule
\end{tabular}
\end{table}

\paragraph{Rate verification.}
\Cref{fig:rates,tab:slopes} show that the oracle estimator follows the finite-sample scaling law predicted by \Cref{cor:rate}. In $d=1$, the fitted oracle slopes are very close to the finite-range benchmark; in $d=2$, they remain negative and of the predicted order, with more visible pre-asymptotic effects on the coarser evaluation grid. Thus the experiment supports the theorem at the level it is stated: oracle-rate recovery for $h_{\rm or}$, while adaptivity is assessed separately through oracle-ratio diagnostics.

\newsavebox{\cltRightBox}
\sbox{\cltRightBox}{%
\begin{minipage}[b]{.45\linewidth}
    \centering
    \includegraphics[width=\linewidth]{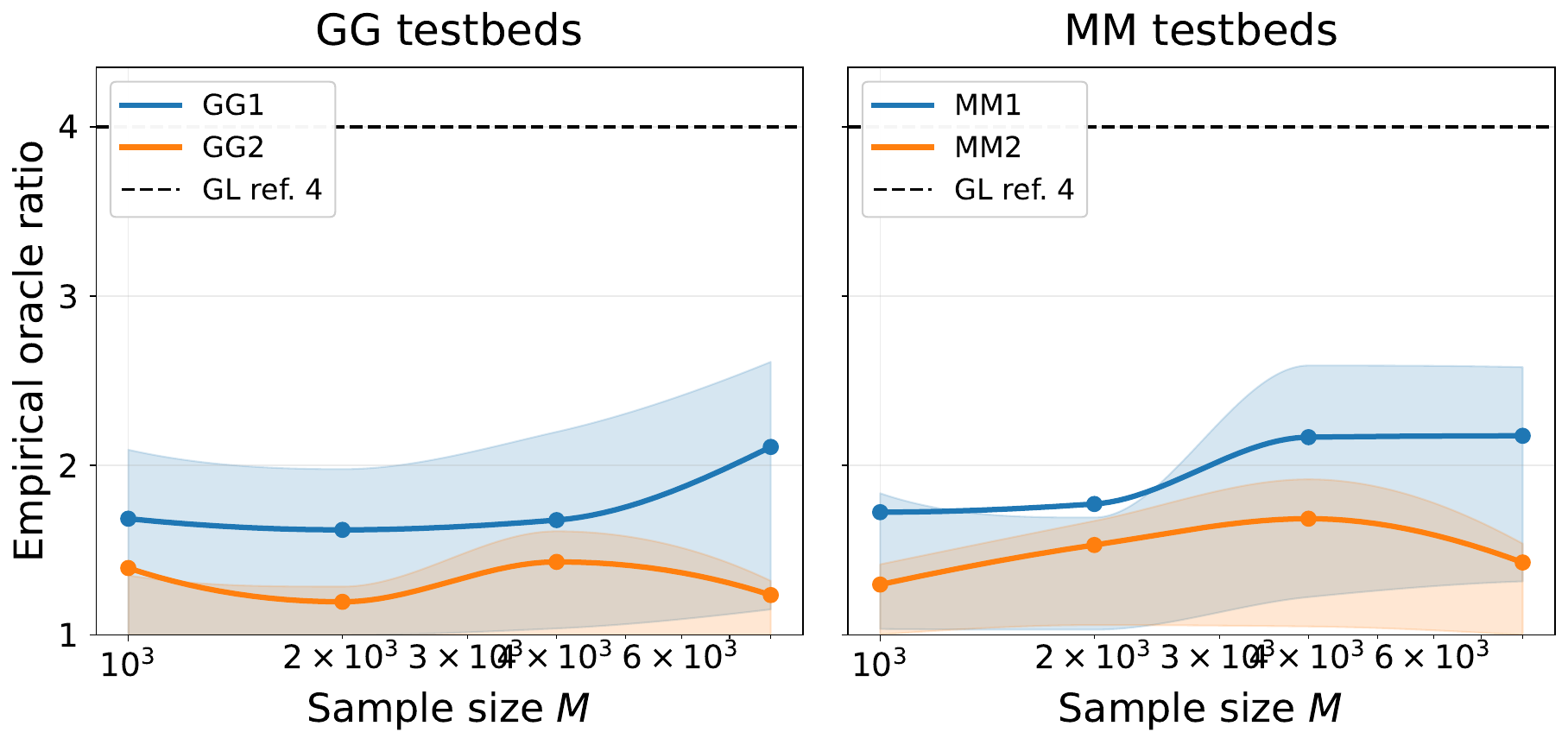}

    \vspace{0.55em}

    \includegraphics[width=.49\linewidth]{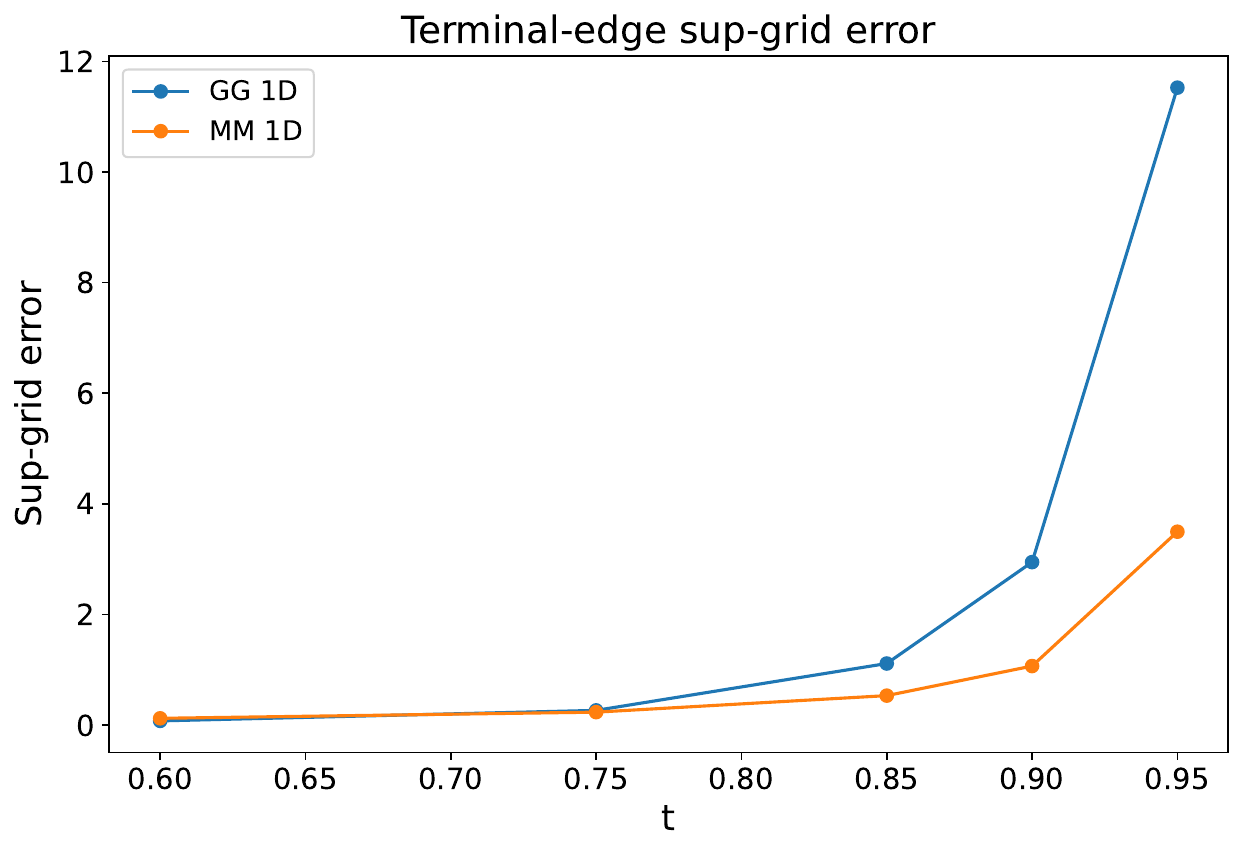}
    \hfill
    \includegraphics[width=.49\linewidth]{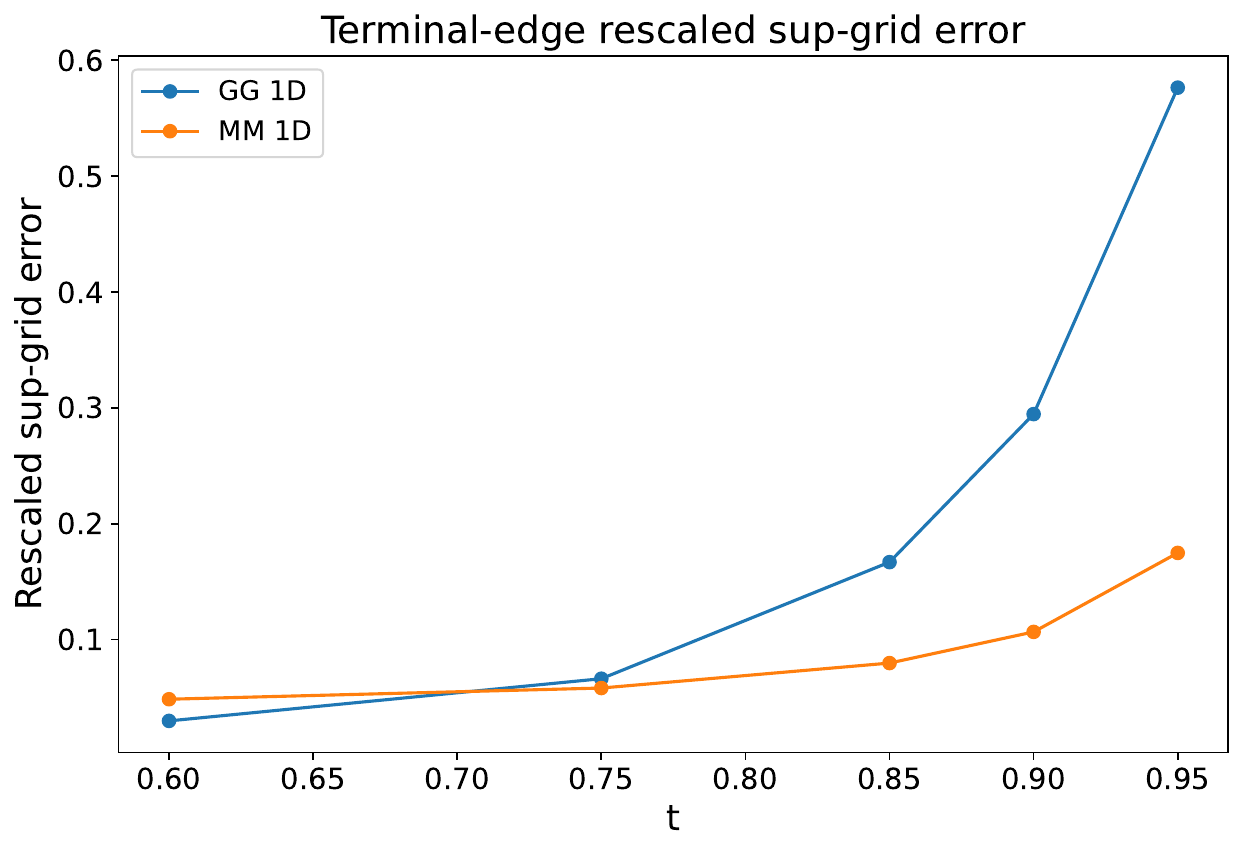}
\end{minipage}%
}

\FloatBarrier
\Needspace{22\baselineskip}
\begin{figure}[H]
    \centering
    \begin{minipage}[b]{.52\linewidth}
        \centering
        \includegraphics[height=\dimexpr\ht\cltRightBox+\dp\cltRightBox\relax]{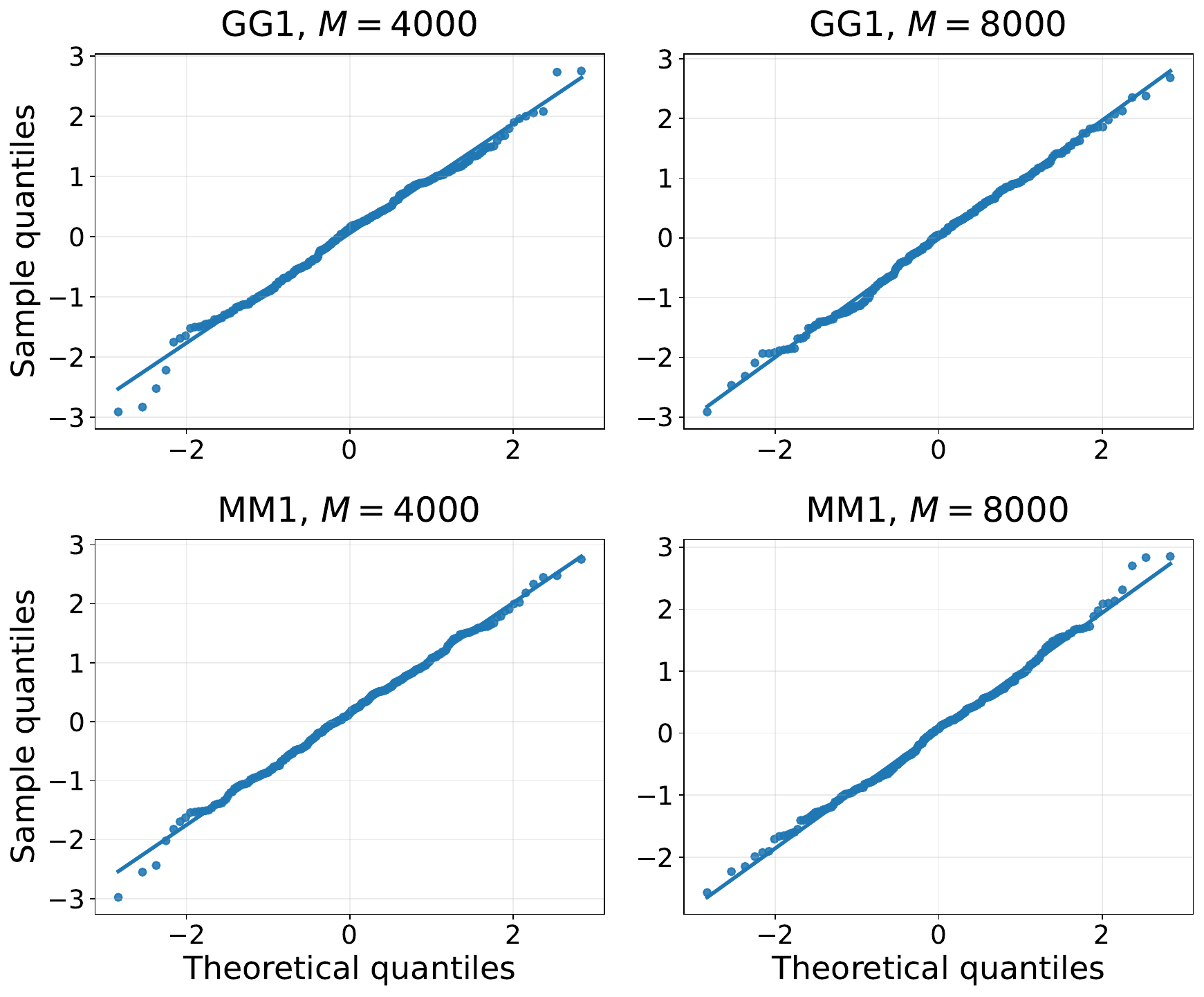}
    \end{minipage}\hfill
    \usebox{\cltRightBox}
    \caption{Pointwise Gaussian approximation, adaptivity, and terminal-edge behavior. Left: QQ-plots for one-dimensional CLT runs at $M=4\times10^3$ and $M=8\times10^3$ in GG1 and MM1. Upper right: empirical oracle ratio $\bar\Gamma_\nu(M)$ for the canonical selector, shown on a denser sample-size grid with interquartile ribbons; the dashed line marks the GL reference level $4$. Lower right: terminal-edge behavior, showing $E_{\infty,\mathcal G}(t,\xi_0;\hat h)$ as $t\uparrow u$ and the rescaled error $\Delta(t)\ E_{\infty,\mathcal G}(t,\xi_0;\hat h)$.}
    \label{fig:clt-adapt-edge}
\end{figure}
\FloatBarrier

\begin{table}[H]
\centering
\small
\setlength{\tabcolsep}{4pt}
\renewcommand{\arraystretch}{1.02}
\caption{Empirical coverage (\%) of nominal $95\%$ confidence intervals in the final $N_{\rm rep}=300$ CLT runs.}
\label{tab:coverage}
\begin{tabular}{lcccc}
    \toprule
    Testbed & $10^3$ & $2\times10^3$ & $4\times10^3$ & $8\times10^3$ \\
    \midrule
    GG & $92.67$ & $96.33$ & $96.67$ & $96.67$ \\
    MM & $92.33$ & $96.67$ & $96.33$ & $96.00$ \\
    \bottomrule
\end{tabular}
\end{table}

\paragraph{Pointwise Gaussian approximation and adaptivity.}
\Cref{fig:clt-adapt-edge,tab:coverage} examine the CLT and the practical selector. The final one-dimensional CLT runs show near-nominal coverage from $M=2\times10^3$ onward, approximately linear QQ-plots at the two largest sample sizes, and formal Shapiro--Wilk / Anderson--Darling diagnostics that do not reject normality at the largest sample size in either testbed; empirical means and variances are reported in Appendix~\ref{app:exp-details}. We therefore interpret the CLT experiment as showing a clear finite-sample Gaussian footprint under genuinely undersmoothed bandwidth sequences, while emphasizing that the evidence is pointwise and design-specific.

For adaptivity, the selector remains in a low-single-digit oracle-ratio regime across all four testbeds, with average ratios between $1.314$ and $1.960$, maximum ratios between $1.431$ and $2.175$, and essentially zero boundary-hit rates. This is the correct oracle-inequality benchmark: the issue is not oracle-tightness but whether the selector stays moderate and stable across the tested range \cite{goldenshluger2011,lacour2016minimal}. Under that criterion, the canonical raw-max one-sided selector is practically oracle-competitive once the bandwidth grid is restricted to the stable regime $Mh^d\ge81$; exact four-point summaries are given in Appendix~\ref{app:exp-bandwidths}.

\paragraph{Terminal singularity and reproducibility.}
To visualize the factor $\Delta(t)^{-1}$, we also evaluate the estimator at times approaching $u$. The raw error $E_{\infty,\mathcal G}(t,\xi_0;\hat h)$ increases as $t\uparrow u$, whereas the rescaled quantity $\Delta(t)\ E_{\infty,\mathcal G}(t,\xi_0;\hat h)$ remains substantially flatter; see \Cref{fig:clt-adapt-edge}. This matches \Cref{rem:eta} and indicates that the deterioration near $u$ is driven mainly by the intrinsic SB singularity. Appendix~\ref{app:exp-stress} reports a bounded-support stress test: in the easier GG1 testbed the wider-support perturbation has only a mild effect, but in a stronger MM1 wide-support variant it sharply increases the variability of the weighted empirical terms, produces nontrivial boundary-hit events, and substantially worsens the selected-estimator error. We release code and configuration files reproducing all figures and tables, including the exact synthetic-model definitions, ground-truth computation protocol, bandwidth grids, repetition counts, selected undersmoothing exponents, stress-test variants, and evaluation grids.

\section{Discussion and Limitations}
\label{sec:discussion}

The paper gives a sharp theory for \emph{direct drift estimation in the SBTS single-interval} setting. Bounded support is central to the current proof because it controls the expanding exponential factor in the SB weight and makes the weighted empirical-process analysis tractable. As a nonparametric method, the estimator also inherits the usual curse of dimensionality, and adaptive selection adds the computational cost of evaluating a bandwidth grid.

The theory is uniform on the query set
$Q_{R,\eta,\rho}=[s,u-\eta]\times B_R\times B_\rho$
and necessarily deteriorates as $t\uparrow u$. The main results are proved for i.i.d.\ trajectory samples in the single-interval model, while the finite-sample ratio analysis works on a natural high-probability event where empirical denominators are separated from zero. Fixed finite-memory extensions are standard, replacing the conditioning variable by a lag block and the ambient conditioning dimension $d$ by $Ld$. Estimating the SB drift from one discretely observed trajectory, where the effective samples are dependent and the past may grow with the grid, requires new tools and is left for future work.

The lower-bound result is local in its construction but global in its implication under the compatibility
conditions of \Cref{cor:global-minimax-explicit-pivot}. \Cref{prop:lower} proves a sharp two-point Le Cam lower bound inside a
shrinking neighborhood of any interior pivot law. \Cref{cor:global-minimax-explicit-pivot} then exhibits an explicit uniform pivot
satisfying these interiority conditions whenever the model constants are compatible with the bounded
support class. Since the resulting local neighborhood is contained in the original global model, the
same construction yields a global minimax lower bound over
\(\mathfrak P(\beta,L_p,f_{\min},B)\). Thus the adaptive selector is globally minimax-rate
optimal up to logarithmic factors, while the proof also gives the stronger local interpretation around
interior laws.

\section{Conclusion}
\label{sec:conclusion}

We studied direct nonparametric estimation of Schr\"odinger bridge time-series drifts. Unlike approaches based on entropic-OT potentials, Sinkhorn iterations, or iterative bridge solvers, our analysis works directly at the drift level and isolates the \emph{statistical error} of the estimator. Under H\"older regularity, a marginal-density floor, and bounded support, we proved an
interior uniform finite-sample bound at the classical kernel rate up to logarithmic factors, a pointwise CLT
under genuine undersmoothing, and an adaptive selector satisfying an oracle inequality. We also proved a pivot-local minimax lower bound and, by exhibiting an explicit uniform pivot in the
global class, derived a global minimax lower bound under transparent compatibility conditions. This
shows that the adaptive selector is minimax-rate optimal up to logarithmic factors. Together, these results provide a self-contained statistical theory for direct kernel estimation of SBTS drifts.

\bibliographystyle{abbrvnat}
\bibliography{example_paper}

@article{marie_rosier_2023_nw_iid_paths,
  author  = {Marie, Nicolas and Rosier, Am{\'e}lie},
  title   = {Nadaraya--Watson Estimator for I.I.D. Paths of Diffusion Processes},
  journal = {Scandinavian Journal of Statistics},
  volume  = {50},
  number  = {2},
  pages   = {589--637},
  year    = {2023},
  doi     = {10.1111/sjos.12593}
}

@article{panloup_tindel_varvenne_2020_drift,
  author  = {Panloup, Fabien and Tindel, Samy and Varvenne, Maylis},
  title   = {A General Drift Estimation Procedure for Stochastic Differential Equations with Additive Fractional Noise},
  journal = {Electronic Journal of Statistics},
  volume  = {14},
  number  = {1},
  pages   = {1075--1136},
  year    = {2020},
  doi     = {10.1214/20-EJS1685}
}

@inproceedings{liu_etal_2024_gsbm,
  author    = {Liu, Guan-Horng and Lipman, Yaron and Nickel, Maximilian and Karrer, Brian and Theodorou, Evangelos A. and Chen, Ricky T. Q.},
  title     = {Generalized Schr{\"o}dinger Bridge Matching},
  booktitle = {International Conference on Learning Representations},
  year      = {2024}
}

@article{schrodinger1931umkehrung,
  author  = {Schr{\"o}dinger, Erwin},
  title   = {{\"U}ber die Umkehrung der Naturgesetze},
  journal = {Sitzungsberichte der Preussischen Akademie der Wissenschaften,
             Physikalisch-mathematische Klasse},
  pages   = {144--153},
  year    = {1931}
}

@article{jamison1975markov,
  author  = {Jamison, Benton},
  title   = {The Markov Processes of Schr{\"o}dinger},
  journal = {Zeitschrift f{\"u}r Wahrscheinlichkeitstheorie und Verwandte Gebiete},
  volume  = {32},
  pages   = {323--331},
  year    = {1975},
  doi     = {10.1007/BF00535844}
}

@incollection{follmer1988randomfields,
  author    = {F{\"o}llmer, Hans},
  title     = {Random Fields and Diffusion Processes},
  booktitle = {{\'E}cole d'{{\'E}}t{\'e} de Probabilit{\'e}s de Saint-Flour XV--XVII, 1985--87},
  series    = {Lecture Notes in Mathematics},
  volume    = {1362},
  pages     = {101--203},
  publisher = {Springer},
  year      = {1988},
  doi       = {10.1007/BFb0086180}
}

@inproceedings{cuturi2013sinkhorn,
  author    = {Cuturi, Marco},
  title     = {Sinkhorn Distances: Lightspeed Computation of Optimal Transport},
  booktitle = {Advances in Neural Information Processing Systems},
  volume    = {26},
  year      = {2013},
  url       = {https://papers.nips.cc/paper/4927-sinkhorn-distances-lightspeed-computation-of-optimal-transport}
}

@inproceedings{genevay2019sinkhorn,
  author    = {Genevay, Aude and Chizat, L{\'e}na{\"i}c and Bach, Francis and Cuturi, Marco and Peyr{\'e}, Gabriel},
  title     = {Sample Complexity of Sinkhorn Divergences},
  booktitle = {Proceedings of the Twenty-Second International Conference on Artificial Intelligence and Statistics},
  series    = {Proceedings of Machine Learning Research},
  volume    = {89},
  pages     = {1574--1583},
  publisher = {PMLR},
  year      = {2019},
  url       = {https://proceedings.mlr.press/v89/genevay19a.html}
}

@inproceedings{mena2019entropic,
  author    = {Mena, Gonzalo and Weed, Jonathan},
  title     = {Statistical Bounds for Entropic Optimal Transport: Sample Complexity and the Central Limit Theorem},
  booktitle = {Advances in Neural Information Processing Systems},
  volume    = {32},
  year      = {2019},
  url       = {https://proceedings.neurips.cc/paper/2019/hash/5acdc9ca5d99ae66afdfe1eea0e3b26b-Abstract.html}
}

@book{vdvwellner1996empirical,
  author    = {van der Vaart, Aad W. and Wellner, Jon A.},
  title     = {Weak Convergence and Empirical Processes: With Applications to Statistics},
  series    = {Springer Series in Statistics},
  publisher = {Springer},
  address   = {New York},
  year      = {1996},
  doi       = {10.1007/978-1-4757-2545-2}
}

@article{bousquet2002bennett,
  author  = {Bousquet, Olivier},
  title   = {A Bennett Concentration Inequality and Its Application to Suprema of Empirical Processes},
  journal = {Comptes Rendus Mathematique},
  volume  = {334},
  number  = {6},
  pages   = {495--500},
  year    = {2002},
  doi     = {10.1016/S1631-073X(02)02292-6}
}

@book{boucheron2013concentration,
  author    = {Boucheron, St{\'e}phane and Lugosi, G{\'a}bor and Massart, Pascal},
  title     = {Concentration Inequalities: A Nonasymptotic Theory of Independence},
  publisher = {Oxford University Press},
  year      = {2013}
}

@book{lecam1986asymptotic,
  author    = {Le Cam, Lucien},
  title     = {Asymptotic Methods in Statistical Decision Theory},
  series    = {Springer Series in Statistics},
  publisher = {Springer},
  address   = {New York},
  year      = {1986}
}

@incollection{yu1997assouad,
  author    = {Yu, Bin},
  title     = {Assouad, Fano, and Le Cam},
  booktitle = {Festschrift for Lucien Le Cam},
  editor    = {Pollard, David and Torgersen, Erik and Yang, Grace L.},
  pages     = {423--435},
  publisher = {Springer},
  year      = {1997},
  doi       = {10.1007/978-1-4612-1880-7_29}
}

@article{gine2002rates,
  author  = {Gin{\'e}, Evarist and Guillou, Armelle},
  title   = {Rates of strong uniform consistency for multivariate kernel density estimators},
  journal = {Annales de l'Institut Henri Poincar{\'e} (B) Probability and Statistics},
  volume  = {38},
  number  = {6},
  pages   = {907--921},
  year    = {2002}
}

@article{einmahl2005uniform,
  author  = {Einmahl, Uwe and Mason, David M.},
  title   = {Uniform in bandwidth consistency of kernel-type function estimators},
  journal = {The Annals of Statistics},
  volume  = {33},
  number  = {3},
  pages   = {1380--1403},
  year    = {2005}
}

@article{lacour2016minimal,
  author  = {Lacour, Claire and Massart, Pascal},
  title   = {Minimal Penalty for Goldenshluger--Lepski Method},
  journal = {Stochastic Processes and their Applications},
  volume  = {126},
  number  = {12},
  pages   = {3774--3789},
  year    = {2016},
  doi     = {10.1016/j.spa.2016.04.015}
}

@article{leonard2014survey,
  author  = {Christian L{\'e}onard},
  title   = {A Survey of the Schr{\"o}dinger Problem and Some of Its Connections with Optimal Transport},
  journal = {Discrete and Continuous Dynamical Systems - A},
  year    = {2014},
  volume  = {34},
  number  = {4},
  pages   = {1533--1574},
  doi     = {10.3934/dcds.2014.34.1533}
}

@article{chen2021siamreview,
  author  = {Yongxin Chen and Tryphon T. Georgiou and Michele Pavon},
  title   = {Stochastic Control Liaisons: Richard Sinkhorn Meets Gaspard Monge on a Schr{\"o}dinger Bridge},
  journal = {SIAM Review},
  year    = {2021},
  volume  = {63},
  number  = {2},
  pages   = {249--313},
  doi = {10.1137/20M1339982}
}

@inproceedings{debortoli2021sb,
  author    = {De Bortoli, Valentin and Thornton, James and Heng, Jeremy and Doucet, Arnaud},
  title     = {Diffusion Schr{\"o}dinger Bridge with Applications to Score-Based Generative Modeling},
  booktitle = {Advances in Neural Information Processing Systems},
  year      = {2021}
}

@inproceedings{strome2023sampling,
  author    = {Stromme, Austin J.},
  title     = {Sampling From a Schr{\"o}dinger Bridge},
  booktitle = {Proceedings of the 26th International Conference on Artificial Intelligence and Statistics},
  series    = {Proceedings of Machine Learning Research},
  volume    = {206},
  pages     = {4058--4067},
  publisher = {PMLR},
  year      = {2023},
  url       = {https://proceedings.mlr.press/v206/stromme23a.html}
}

@article{nadaraya1964estimating,
  author  = {Nadaraya, Elizbar A.},
  title   = {On Estimating Regression},
  journal = {Theory of Probability and Its Applications},
  volume  = {9},
  number  = {1},
  pages   = {141--142},
  year    = {1964},
  doi     = {10.1137/1109020}
}

@article{watson1964smooth,
  author  = {Watson, Geoffrey S.},
  title   = {Smooth Regression Analysis},
  journal = {Sankhy{\=a}: The Indian Journal of Statistics, Series A},
  volume  = {26},
  number  = {4},
  pages   = {359--372},
  year    = {1964}
}

@article{hamdouche2023sbts,
  author  = {Hamdouche, Mohamed and Henry{-}Labord{\`e}re, Pierre and Pham, Huy{\^e}n},
  title   = {Generative Modeling for Time Series via {Schr{\"o}dinger Bridge}},
  journal = {Journal of Machine Learning Research},
  year    = {2026},
  note    = {To appear},
  url     = {https://arxiv.org/abs/2304.05093}
}

@inproceedings{alouadi2025robust,
  author    = {Alouadi, Alexandre and Barreau, Baptiste and Carlier, Laurent and Pham, Huy{\^e}n},
  title     = {Robust time series generation via {Schr{\"o}dinger Bridge}: a comprehensive evaluation},
  booktitle = {Proceedings of the 6th ACM International Conference on AI in Finance},
  series    = {ICAIF '25},
  pages     = {906--914},
  year      = {2025},
  url       = {https://doi.org/10.1145/3768292.3770391}
}

@article{lepski1997,
  author  = {Lepski, Oleg V. and Spokoiny, Vladimir G.},
  title   = {Optimal Pointwise Adaptive Methods in Nonparametric Estimation},
  journal = {The Annals of Statistics},
  volume  = {25},
  number  = {6},
  pages   = {2512--2546},
  year    = {1997},
  doi     = {10.1214/aos/1030741083}
}

@article{pooladian2024plugin,
  author  = {Aram{-}Alexandre Pooladian and Jonathan Niles{-}Weed},
  title   = {Plug-in Estimation of Schr{\"o}dinger Bridges},
  journal = {arXiv preprint arXiv:2408.11686},
  year    = {2024}
}

@article{belomestny2025forwardreverse,
  author  = {Belomestny, Denis and Schoenmakers, John},
  title   = {Forward Reverse Kernel Regression for the Schr{\"o}dinger Bridge Problem},
  journal = {arXiv preprint arXiv:2507.00640},
  year    = {2025}
}

@book{tsybakov2009nonparametric,
  author    = {Alexandre B. Tsybakov},
  title     = {Introduction to Nonparametric Estimation},
  publisher = {Springer},
  year      = {2009},
  doi       = {10.1007/b13794}
}

@book{gyorfi2002,
  author    = {L{\'a}szl{\'o} Gy{\"o}rfi and Michael Kohler and Adam Krzy{\.z}ak and Harro Walk},
  title     = {A Distribution-Free Theory of Nonparametric Regression},
  publisher = {Springer},
  year      = {2002},
  doi       = {10.1007/978-1-4613-0003-1}
}

@article{goldenshluger2011,
  author  = {Goldenshluger, Alexander and Lepski, Oleg},
  title   = {Bandwidth Selection in Kernel Density Estimation: Oracle Inequalities and Adaptive Minimax Optimality},
  journal = {The Annals of Statistics},
  volume  = {39},
  number  = {3},
  pages   = {1608--1632},
  year    = {2011},
  doi     = {10.1214/11-AOS883}
}

\newpage
\appendix
\onecolumn
\appendix

\section{Proof roadmap and auxiliary notation}
\label{Appendix A}

Throughout the Appendices, we keep the notation of Sections~3--4. In particular,
$g_1,g_2,D^*,N^*,Q^*,a^*,\widehat f_j,\widehat g_j,\widehat D,\widehat N,\widehat a$
retain their meanings from the main text. We also write $C_f:=|f|_\infty$, $C_B:=\lambda(B)$, and
$|K|_2:=(\int_{\mathbb R^d}K(z)^2\,dz)^{1/2}$ and
$\mu_\infty(K):=|K|_\infty$.
For fixed $(t,x,\xi)$, we write
$\Delta_t:=\Delta(t)$ and $F_{t,x,\xi}(y):=F(t,\xi,x,y)$ when convenient.

For an open neighborhood $U$ of $B$, write $\Sigma_U(\beta,L)$ for the
Hölder ball of functions $r:U\to\mathbb{R}$ satisfying
\begin{equation}
\sum_{|\alpha|\le \ell_\beta}|D^\alpha r|_\infty
+
\sum_{|\alpha|=\ell_\beta}
\sup_{z\neq z'\in U}
\frac{|D^\alpha r(z)-D^\alpha r(z')|}
{|z-z'|^{\beta-\ell_\beta}}
\le L .
\end{equation}
When the neighborhood $U$ is clear, we write $\Sigma(\beta,L)$.

\subsection{Uniform KDE bound}

\begin{lemma}[Uniform KDE bound for $\widehat f_j$]
\label{lem:kde}
For each $j\in\{1,2\}$, there exist constants $C_1,C_2,C_K>0$, depending only on
$K,d,\operatorname{diam}(B)$, such that for any $\delta\in(0,1)$, with probability at least
$1-\delta$,
\begin{align}
\sup_{\xi\in B_\rho}\big|\widehat f_j(\xi)-f(\xi)\big|
&\le
\frac{C_B L_p}{\ell_\beta!}\mu_\beta(K)h_j^\beta +
C_1 \sqrt{C_f}\ |K|_2
\sqrt{\frac{d\log\!\big(\tfrac{C_K}{h_j}\big)+\log\!\big(\tfrac{2}{\delta}\big)}{M h_j^d}}\\
&\qquad\qquad+
C_2 \frac{\mu_\infty(K)}{h_j^d}
\frac{d\log\!\big(\tfrac{C_K}{h_j}\big)+\log\!\big(\tfrac{2}{\delta}\big)}{M}.
\end{align}
\end{lemma}

\begin{proof}
This is a standard uniform kernel-density estimate for the translated-kernel class given as
$\{K_{h_j}(\cdot-\xi):\xi\in B_\rho\}$ under Hölder regularity of $f$. For $h_j\le \rho/R_K$, the kernel window stays inside $B$ for
$\xi\in B_\rho$. The bias is the standard Tsybakov-order expansion: all monomials
$1\le |\alpha|\le \ell_\beta$ vanish, and since
$f(\xi)=\int_B p(\xi,y)\,dy
\ \text{and}\ \sup_{y\in B}|\bar p(\cdot,y)|_{C^\beta}\le L_p$,
the marginal satisfies $|f|_{C^\beta}\le C_B L_p$. Hence the Hölder
remainder is bounded by
$(C_B L_p/\ell_\beta!)h_j^\beta\mu_\beta(K)$.
The stochastic term is the usual uniform empirical-process bound for
translated kernels. Results of this type may be found, for example, in
Gin\'e and Guillou~\cite{gine2002rates} and Einmahl and Mason~\cite{einmahl2005uniform}.
Applied to the present fixed-bandwidth setting, they yield a bound of the displayed form, so we
omit the routine details.
\end{proof}

\subsection{Uniform weighted bound for $\widehat g_1$}

\begin{lemma}[Uniform weighted bound for $\widehat g_1$]
\label{lem:g1}
For any $\delta\in(0,1)$, with probability at least $1-\delta$,
\begin{align}
\sup_{t\in[s,u-\eta],\,x\in\mathbb R^d,\,\xi\in B_\rho}
\big|\widehat g_1(t,x;\xi)-g_1(t,x;\xi)\big|
&\le
\frac{C_BC_F L_p}{\ell_\beta!}\mu_\beta(K)h_1^\beta \\
&\quad+
C_1\ C_F\ \sqrt{C_f}\ |K|_2
\sqrt{\frac{d\log\!\big(\tfrac{C_K}{h_1}\big)+\log\!\big(\tfrac{2}{\delta}\big)}{M h_1^{d}}} \\
&\quad+
C_2\ \frac{C_F\ \mu_\infty(K)}{h_1^{d}}
\frac{d\log\!\big(\tfrac{C_K}{h_1}\big)+\log\!\big(\tfrac{2}{\delta}\big)}{M}.
\end{align}
\end{lemma}

\begin{proof}
Let $v_d$ denote the volume of the Euclidean unit ball in $\mathbb R^d$, and choose
$R_K>0$ such that $\operatorname{supp}(K)\subseteq B(0,R_K)$. Fix $(t,x,\xi)$ and write, with the change of variables $u=(z-\xi)/h_1$,
\begin{align}
\E \hat g_1(t,x;\xi)
&=
\iint F(t,\xi,x,y)\,K_{h_1}(z-\xi)\,p(z,y)\,dz\,dy\\
&=
\int_{\mathbb{R}^d} K(u)
\Bigl(\int_B F(t,\xi,x,y)\,\bar p(\xi+h_1u,y)\,dy\Bigr)\,du,
\end{align}
where $\bar p(\cdot,y)$ is the extension from Assumption~\ref{ass:pair}(i). Define
$H_{t,x,\xi}(z):=\int_B F(t,\xi,x,y)\,\bar p(z,y)\,dy$.
By bounded support and the definition of $C_F$ in \Cref{rem:bounded-support}, we have
$F(t,\xi,x,y)\le C_F$ for all $(t,x,\xi,y)\in[s,u-\eta]\times\R^d\times B\times B$.
By Assumption~\ref{ass:pair} (i), the map $z\mapsto \bar p(z,y)$ is H\"older--$\beta$
uniformly in $y\in B$. Hence
$H_{t,x,\xi}\in\Sigma(\beta,C_BC_FL_p)$ in its $z$-argument, uniformly over
$(t,x,\xi)\in [s,u-\eta]\times\mathbb R^d\times B_\rho$.
More explicitly, for every multi-index $\alpha$ with $|\alpha|\le \ell_\beta$,
\begin{equation}
|D^\alpha H_{t,x,\xi}(z)|
\le
\int_B F(t,\xi,x,y)\,|D^\alpha \bar p(z,y)|\,dy
\le
C_B C_F L_p,
\end{equation}
and for $|\alpha|=\ell_\beta$,
$|D^\alpha H_{t,x,\xi}(z)-D^\alpha H_{t,x,\xi}(z')|
\le
C_B C_F L_p\,|z-z'|^{\beta-\ell_\beta}$.

With $\ell_\beta=\max\{m\in\mathbb N_0:m<\beta\}$, Taylor expansion of
$H_{t,x,\xi}$ at $\xi$ to order $\ell_\beta$, together with the moment conditions
$\int_{\mathbb R^d}u^\alpha K(u)\,du=0,\  1\le |\alpha|\le \ell_\beta$,
yields
\begin{align}
\big|\mathbb E\ \hat g_1(t,x;\xi)-H_{t,x,\xi}(\xi)\big|
&=\Big|\int K(u)\big(H_{t,x,\xi}(\xi+h_1u)-H_{t,x,\xi}(\xi)\big)\ du\Big|\\
&\le \frac{C_B C_F L_p}{\ell_\beta!}\ h_1^\beta \int |u|^\beta |K(u)|\ du.
\end{align}
As $\xi\in B$ and $\bar p(\xi,y)=p(\xi,y)$ on $B$:
$H_{t,x,\xi}(\xi)=\int_B F(t,\xi,x,y)\ p(\xi,y)\ dy=g_1(t,x;\xi)$.
We conclude
\begin{equation}
\sup_{t\in[s,u-\eta],\,x\in\mathbb R^d,\,\xi\in B_\rho}
|\mathbb E\hat g_1(t,x;\xi)-g_1(t,x;\xi)|
\le
\frac{C_BC_FL_p}{\ell_\beta!}\mu_\beta(K)h_1^\beta .
\end{equation}

Now define the weighted class
\begin{equation}
\mathcal H_{h_1}
:=
\left\{
h_{t,x,\xi}(z,y):=
F(t,\xi,x,y)K_{h_1}(z-\xi):
(t,x,\xi)\in [s,u-\eta]\times\mathbb R^d\times B_\rho
\right\}
\end{equation}
and the centered process
$\mathbb G_M h:=\frac1M\sum_{m=1}^M\big(h(X_s^{(m)},X_u^{(m)})-\mathbb E h(X_s,X_u)\big)$.
Then
\begin{equation}
\sup_{t\in[s,u-\eta],\,x\in\mathbb R^d,\,\xi\in B_\rho}
|\hat g_1(t,x;\xi)-\mathbb E\hat g_1(t,x;\xi)|
=
\sup_{h\in\mathcal H_{h_1}}|\mathbb G_Mh|.
\end{equation}

The class $\mathcal H_{h_1}$ has envelope
$U_{h_1}:=\sup_{h\in\mathcal H_{h_1}}|h|_\infty\le C_F\ \mu_\infty(K)/h_1^d$.
Moreover, if $P$ denotes the law of $(X_s,X_u)$, then
$\sigma_{h_1}:=\sup_{h\in\mathcal H_{h_1}}|h|_{L_2(P)}
\le C_F\sqrt{C_f}\ |K|_2/h_1^{d/2}$, because
\begin{align}
|h_{t,x,\xi}|_{L_2(P)}^2
&= \iint F(t,\xi,x,y)^2\ K_{h_1}(z-\xi)^2\ p(z,y)\ dz\ dy\\
&\le C_F^2\int K_{h_1}(z-\xi)^2\ f(z)\ dz
\le C_F^2 C_f\ |K_{h_1}|_{L_2(\lambda)}^2.
\end{align}

Write $F(t,\xi,x,y)=\chi(t,x,y)\ G(\xi,y)$, where
$\chi(t,x,y):=e^{-\frac{|y-x|^2}{2\Delta(t)}}$ and
$G(\xi,y):=e^{\frac{|y-\xi|^2}{2\Delta}}$.
On $[s,u-\eta]\times\R^d\times B$, the function $\chi$ is globally Lipschitz in $(t,x)$: there
exists $L_\chi>0$, depending only on $\eta,\Delta$, such that
\begin{equation}
|\chi(t,x,y)-\chi(t',x',y)|
\le L_\chi\big(|t-t'|+|x-x'|\big)
\qquad
\text{for all }t,t'\in[s,u-\eta],\ x,x'\in\R^d,\ y\in B.
\end{equation}
Indeed,
$\sup_{r\ge0,\ \tau\in[\eta,\Delta]}\frac{r}{\tau}e^{-r^2/(2\tau)}<\infty,
\ \sup_{r\ge0,\ \tau\in[\eta,\Delta]}\frac{r^2}{2\tau^2}e^{-r^2/(2\tau)}<\infty$.
Also, $G$ is Lipschitz in $\xi$ on $B\times B$: there exists $L_G>0$, depending only on
$\Delta$ and $\operatorname{diam}(B)$, such that
\begin{equation}
|G(\xi,y)-G(\xi',y)|\le L_G |\xi-\xi'|
\qquad
\text{for all }\xi,\xi',y\in B.
\end{equation}

Now fix $(t,x,\xi)$ and $(t',x',\xi')$. Then
\begin{align}
|h_{t,x,\xi}(z,y)-h_{t',x',\xi'}(z,y)|
&\le |\chi(t,x,y)-\chi(t',x',y)|\ G(\xi,y)\ |K_{h_1}(z-\xi)| \\
&\quad + \chi(t',x',y)\ |G(\xi,y)-G(\xi',y)|\ |K_{h_1}(z-\xi)| \\
&\quad + \chi(t',x',y)\ G(\xi',y)\ |K_{h_1}(z-\xi)-K_{h_1}(z-\xi')|.
\end{align}
Using $\chi\le1$, $G\le C_F$, the two Lipschitz bounds above, and
$|K_{h_1}(\cdot-\xi)|_{L_2(P)}\le \sqrt{C_f}\ |K|_2/h_1^{d/2}$, together with
$|K_{h_1}(\cdot-\xi)-K_{h_1}(\cdot-\xi')|_{L_2(P)}
\le \sqrt{C_f}\ \frac{L_K\sqrt{2^d v_d}\ R_K^{d/2}}{h_1^{d/2+1}}\ |\xi-\xi'|$,
we obtain
\begin{equation}
|h_{t,x,\xi}-h_{t',x',\xi'}|_{L_2(P)}
\le
\sigma_{h_1}\Big(A_1 |t-t'|+A_2 |x-x'|+A_3 \frac{|\xi-\xi'|}{h_1}\Big),
\end{equation}
for constants $A_1,A_2,A_3>0$ depending only on $K$, $d$, $\eta$, $\Delta$, and
$\operatorname{diam}(B)$.

For the unbounded $x$-parameter by exponential decay, let
$\operatorname{dist}(x,B):=\inf_{y\in B}|x-y|$.
Since $\Delta(t)\le \Delta$ and $y\in B$, we have
$\chi(t,x,y)\le e^{-\operatorname{dist}(x,B)^2/(2\Delta)}$, thus
$|h_{t,x,\xi}|_{L_2(P)}
\le \sigma_{h_1}\ e^{-\operatorname{dist}(x,B)^2/(2\Delta)}$
uniformly in $(t,\xi)$. Given $\varepsilon\in(0,1]$, define
$r_\varepsilon:=\sqrt{2\Delta\log(4/\varepsilon)}$. If $\operatorname{dist}(x,B)\ge r_\varepsilon$,
then $|h_{t,x,\xi}|_{L_2(P)}\le \frac{\varepsilon}{4}\sigma_{h_1}$. Therefore, up to an additional
ball of radius $\varepsilon\sigma_{h_1}/4$ around the zero function, it suffices to cover the
parameter set with $x\in B^{+}_{r_\varepsilon}:=\{x\in\R^d:\operatorname{dist}(x,B)\le r_\varepsilon\}$.

Choose a grid in $t\in[s,u-\eta]$ with mesh $\varepsilon/(4A_1)$, a grid in
$x\in B^{+}_{r_\varepsilon}$ with mesh $\varepsilon/(4A_2)$, and a grid in $\xi\in B_\rho$ with mesh
$\varepsilon h_1/(4A_3)$. 

By the Lipschitz estimate, this induces an $L_2(P)$
$\varepsilon\sigma_{h_1}$-net of $\mathcal H_{h_1}$. Equivalently, the covering number is bounded by the product of coverings in the $t$-, $x$-, and $\xi$-coordinates:
$N(\varepsilon \sigma_{h_1},\mathcal H_{h_1},L_2(P))
\le
N_t(\varepsilon)\,N_x(\varepsilon)\,N_\xi(\varepsilon),
$ with $
\log N_t(\varepsilon)\lesssim \log(C/\varepsilon),\ 
\log N_x(\varepsilon)\lesssim d\log(C/\varepsilon),\ 
\log N_\xi(\varepsilon)\lesssim d\log\!(C/(\varepsilon h_1))$.
 Since
$r_\varepsilon\lesssim 1+\sqrt{\log(1/\varepsilon)}$, there exists $C_K>0$, depending only on
$K$, $d$, $\eta$, $\Delta$, and $\operatorname{diam}(B)$, such that
$\log N\!\left(\varepsilon\sigma_{h_1},\mathcal H_{h_1},L_2(P)\right)
\le d\log\!\Big(\frac{C_K}{\varepsilon h_1}\Big)
+ C\log\!\Big(\frac{C_K}{\varepsilon}\Big)$,
for a universal constant $C>0$. In particular, after enlarging $C_K$ if necessary,
$\log N\!\left(\varepsilon\sigma_{h_1},\mathcal H_{h_1},L_2(P)\right)
\le C\ d\log\!\Big(\frac{C_K}{\varepsilon h_1}\Big)$.

We now apply standard symmetrization and Dudley's entropy integral~\cite{vdvwellner1996empirical},
together with Bousquet's inequality~\cite{bousquet2002bennett}, using envelope $U_{h_1}$
and variance proxy $\sigma_{h_1}^2$; see also~\cite{boucheron2013concentration}.
By symmetrization and Dudley’s entropy integral, there is a universal $c_0>0$ such that
\begin{equation}
\mathbb E\sup_{h\in\mathcal H_{h_1}}|\mathbb G_M h|
\le
\frac{2c_0}{\sqrt M}
\int_0^{\sigma_{h_1}}\sqrt{\log N(\tau,\mathcal H_{h_1},L_2(P))}\ d\tau.
\end{equation}
Using the entropy bound above and the change of variables $\tau=\varepsilon\sigma_{h_1}$,
\begin{equation}
\mathbb E\sup_{h\in\mathcal H_{h_1}}|\mathbb G_M h|
\le
C\ C_F\ \sqrt{C_f}\ |K|_2
\sqrt{\frac{d\bigl(1+\log(C_K/h_1)\bigr)}{M h_1^{d}}},
\end{equation}
for a universal constant $C>0$. 

Finally, Bousquet’s inequality for the class $\mathcal H_{h_1}$ with envelope $U_{h_1}$ and variance
proxy $\sigma_{h_1}^2$ yields that, with probability at least $1-\delta$,
\begin{equation}
\sup_{h\in\mathcal H_{h_1}}|\mathbb G_M h|
\le
C_1\ C_F\ \sqrt{C_f}\ |K|_2
\sqrt{\frac{d\log\!\big(\tfrac{C_K}{h_1}\big)+\log\!\big(\tfrac{2}{\delta}\big)}{M h_1^{d}}}
+
C_2\ \frac{C_F\ \mu_\infty(K)}{h_1^{d}}
\frac{d\log\!\big(\tfrac{C_K}{h_1}\big)+\log\!\big(\tfrac{2}{\delta}\big)}{M}.
\end{equation}

Combining this stochastic bound with the bias bound proved above proves the lemma.
\end{proof}

\subsection{Uniform weighted bound for $\widehat g_2$}

\begin{lemma}[Uniform weighted bound for $\widehat g_2$]
\label{lem:g2}
Recall $C_F$ and $C_Y$ from Remark \ref{rem:bounded-support}. 
Then, with $\ell_\beta$ as in Assumption 3.1, there exist constants $C_1,C_2,C_K>0$,
where $C_K$ depends only on $K,d,\eta,\Delta$, and $\operatorname{diam}(B)$, such that for any $\delta\in(0,1)$, with probability at
least $1-\delta$,
\begin{align}
\sup_{t\in[s,u-\eta],\ x\in\mathbb R^d,\ \xi\in B_\rho}
\big|\widehat g_2(t,x;\xi)-g_2(t,x;\xi)\big|
&\le
\frac{C_B C_Y C_F L_p}{\ell_\beta!}\ \mu_\beta(K)\ h_2^\beta \\
&\quad+
C_1\ C_Y C_F\ \sqrt{C_f}\ |K|_2
\sqrt{\frac{d\log\!\big(\tfrac{C_K}{h_2}\big)+\log\!\big(\tfrac{2}{\delta}\big)}{M h_2^{d}}} \\
&\quad+
C_2\ \frac{C_Y C_F\ \mu_\infty(K)}{h_2^{d}}
\frac{d\log\!\big(\tfrac{C_K}{h_2}\big)+\log\!\big(\tfrac{2}{\delta}\big)}{M}.
\end{align}
\end{lemma}

\begin{proof}
For 
$r\in1:d$, define
$\widehat g_{2,r}(t,x;\xi)
:=\frac1M\sum_{m=1}^M X_{u,r}^{(m)}\ F(t,\xi,x,X_u^{(m)})\ K_{h_2}(X_s^{(m)}-\xi),
$ $
g_{2,r}(t,x;\xi):=\int_B y_r\ F(t,\xi,x,y)\ p(\xi,y)\ dy$.
Then $\widehat g_2=(\widehat g_{2,1},\dots,\widehat g_{2,d})$ and $g_2=(g_{2,1},\dots,g_{2,d})$.

Fix $r\in\{1,\dots,d\}$, $t\in[s,u-\eta]$, $x\in\R^d$, and $\xi\in B_\rho$. With the change of
variables $u=(z-\xi)/h_2$,
\begin{align}
\E \hat g_{2,r}(t,x;\xi)
&=
\int_{\mathbb{R}^d} K(u)
\Bigl(\int_B y_r F(t,\xi,x,y)\,\bar p(\xi+h_2u,y)\,dy\Bigr)\,du,
\end{align}
where $\bar p(\cdot,y)$ is the extension from Assumption~\ref{ass:pair}(i). Define
$H_{t,x,\xi,r}(z):=\int_B y_r F(t,\xi,x,y)\,\bar p(z,y)\,dy$.
Since $|y_r|\le C_Y$ on $B$ and $\bar p(\cdot,y)\in\Sigma(\beta,L_p)$ uniformly in $y\in B$, the
map $z\mapsto H_{t,x,\xi,r}(z)$ belongs to
$\Sigma(\beta,C_BC_YC_FL_p)$ uniformly over
$(t,x,\xi)\in [s,u-\eta]\times\mathbb R^d\times B_\rho$.
With $\ell_\beta=\max\{m\in\mathbb N_0:m<\beta\}$, Taylor expansion to order
$\ell_\beta$, together with
$\int_{\mathbb R^d}u^\alpha K(u)\,du=0,\ 1\le|\alpha|\le \ell_\beta$,
yields
\begin{equation}
|\mathbb E\hat g_{2,r}(t,x;\xi)-H_{t,x,\xi,r}(\xi)|
\le
\frac{C_BC_YC_FL_p}{\ell_\beta!}\mu_\beta(K)h_2^\beta .
\end{equation}
Since $\xi\in B_\rho\subset B$ and $\bar p(\xi,y)=p(\xi,y)$ on $B$, we have
$H_{t,x,\xi,r}(\xi)=g_{2,r}(t,x;\xi)$.
Thus
\begin{equation}
\sup_{t\in[s,u-\eta],\,x\in\mathbb R^d,\,\xi\in B_\rho}
|\mathbb E\hat g_{2,r}(t,x;\xi)-g_{2,r}(t,x;\xi)|
\le
\frac{C_BC_YC_FL_p}{\ell_\beta!}\mu_\beta(K)h_2^\beta .
\end{equation}
Now consider the scalar weighted class
\begin{equation}
\mathcal H^{(r)}_{h_2}
:=
\left\{
h^{(r)}_{t,x,\xi}(z,y):=
y_r F(t,\xi,x,y)K_{h_2}(z-\xi):
(t,x,\xi)\in [s,u-\eta]\times\mathbb R^d\times B_\rho
\right\}.
\end{equation}
The proof of \Cref{lem:g1} applies verbatim to $\mathcal H_{h_2}^{(r)}$, the only change
being the extra scalar factor $y_r$, which is uniformly bounded by $C_Y$ on $B$. Therefore, we have
$U_{h_2}^{(r)}\le C_Y C_F\ \mu_\infty(K)/h_2^d$, 
$\sigma_{h_2}^{(r)}\le C_Y C_F\sqrt{C_f}\ |K|_2/h_2^{d/2}$, and the entropy bound has the same
form as in \Cref{lem:g1}, namely
\begin{equation}
\log N\!\left(\varepsilon\ \sigma_{h_2}^{(r)},\mathcal H_{h_2}^{(r)},L_2(P)\right)
\le
C\ d\log\!\Big(\frac{C_K}{\varepsilon h_2}\Big).
\end{equation}
Hence, by the same symmetrization, Dudley-integral, and Bousquet-inequality argument as in
\Cref{lem:g1}, for any $\delta\in(0,1)$, with probability at least $1-\delta/d$,
\begin{align}
\sup_{t\in[s,u-\eta],\ x\in\R^d,\ \xi\in B_\rho}
\big|\widehat g_{2,r}(t,x;\xi)-\mathbb E\widehat g_{2,r}(t,x;\xi)\big|
&\le
C_1'\ C_Y C_F\ \sqrt{C_f}\ |K|_2
\sqrt{\frac{d\log\!\big(\tfrac{C_K}{h_2}\big)+\log\!\big(\tfrac{2d}{\delta}\big)}{M h_2^{d}}}
\nonumber\\
&\quad+
C_2'\ \frac{C_Y C_F\ \mu_\infty(K)}{h_2^{d}}
\frac{d\log\!\big(\tfrac{C_K}{h_2}\big)+\log\!\big(\tfrac{2d}{\delta}\big)}{M},
\label{eq:g2-stoch-r}
\end{align}
where $C_1',C_2'>0$ depend only on universal numerical constants.

Combining the last two bounds, we get that, with probability at least $1-\delta/d$,
\begin{align}
\sup_{t\in[s,u-\eta],\ x\in\R^d,\ \xi\in B_\rho}
\big|\widehat g_{2,r}(t,x;\xi)-g_{2,r}(t,x;\xi)\big|
&\le
\frac{C_B C_Y C_F L_p}{\ell_\beta!}\ \mu_\beta(K)\ h_2^\beta
\nonumber\\
&\quad+
C_1'\ C_Y C_F\ \sqrt{C_f}\ |K|_2
\sqrt{\frac{d\log\!\big(\tfrac{C_K}{h_2}\big)+\log\!\big(\tfrac{2d}{\delta}\big)}{M h_2^{d}}}
\nonumber\\
&\quad+
C_2'\ \frac{C_Y C_F\ \mu_\infty(K)}{h_2^{d}}
\frac{d\log\!\big(\tfrac{C_K}{h_2}\big)+\log\!\big(\tfrac{2d}{\delta}\big)}{M}.
\label{eq:g2-coordinate-bound}
\end{align}

By a union bound over $r=1,\dots,d$, with probability at least $1-\delta$, the bound
\eqref{eq:g2-coordinate-bound} holds simultaneously for all coordinates. On that event,
\begin{equation}
\sup_{t,x,\xi}\big|\widehat g_2(t,x;\xi)-g_2(t,x;\xi)\big|
\le
\sqrt d\ 
\max_{1\le r\le d}
\sup_{t,x,\xi}\big|\widehat g_{2,r}(t,x;\xi)-g_{2,r}(t,x;\xi)\big|.
\end{equation}
Since the paper allows constants to depend on $d$, the factors $\sqrt d$ and $\log d$ can be
absorbed into $C_1,C_2$ and $C_K$ by enlarging them if necessary. Thus, with probability at least
$1-\delta$, the displayed bound in the statement of the lemma holds, which proves the result.
\end{proof}

\section{Deterministic Ratio Stability}
\label{app:ratio}

\begin{lemma}[Deterministic denominator stability]
\label{lem:denom-stability}
Under \Cref{ass:pair}(ii), fix $(t,x,\xi)\in Q_{R,\eta,\rho}$. On the event
$\left\{\hat f_1(\xi)\ge \frac{1}{2}f_{\min}\right\}$,
we have
\begin{equation}
\bigl|\hat D(t,x;\xi)-D^*(t,x;\xi)\bigr|
\le \frac{2}{f_{\min}^2}
\Big(
f(\xi)\,\bigl|\hat g_1(t,x;\xi)-g_1(t,x;\xi)\bigr|
+ |g_1(t,x;\xi)|\,\bigl|\hat f_1(\xi)-f(\xi)\bigr|
\Big).
\end{equation}
Consequently, with $D_{\min}$ defined in \Cref{ass:pair}, on the event
$\left\{\inf_{\xi\in B_\rho}\hat f_1(\xi)\ge \frac{1}{2}f_{\min}\right\}$,
if
\begin{equation}
\sup_{(t,x,\xi)\in Q_{R,\eta,\rho}}
\frac{2}{f_{\min}^2}
\Big( f(\xi)\,\bigl|\hat g_1(t,x;\xi)-g_1(t,x;\xi)\bigr|
+ |g_1(t,x;\xi)|\,\bigl|\hat f_1(\xi)-f(\xi)\bigr|
\Big)
\le \frac{D_{\min}}{2},
\end{equation}
then
\begin{equation}
\inf_{(t,x,\xi)\in Q_{R,\eta,\rho}} \hat D(t,x;\xi)\ge \frac{D_{\min}}{2}.
\end{equation}
\end{lemma}

\begin{proof}
Fix $(t,x,\xi)\in Q_{R,\eta,\rho}$, we have
\begin{equation}
\hat D-D^*
=\frac{\hat g_1}{\hat f_1}-\frac{g_1}{f(\xi)}
=\frac{(\hat g_1-g_1)\ f(\xi)-g_1\ (\hat f_1-f(\xi))}{\hat f_1\ f(\xi)}.
\end{equation}
On the event $\{\hat f_1(\xi)\ge f_{\min}/2\}$, \Cref{ass:pair} gives
$\hat f_1(\xi)\ f(\xi)\ge \frac{f_{\min}^2}{2}$,
and therefore
\begin{equation}
\big|\hat D-D^*\big|
\le
\frac{2}{f_{\min}^{\ 2}}
\Big(
f(\xi)\ \big|\hat g_1-g_1\big|
+
|g_1|\ \big|\hat f_1-f(\xi)\big|
\Big).
\end{equation}
This proves the first claim. Now, assume
$\inf_{\xi\in B_\rho}\hat f_1(\xi)\ge \frac{f_{\min}}{2}$
and that the displayed supremum is at most $D_{\min}/2$. Then, for every
$(t,x,\xi)\in Q_{R,\eta,\rho}$, the first part gives
$\big|\hat D(t,x;\xi)-D^*(t,x;\xi)\big|\le \frac{D_{\min}}{2}$.
Using \Cref{ass:pair},
$D^*(t,x;\xi)\ge D_{\min}$,
hence
\begin{equation}
\hat D(t,x;\xi)
\ge D^*(t,x;\xi)-\big|\hat D(t,x;\xi)-D^*(t,x;\xi)\big|
\ge D_{\min}-\frac{D_{\min}}{2}
=\frac{D_{\min}}{2}.
\end{equation}
Taking the infimum over $(t,x,\xi)\in Q_{R,\eta,\rho}$ proves the claim.
\end{proof}

\begin{proposition}[Deterministic ratio-stability transfer]
\label{prop:ratio}
Fix $(t,x,\xi)\in Q_{R,\eta,\rho}$. On the event
\begin{equation}
\mathcal E(t,x,\xi):=\Big\{
\hat D(t,x;\xi)\ge \tfrac12 D_{\min},\ 
\hat f_1(\xi)\ge \tfrac12 f_{\min},\ 
\hat f_2(\xi)\ge \tfrac12 f_{\min}
\Big\},
\end{equation}
we have
\begin{equation}
|\hat a-a^\ast|
\le
\frac{4}{\Delta(t)\ D_{\min}}
\left(
\frac{f(\xi)\ |\hat g_2-g_2|+|g_2|\ |\hat f_2-f(\xi)|}{f_{\min}^{\ 2}}
+
|Q^\ast|\ 
\frac{f(\xi)\ |\hat g_1-g_1|+|g_1|\ |\hat f_1-f(\xi)|}{f_{\min}^{\ 2}}
\right).
\end{equation}
\end{proposition}

\begin{proof}
Fix $(t,x,\xi)\in Q_{R,\eta,\rho}$ and write $\Delta(t):=u-t$.
Recall
$|\hat a-a^\ast|=\frac{1}{\Delta(t)}\ |\hat Q-Q^\ast|$. Using
\begin{equation}
\hat Q-Q^*
=\frac{\hat N}{\hat D}-\frac{N^*}{D^*}
=\frac{\hat N-N^*}{\hat D}
+N^*\!\left(\frac{1}{\hat D}-\frac{1}{D^*}\right)
=\frac{\hat N-N^*}{\hat D}-\frac{Q^*}{\hat D}(\hat D-D^*),
\end{equation}
we obtain on the event $\{\hat D\ge D_{\min}/2\}$,
\begin{equation}
|\hat Q-Q^*|
\le\frac{2}{D_{\min}}\Big(|\hat N-N^*|+|Q^*|\,|\hat D-D^*|\Big).
\end{equation}
Taking absolute values and using the event $\mathcal E(t,x,\xi)$ (so $\hat D(t,x;\xi)\ge D_{\min}/2$),
\begin{equation}\label{eq:ratio-Q}
|\hat Q-Q^\ast|
\ \le\ \frac{2}{D_{\min}}\Big(|\hat N-N^\ast|+|Q^\ast|\ |\hat D-D^\ast|\Big).
\end{equation}

For $N^*(t,x;\xi)=g_2(t,x;\xi)/f(\xi)$,
$\hat N-N^\ast=\frac{\hat g_2}{\hat f_2}-\frac{g_2}{f(\xi)}
=\frac{(\hat g_2-g_2)\ f(\xi) - g_2\ (\hat f_2-f(\xi))}{\hat f_2\ f(\xi)}$.
On $\mathcal E(t,x,\xi)$, $\hat f_2(\xi)\ge f_{\min}/2$ and \Cref{ass:pair} gives $f(\xi)\ge f_{\min}$, so
$\hat f_2(\xi)\ f(\xi)\ge \frac{f_{\min}^2}{2}$,
and therefore
\begin{equation}\label{eq:ratio-N}
|\hat N-N^\ast|
\ \le\ \frac{2}{f_{\min}^{\ 2}}
\Big(f(\xi)\ |\hat g_2-g_2|+|g_2|\ |\hat f_2-f(\xi)|\Big).
\end{equation}
Similarly, for $D^*=\frac{g_1}{f(\xi)}$,
$\hat D-D^\ast=\frac{\hat g_1}{\hat f_1}-\frac{g_1}{f(\xi)}
=\frac{(\hat g_1-g_1)\ f(\xi) - g_1\ (\hat f_1-f(\xi))}{\hat f_1\ f(\xi)}$,
and on $\mathcal E(t,x,\xi)$,
$\hat f_1(\xi)\ f(\xi)\ge \frac{f_{\min}^2}{2}$,
hence
\begin{equation}\label{eq:ratio-D}
|\hat D-D^\ast|
\ \le\ \frac{2}{f_{\min}^{\ 2}}
\Big(f(\xi)\ |\hat g_1-g_1|+|g_1|\ |\hat f_1-f(\xi)|\Big).
\end{equation}

Plug \eqref{eq:ratio-N}--\eqref{eq:ratio-D} into \eqref{eq:ratio-Q} to get
\begin{equation}
|\hat Q-Q^\ast|
\ \le\ \frac{4}{D_{\min}}\left(
\frac{f(\xi)\ |\hat g_2-g_2|+|g_2|\ |\hat f_2-f(\xi)|}{f_{\min}^{\ 2}}
\ +\ 
|Q^\ast|\ \frac{f(\xi)\ |\hat g_1-g_1|+|g_1|\ |\hat f_1-f(\xi)|}{f_{\min}^{\ 2}}
\right).
\end{equation}
Finally, multiply by $1/\Delta(t)$ to pass from $Q$ to the drift:
\begin{equation}
|\hat a-a^\ast|
\ \le\  \frac{4}{\Delta(t)\ D_{\min}}\left(
\frac{f(\xi)\ |\hat g_2-g_2|+|g_2|\ |\hat f_2-f(\xi)|}{f_{\min}^{\ 2}}
\ +\ 
|Q^\ast|\ \frac{f(\xi)\ |\hat g_1-g_1|+|g_1|\ |\hat f_1-f(\xi)|}{f_{\min}^{\ 2}}
\right).
\end{equation}
This is the claimed bound.
\end{proof}

\section{Uniform Kernel Deviations}
\label{app:kde}

Let $C_K>0$ denote the entropy constant appearing in \Cref{lem:kde,lem:g1,lem:g2}, and define
\begin{equation}
\widetilde r_j(\delta)
:=
h_j^\beta
+
\sqrt{\frac{d\log(\tfrac{C_K}{h_j})+\log(\tfrac{8}{\delta})}{M h_j^{d}}},
\qquad j\in\{1,2\}.
\end{equation}
We call a bandwidth pair $(h_1,h_2)$ \emph{$\delta$-admissible} if
\begin{equation}
\widetilde r_1(\delta)+\widetilde r_2(\delta)\le c_0.
\end{equation}

\begin{theorem}[Uniform finite-sample bound]
\label{thm:drift}
Assume \Cref{ass:kernel,ass:pair}. Then there exist constants
\begin{equation}
c_0>0,
\qquad
C = C(\beta,L_K,\mu_\beta(K),\mu_\infty(K),|K|_2,C_f,d,L_p,f_{\min},D_{\min},R,\eta,\rho,\Delta,\operatorname{diam}(B))
\end{equation}
such that, for every $\delta$-admissible bandwidth pair $(h_1,h_2)$, with probability at least
$1-\delta$,
\begin{equation}
\sup_{(t,x,\xi)\in Q_{R,\eta,\rho}}|\hat a-a^\ast|
\le
\frac{C}{\eta}\Bigg(
h_1^\beta+h_2^\beta
+\sqrt{\frac{d\log(\tfrac{C_K}{h_1})+\log(\tfrac{2}{\delta})}{M h_1^{d}}}
+\sqrt{\frac{d\log(\tfrac{C_K}{h_2})+\log(\tfrac{2}{\delta})}{M h_2^{d}}}
\Bigg).
\end{equation}
\end{theorem}

\begin{remark}
\label{rem:drift-implies-finite}
Since $\Delta(t)\ge \eta$ on $ Q_{R,\eta,\rho}$, \Cref{thm:drift} implies \Cref{thm:finite}
for every $\delta$-admissible bandwidth pair $(h_1,h_2)$.
\end{remark}

\begin{proof}
By bounded support and continuity, the quantities
$f(\xi),\ g_1(t,x;\xi),\ g_2(t,x;\xi),\ Q^\ast(t,x;\xi)$
are uniformly bounded on $ Q_{R,\eta,\rho}$. Hence there exist deterministic constants
$B_f,\ B_{g_1},\ B_{g_2},\ B_Q>0$ such that
\begin{equation}
\sup_{\xi\in B_\rho} f(\xi)\le B_f,\qquad
\sup_{(t,x,\xi)\in Q_{R,\eta,\rho}} |g_1(t,x;\xi)|\le B_{g_1},
\end{equation}
\begin{equation}
\sup_{(t,x,\xi)\in Q_{R,\eta,\rho}} |g_2(t,x;\xi)|\le B_{g_2},
\qquad
\sup_{(t,x,\xi)\in Q_{R,\eta,\rho}} |Q^\ast(t,x;\xi)|\le B_Q.
\end{equation}

By \Cref{lem:kde,lem:g1,lem:g2}, there exists a constant $A_1,\ A_2,\ A_3,\ A_4>0$ and events of probability at least $1-\delta/4$ for each such that the following inequalities hold,
\begin{align}
\label{eq:g1g2f1f2-bound}
\sup_{t\in[s,u-\eta],\ x\in\R^d,\ \xi\in B_\rho}
|\hat g_1-g_1|
\le A_1\ \widetilde r_1(\delta),&\qquad
\sup_{t\in[s,u-\eta],\ x\in\R^d,\ \xi\in B_\rho}
|\hat g_2-g_2|
\le A_2\ \widetilde r_2(\delta),\\
\sup_{\xi\in B_\rho} |\hat f_1-f|
\le A_3\ \widetilde r_1(\delta),&\qquad
\sup_{\xi\in B_\rho} |\hat f_2-f|
\le A_4\ \widetilde r_2(\delta).
\end{align}
Since $(h_1,h_2)$ is $\delta$-admissible and $c_0$ is chosen small, the quantities
$\widetilde\Lambda(h_j,\delta)/(M h_j^d),\  j\in\{1,2\}$,
are bounded by $1$. Therefore
\begin{equation}
\frac{\widetilde\Lambda(h_j,\delta)}{M h_j^d}
\le
\sqrt{\frac{\widetilde\Lambda(h_j,\delta)}{M h_j^d}},
\qquad j\in\{1,2\},
\end{equation}
so the lower-order linear concentration remainders in
\Cref{lem:kde,lem:g1,lem:g2} are absorbed into the constants multiplying
$\widetilde r_j(\delta)$.

Let $\Omega$ denote the intersection of the four events in
\eqref{eq:g1g2f1f2-bound}. By a union bound, $\Pr(\Omega)\ge 1-\delta$.

We now verify the denominator-floor event on $\Omega$. First, on $\Omega$,
\begin{equation}
\inf_{\xi\in B_\rho}\hat f_1(\xi)\ge f_{\min}-A_3\ \widetilde r_1(\delta),
\qquad
\inf_{\xi\in B_\rho}\hat f_2(\xi)\ge f_{\min}-A_4\ \widetilde r_2(\delta).
\end{equation}
Next, by \Cref{lem:denom-stability}, on the event
$\Big\{\inf_{\xi\in B_\rho}\hat f_1(\xi)\ge \tfrac12 f_{\min}\Big\}$,
we have
\begin{equation}
\sup_{(t,x,\xi)\in Q_{R,\eta,\rho}} |\hat D-D^\ast|
\le
\frac{2}{f_{\min}^{\ 2}}
\Big(
B_f\ \sup_{t,x,\xi}|\hat g_1-g_1|
+
B_{g_1}\ \sup_{\xi}|\hat f_1-f|
\Big),
\end{equation}
and therefore, on $\Omega$,
$\sup_{(t,x,\xi)\in Q_{R,\eta,\rho}} |\hat D-D^\ast|
\le \frac{2}{f_{\min}^{\ 2}}
\Big(B_fA_1+B_{g_1}A_3\Big)\ \widetilde r_1(\delta)$.

Let
$c_0 := \min\!\left\{
\frac{f_{\min}}{2A_3},\
\frac{f_{\min}}{2A_4},\
\frac{D_{\min}f_{\min}^{\ 2}}{4(B_fA_1+B_{g_1}A_3)}
\right\}$.
If $(h_1,h_2)$ is $\delta$-admissible, then
$\widetilde r_1(\delta)+\widetilde r_2(\delta)\le c_0$,
hence, in particular,
$A_3\widetilde r_1(\delta)\le \frac{f_{\min}}{2},
\ A_4\widetilde r_2(\delta)\le \frac{f_{\min}}{2}$,
and
$\frac{2}{f_{\min}^{\ 2}}
\Big(B_fA_1+B_{g_1}A_3\Big)\ \widetilde r_1(\delta)\le \frac{D_{\min}}{2}$.
Therefore, on $\Omega$,
$\inf_{\xi\in B_\rho}\hat f_1(\xi)\ge \frac{f_{\min}}{2},
\ \inf_{\xi\in B_\rho}\hat f_2(\xi)\ge \frac{f_{\min}}{2},
\ \inf_{(t,x,\xi)\in Q_{R,\eta,\rho}}\hat D(t,x;\xi)\ge \frac{D_{\min}}{2}$.
In other words, on $\Omega$ the event
\begin{equation}
\mathcal E:=
\Big\{
\inf_{(t,x,\xi)\in Q_{R,\eta,\rho}}\hat D(t,x;\xi)\ge \tfrac12 D_{\min},\
\inf_{\xi\in B_\rho}\hat f_1(\xi)\ge \tfrac12 f_{\min},\
\inf_{\xi\in B_\rho}\hat f_2(\xi)\ge \tfrac12 f_{\min}
\Big\}
\end{equation}
holds.

Now fix $(t,x,\xi)\in Q_{R,\eta,\rho}$. On $\mathcal E$, \Cref{prop:ratio} gives
\begin{equation}
\label{eq:ratio-master}
|\hat a-a^\ast|
\le \frac{4}{\Delta(t)\ D_{\min}}\left(
\frac{f(\xi)\ |\hat g_2-g_2|+|g_2|\ |\hat f_2-f|}{f_{\min}^{\ 2}}
+ |Q^\ast|\ \frac{f(\xi)\ |\hat g_1-g_1|+|g_1|\ |\hat f_1-f|}{f_{\min}^{\ 2}}
\right).
\end{equation}
Since $\Omega\subset\mathcal E$, we may insert the uniform bounds
\eqref{eq:g1g2f1f2-bound} into \eqref{eq:ratio-master}. Using the deterministic bounds
$f(\xi)\le B_f,\  |g_1|\le B_{g_1},\  |g_2|\le B_{g_2},\  |Q^\ast|\le B_Q$,
we obtain on $\Omega$
\begin{equation}
|\hat a-a^*|
\le \frac{C}{\eta}\bigl(\widetilde r_1(\delta)+\widetilde r_2(\delta)\bigr),
\end{equation}
for a constant $C>0$ depending only on
\begin{equation}
(\beta,L_K,\mu_\beta(K),\mu_\infty(K),|K|_2,C_f,d,L_p,f_{\min},D_{\min},R,\eta,\Delta,\operatorname{diam}(B)).
\end{equation}
Expanding the definition of $\widetilde r_1(\delta)$ and $\widetilde r_2(\delta)$, and absorbing the harmless
difference between $\log(8/\delta)$ and $\log(2/\delta)$ into the constant $C$, gives the claimed bound.
\end{proof}

\begin{corollary}[Oracle bandwidth choice]
\label{cor:oracle}
Set $h_1=h_2=h$ and let
\begin{equation}
L := d\log(C_K M)+\log(2/\delta), \qquad
h^\star := \left(\frac{L}{M}\right)^{1/(2\beta+d)}.
\end{equation}
Then for each fixed $\delta\in(0,1)$ there exists $M_0(\delta)\ge 1$, depending on $\delta$ and the model constants, such that for all $M\ge M_0(\delta)$, the pair $(h^\star,h^\star)$ is $\delta$-admissible and, with probability at least $1-\delta$,
\begin{equation}
\sup_{(t,x,\xi)\in Q_{R,\eta,\rho}}
|\hat a(t,x;\xi)-a^*(t,x;\xi)|
\le
\frac{C}{\eta}\left(\frac{L}{M}\right)^{\beta/(2\beta+d)}.
\end{equation}
Moreover, for $\delta=M^{-\gamma}$ with fixed $\gamma>0$,
\begin{equation}
\sup_{(t,x,\xi)\in Q_{R,\eta,\rho}}
|\hat a(t,x;\xi)-a^*(t,x;\xi)|
=
O_{\P}\!\left(\left(\frac{\log M}{M}\right)^{\beta/(2\beta+d)}\right).
\end{equation}
\end{corollary}

\begin{proof}
Since $(h^\star)^\beta=(L/M)^{\beta/(2\beta+d)}$, we have
\begin{equation}
\log\!\Big(\frac{C_K}{h^\star}\Big)
=
\log C_K + \frac{1}{2\beta+d}\log\!\Big(\frac{M}{L}\Big)
\le
\log(C_K M)
\le
\frac{L}{d}.
\end{equation}
Hence
\begin{equation}
\sqrt{\frac{d\log(C_K/h^\star)+\log(8/\delta)}{M (h^\star)^d}}
\le
\sqrt{2}\left(\frac{L}{M}\right)^{\beta/(2\beta+d)}.
\end{equation}
Therefore, for each fixed $\delta\in(0,1)$, there exists $M_0(\delta)\ge 1$ such that for all $M\ge M_0(\delta)$,
\begin{equation}
\widetilde r_1(\delta)+\widetilde r_2(\delta)\le c_0.
\end{equation}
Thus $(h^\star,h^\star)$ is $\delta$-admissible for all $M\ge M_0(\delta)$, and \Cref{thm:drift} applies. Since $\Delta(t)\ge \eta$ on $Q_{R,\eta,\rho}$,
\begin{equation}
\sup_{(t,x,\xi)\in Q_{R,\eta,\rho}}
|\hat a(t,x;\xi)-a^*(t,x;\xi)|
\le
\frac{C}{\eta}\left(\frac{L}{M}\right)^{\beta/(2\beta+d)}.
\end{equation}
If $\delta=M^{-\gamma}$, then $L=d\log(C_K M)+\gamma\log M\le (d+\gamma)\log M+O(1)$, and therefore
\begin{equation}
\left(\frac{L}{M}\right)^{\beta/(2\beta+d)}
=
O\!\left(\left(\frac{\log M}{M}\right)^{\beta/(2\beta+d)}\right).
\end{equation}
This proves the final claim.
\end{proof}

\section{Unequal-bandwidth pointwise asymptotic normality}
\label{app:clt}

In this Appendix we state and prove a pointwise CLT for the estimator \eqref{eq:adrift-hat} with
possibly different bandwidths $h_1$ and $h_2$. The main-text \Cref{thm:clt} is recovered as the
diagonal specialization $h_1=h_2=h$; see \Cref{rem:clt-diagonal} below.

Fix $(t,x,\xi)\in [s,u-\eta]\times B_R\times \operatorname{int}(B)$. Write
$U(y):=y\ F_{t,x,\xi}(y)\in\R^d,\ 
V(y):=F_{t,x,\xi}(y)\in\R$.
Then
\begin{equation}
N^*(t,x;\xi)=\E\!\big[U(X_u)\mid X_s=\xi\big],\qquad
D^*(t,x;\xi)=\E\!\big[V(X_u)\mid X_s=\xi\big],
\end{equation}
\begin{equation}
Q^*(t,x;\xi)=\frac{N^*(t,x;\xi)}{D^*(t,x;\xi)},
\qquad
a^*(t,x;\xi)=\frac{Q^*(t,x;\xi)-x}{\Delta_t}.
\end{equation}
Also let
$R(K):=\int_{\R^d}K(u)^2\ du,
\ R_\rho(K):=\int_{\R^d}K(u)K(\rho u)\ du,\ \rho>0$. Writing $U:=U(X_u)$ and $V:=V(X_u)$ for brevity, define
\begin{equation}
n_U:=\E[U\mid X_s=\xi],\qquad
n_V:=\E[V\mid X_s=\xi],\qquad
n_{UV}:=\E[UV\mid X_s=\xi],
\end{equation}
\begin{equation}
m_{UU}:=\E[UU^\top\mid X_s=\xi],\qquad
m_{V^2}:=\E[V^2\mid X_s=\xi].
\end{equation}
Then the block covariance matrix is
\begin{equation}
\setlength{\arraycolsep}{3pt}
\Omega_\rho(t,x;\xi)
:=
f(\xi)
\begin{bmatrix}
\rho^{-d}R(K)\ m_{UU}
&
R_\rho(K)\ n_{UV}
&
R_\rho(K)\ n_U
&
\rho^{-d}R(K)\ n_U
\\
R_\rho(K)\ n_{UV}^\top
&
R(K)\ m_{V^2}
&
R(K)\ n_V
&
R_\rho(K)\ n_V
\\
R_\rho(K)\ n_U^\top
&
R(K)\ n_V
&
R(K)
&
R_\rho(K)
\\
\rho^{-d}R(K)\ n_U^\top
&
R_\rho(K)\ n_V
&
R_\rho(K)
&
\rho^{-d}R(K)
\end{bmatrix}.
\end{equation}

Let
$\Gamma_t(a,b,c,e):=\frac{1}{\Delta_t}\left(\frac{ac}{be}-x\right),
\ (a,b,c,e)\in \R^d\times\R\times\R\times\R$,
where $ac$ denotes scalar multiplication of the vector $a$ by the scalar $c$. Denote by
$J_\Gamma(t,x;\xi):=
D\Gamma_t\!\big(g_2(t,x;\xi),\ g_1(t,x;\xi),\ f(\xi),\ f(\xi)\big)$
the Jacobian of $\Gamma_t$ at the population point. Explicitly, since $g_1(t,x;\xi)=f(\xi)D^*(t,x;\xi)$, set
\begin{equation}
A:=\bigl(f(\xi)D^*(t,x;\xi)\bigr)^{-1}
    =g_1(t,x;\xi)^{-1},
\qquad
B:=f(\xi)^{-1}.
\end{equation}
Then
\begin{equation}
J_\Gamma(t,x;\xi)
=
\frac{1}{\Delta_t}
\Bigl[
A I_d,\,
-AQ^*(t,x;\xi),\,
BQ^*(t,x;\xi),\,
-BQ^*(t,x;\xi)
\Bigr].
\end{equation}
This is viewed as a $d\times(d+3)$ block row. Finally, define
$\Sigma_\rho(t,x;\xi):=
J_\Gamma(t,x;\xi)\ \Omega_\rho(t,x;\xi)\ J_\Gamma(t,x;\xi)^\top$.

\begin{theorem}[Unequal-bandwidth pointwise CLT]
\label{thm:clt-unequal}
Assume \Cref{ass:kernel,ass:pair}. Let $h_1\to0$, $h_2\to0$ satisfy
$M h_1^d\to\infty,\ 
M h_2^d\to\infty,\ 
h_1^\beta\sqrt{M h_1^d}\to0,\ 
h_2^\beta\sqrt{M h_1^d}\to0$,
and suppose moreover that
$h_2/h_1\to \rho\in(0,\infty)$.
Then, for every fixed $(t,x,\xi)\in [s,u-\eta]\times B_R\times \operatorname{int}(B)$,
\begin{equation}
\sqrt{M h_1^d}\ 
\bigl(\hat a(t,x;\xi)-a^*(t,x;\xi)\bigr)
\ \Rightarrow\
\mathcal N\!\bigl(0,\Sigma_\rho(t,x;\xi)\bigr),
\end{equation}
where $\Sigma_\rho(t,x;\xi)$ is defined above.
\end{theorem}

\begin{remark}[Diagonal specialization]
\label{rem:clt-diagonal}
If $h_1=h_2=h$, then $\rho=1$ and $R_\rho(K)=R(K)$. In this case, a direct simplification of
$J_\Gamma\ \Omega_1\ J_\Gamma^\top$ yields
\begin{equation}
\Sigma_1(t,x;\xi)
=
\frac{R(K)}{f(\xi)}
\cdot
\frac{
\operatorname{Var}\!\bigl(\psi_{t,x,\xi}(X_u)\mid X_s=\xi\bigr)
}{
\Delta(t)^2\ D^*(t,x;\xi)^2
},
\end{equation}
with
\begin{equation}
\psi_{t,x,\xi}(y)=
\bigl(y-x-\Delta(t)a^*(t,x;\xi)\bigr)\ F(t,\xi,x,y).
\end{equation}
Hence \Cref{thm:clt-unequal} implies the main-text \Cref{thm:clt}.
\end{remark}

\begin{proof}[Proof of \Cref{thm:clt-unequal}]
Fix $(t,x,\xi)\in [s,u-\eta]\times B_R\times \operatorname{int}(B)$.
Since $\xi\in\operatorname{int}(B)$ and $K$ is compactly supported, for all sufficiently large $M$ we have
$\xi+h_j u\in B$ for all $u\in\operatorname{supp}(K)$ and $j\in\{1,2\}$,
so the kernel changes of variables below are valid without boundary effects.

Define
$\hat T:=\bigl(\hat g_2^\top,\hat g_1,\hat f_1,\hat f_2\bigr)^\top,
\ T:=\bigl(g_2(t,x;\xi)^\top,g_1(t,x;\xi),f(\xi),f(\xi)\bigr)^\top$. By the theorem's assumptions,
$\sqrt{M h_1^d}\,h_1^\beta \to 0,
\  \sqrt{M h_1^d}\,h_2^\beta \to 0$.
Since the componentwise bias bounds satisfy
\begin{equation}
|\mathbb E\hat g_1-g_1| \lesssim h_1^\beta,\qquad
|\mathbb E\hat g_2-g_2| \lesssim h_2^\beta,\qquad
|\mathbb E\hat f_j-f| \lesssim h_j^\beta,
\end{equation}
it follows that
$\sqrt{M h_1^d}\,|\mathbb E\hat T-T| \to 0$.
Since $\Gamma_t$ is continuously differentiable in a neighborhood of $T$, it follows that
$\sqrt{M h_1^d}\ \bigl(\Gamma_t(\E\hat T)-\Gamma_t(T)\bigr)\to0$.

For each $m$, define
\begin{equation}
Z_m:=
\begin{pmatrix}
U(X_u^{(m)})\ K_{h_2}(X_s^{(m)}-\xi)\\
V(X_u^{(m)})\ K_{h_1}(X_s^{(m)}-\xi)\\
K_{h_1}(X_s^{(m)}-\xi)\\
K_{h_2}(X_s^{(m)}-\xi)
\end{pmatrix}
\in\R^{d+3},
\qquad
\hat T=\frac1M\sum_{m=1}^M Z_m.
\end{equation}

To identify the covariance limit, write
$\widetilde n_A(z) := \int_B A(y)\,p(z,y)\,dy$
for bounded measurable $A:B\to\mathbb R$, $A:B\to\mathbb R^d$, or $A:B\to\mathbb R^{d\times d}$. Then
$\widetilde n_U(\xi)=g_2(t,x;\xi),\  \widetilde n_V(\xi)=g_1(t,x;\xi)$,
and
\begin{equation}
\widetilde n_{UU^\top}(\xi)=f(\xi)\,E[UU^\top\mid X_s=\xi],\ 
\widetilde n_{UV}(\xi)=f(\xi)\,E[UV\mid X_s=\xi],\ 
\widetilde n_{V^2}(\xi)=f(\xi)\,E[V^2\mid X_s=\xi].
\end{equation}

A standard change of variables gives
\begin{align}
h_1^d\ \E\!\big[U U^\top K_{h_2}(X_s-\xi)^2\big]
&=\frac{h_1^d}{h_2^d}\int_{\R^d} K(u)^2\ \widetilde n_{UU^\top}(\xi+h_2u)\ du
\to
\rho^{-d}R(K)\ \widetilde n_{UU^\top}(\xi), \\
h_1^d\ \E\!\big[V^2 K_{h_1}(X_s-\xi)^2\big]
&=
\int_{\R^d} K(u)^2\ \widetilde n_{V^2}(\xi+h_1u)\ du
\to
R(K)\ \widetilde n_{V^2}(\xi), \\
h_1^d\ \E\!\big[K_{h_1}(X_s-\xi)^2\big]
&\to R(K)\ f(\xi),\qquad
h_1^d\ \E\!\big[K_{h_2}(X_s-\xi)^2\big]
\to \rho^{-d}R(K)\ f(\xi).
\end{align}

Writing $\rho_M:=h_2/h_1$, the mixed terms satisfy
\begin{equation}
h_1^d\ \E\!\big[U V\ K_{h_2}(X_s-\xi)K_{h_1}(X_s-\xi)\big]
=
\int_{\R^d} K(u)\ K(\rho_Mu)\ \widetilde n_{UV}(\xi+h_2u)\ du
\to
R_\rho(K)\ \widetilde n_{UV}(\xi),
\end{equation}
by dominated convergence. Similarly,
\begin{align}
h_1^d\ \E\!\big[U\ K_{h_2}(X_s-\xi)K_{h_1}(X_s-\xi)\big]
&\to R_\rho(K)\ \widetilde n_U(\xi), \\
h_1^d\ \E\!\big[V\ K_{h_2}(X_s-\xi)K_{h_1}(X_s-\xi)\big]
&\to R_\rho(K)\ \widetilde n_V(\xi), \\
h_1^d\ \E\!\big[K_{h_2}(X_s-\xi)K_{h_1}(X_s-\xi)\big]
&\to R_\rho(K)\ f(\xi), \\
h_1^d\ \E\!\big[U\ K_{h_2}(X_s-\xi)^2\big]
&\to \rho^{-d}R(K)\ \widetilde n_U(\xi), \\
h_1^d\ \E\!\big[V\ K_{h_1}(X_s-\xi)^2\big]
&\to R(K)\ \widetilde n_V(\xi).
\end{align}

Subtracting products of expectations does not change the limit. Indeed, since
$\widehat T=\frac1M\sum_{m=1}^M Z_m$,
we have
$M h_1^d\operatorname{Cov}(\widehat T)
= h_1^d\Bigl(\mathbb{E}[Z_m Z_m^\top]
-\mathbb{E}Z_m\,\mathbb{E}Z_m^\top\Bigr)$.
The preceding displays identify the limit of
$h_1^d\mathbb{E}[Z_m Z_m^\top]$. Moreover, the components of
$\mathbb{E}Z_m$ are uniformly bounded, because $U,V$ are bounded on $B$
and the kernel averages converge to the corresponding population
quantities. Therefore
$h_1^d\,\mathbb{E}Z_m\,\mathbb{E}Z_m^\top
= h_1^d O(1)\to 0$.
Consequently,
\begin{equation}
M h_1^d\operatorname{Cov}(\widehat T)
\to
\Omega_\rho(t,x;\xi).
\end{equation}

Now define the centered triangular-array variables
$Y_{m,M}:=\sqrt{\frac{h_1^d}{M}}\ (Z_m-\E Z_m),\  m=1,\dots,M$.
Then
$\sqrt{M h_1^d}\ (\hat T-\E\hat T)=\sum_{m=1}^M Y_{m,M}$.
Because $X_s,X_u$ are supported on the bounded set $B\times B$, the weights $U(X_u)$ and $V(X_u)$
are bounded, and $K$ is bounded with compact support, there exists $C_Z<\infty$ such that
$|Z_m|\le C_Z\big(h_1^{-d}+h_2^{-d}\big)$.
Since $h_2/h_1\to\rho\in(0,\infty)$, we have $h_2^{-d}\lesssim h_1^{-d}$, and hence
\begin{equation}
|Y_{m,M}|
\le
C\ \frac{\sqrt{h_1^d/M}}{h_1^d}
=
\frac{C}{\sqrt{M h_1^d}}
\to0.
\end{equation}
Therefore, for every $\varepsilon>0$, the Lindeberg condition holds:
$\sum_{m=1}^M \E\!\left[|Y_{m,M}|^2\mathbf 1_{\{|Y_{m,M}|>\varepsilon\}}\right]\to0$.
By the multivariate Lindeberg--Feller theorem,
$\sqrt{M h_1^d}\ (\hat T-\E\hat T)
\ \Rightarrow\
\mathcal N\!\bigl(0,\Omega_\rho(t,x;\xi)\bigr)$.

By definition,
$\hat a(t,x;\xi)=\Gamma_t(\hat T),
\ a^*(t,x;\xi)=\Gamma_t(T)$.
Since
$g_1(t,x;\xi)=f(\xi)D^*(t,x;\xi)$, $f(\xi)\ge f_{\min}>0$, $D^*(t,x;\xi)\ge D_{\min}>0$
by \Cref{ass:pair}, the point
\begin{equation}
T=\bigl(g_2(t,x;\xi)^\top,g_1(t,x;\xi),f(\xi),f(\xi)\bigr)^\top
\end{equation}
lies in the open set $\R^d\times(0,\infty)\times(0,\infty)\times(0,\infty)$, on which $\Gamma_t$ is
$C^1$. Therefore, by the multivariate delta method,
\begin{equation}
\sqrt{M h_1^d}\ \bigl(\Gamma_t(\hat T)-\Gamma_t(\E\hat T)\bigr)
\ \Rightarrow\
\mathcal N\!\bigl(0,\Sigma_\rho(t,x;\xi)\bigr),
\end{equation}
with
$\Sigma_\rho(t,x;\xi)=
J_\Gamma(t,x;\xi)\ \Omega_\rho(t,x;\xi)\ J_\Gamma(t,x;\xi)^\top$.

Finally,
\begin{align}
\sqrt{M h_1^d}\ \bigl(\hat a(t,x;\xi)-a^*(t,x;\xi)\bigr)
&=
\sqrt{M h_1^d}\ \bigl(\Gamma_t(\hat T)-\Gamma_t(\E\hat T)\bigr) \\
&\quad+
\sqrt{M h_1^d}\ \bigl(\Gamma_t(\E\hat T)-\Gamma_t(T)\bigr).
\end{align}
The first term converges to $\mathcal N(0,\Sigma_\rho(t,x;\xi))$ by the delta-method argument above,
and the second vanishes by the bias bound. This proves the claim.
\end{proof}

\section{Adaptive selector and oracle inequality}
\label{app:adaptive}

In this Appendix we state and prove an adaptive oracle inequality for the estimator
\eqref{eq:adrift-hat} with possibly different bandwidths $h_1$ and $h_2$. The main-text
\Cref{thm:oracle,cor:adaptive} are recovered by restricting to the diagonal choice
$h_1=h_2=h$; see \Cref{rem:oracle-diagonal} below.

\subsection{Decoupled GL selectors: grids, deviations, and penalties}
\label{app:GL}

Let $\mathcal H_j=\{h_{j,0}q^k:\ k=0,\dots,K_j\}\subset(0,1]$ be geometric grids for
$j\in\{1,2\}$, with ratio $q\in(0,1)$, largest element $h_{j,0}\asymp 1$, and smallest element
$h_{j,K_j}\asymp M^{-1/d}$. For $h\in\mathcal H_1$, let
\begin{equation}
\hat g_1^{\ h}(t,x;\xi)
:=
\frac1M\sum_{m=1}^M F(t,\xi,x,X_u^{(m)})\ K_h(X_s^{(m)}-\xi),
\qquad
\hat f_1^{\ h}(\xi)
:=
\frac1M\sum_{m=1}^M K_h(X_s^{(m)}-\xi),
\end{equation}
and for $h\in\mathcal H_2$, let
\begin{equation}
\hat g_2^{\ h}(t,x;\xi)
:=
\frac1M\sum_{m=1}^M X_u^{(m)}\ F(t,\xi,x,X_u^{(m)})\ K_h(X_s^{(m)}-\xi),
\qquad
\hat f_2^{\ h}(\xi)
:=
\frac1M\sum_{m=1}^M K_h(X_s^{(m)}-\xi).
\end{equation}
Write $\hat a_{h_1,h_2}$ for the estimator \eqref{eq:adrift-hat} constructed from
$(\hat g_1^{\ h_1},\hat g_2^{\ h_2},\hat f_1^{\ h_1},\hat f_2^{\ h_2})$.

For technical completeness, we use the standard denominator-floor convention: on the complement of
the event
\begin{equation}
\mathcal E(h_1,h_2)
:=
\Big\{
\inf_{\xi\in B_\rho}\hat f_1^{\ h_1}(\xi)\ge \tfrac12 f_{\min},
\ \inf_{\xi\in B_\rho}\hat f_2^{\ h_2}(\xi)\ge \tfrac12 f_{\min},
\ \inf_{(t,x,\xi)\in Q_{R,\eta,\rho}}\hat D_{h_1,h_2}(t,x;\xi)\ge \tfrac12 D_{\min}
\Big\},
\end{equation}
we set the estimator equal to $0$. This convention is inactive on the good event used in the proof
below and contributes only to the negligible remainder term in the final risk bound.

\paragraph{Block errors and pairwise deviations.}
For $h\in\mathcal H_1$, define
\begin{equation}
E_1(h)
:=
\max\!\left\{
\sup_{(t,x,\xi)\in Q_{R,\eta,\rho}}
\big|\hat g_1^{\ h}(t,x;\xi)-g_1(t,x;\xi)\big|,
\ 
\sup_{\xi\in B_\rho}
\big|\hat f_1^{\ h}(\xi)-f(\xi)\big|
\right\},
\end{equation}
and for $h\in\mathcal H_2$, define
\begin{equation}
E_2(h)
:=
\max\!\left\{
\sup_{(t,x,\xi)\in Q_{R,\eta,\rho}}
\big|\hat g_2^{\ h}(t,x;\xi)-g_2(t,x;\xi)\big|,
\ 
\sup_{\xi\in B_\rho}
\big|\hat f_2^{\ h}(\xi)-f(\xi)\big|
\right\}.
\end{equation}
For $h'\le h$ in $\mathcal H_j$, define
\begin{equation}
\label{eq:Dev-app}
\mathrm{Dev}_j(h',h)
:=
\max\!\left\{
\sup_{(t,x,\xi)\in Q_{R,\eta,\rho}}
\big|\hat g_j^{\ h}(t,x;\xi)-\hat g_j^{\ h'}(t,x;\xi)\big|,
\
\sup_{\xi\in B_\rho}
\big|\hat f_j^{\ h}(\xi)-\hat f_j^{\ h'}(\xi)\big|
\right\},
\qquad j\in\{1,2\}.
\end{equation}

\paragraph{Penalties.}
Let
\begin{equation}
\Lambda_{\mathcal H}(h,\delta)
:=
d\log\!\left(\frac{C_K}{h}\right)
+
\log\!\left(\frac{4(|\mathcal H_1|+|\mathcal H_2|)}{\delta}\right),
\end{equation}
where $C_K>0$ is the entropy constant from Appendix~\ref{Appendix A}.
Let $A_j,B_{j,1},B_{j,2}>0$ be deterministic constants, depending only on
the model parameters, such that the uniform deviation bounds of
\Cref{lem:kde,lem:g1,lem:g2}, together with a union bound over
$\mathcal H_1\cup\mathcal H_2$, imply that, with probability at least
$1-\delta$,
\begin{equation}
E_j(h)
\le
A_jh^\beta
+
B_{j,1}
\sqrt{\frac{\Lambda_{\mathcal H}(h,\delta)}{Mh^d}}
+
B_{j,2}
\frac{\Lambda_{\mathcal H}(h,\delta)}{Mh^d}
\end{equation}
simultaneously for all $h\in\mathcal H_j$ and $j\in\{1,2\}$. Set
\begin{equation}
b_j(h):=A_jh^\beta,
\qquad
v_j(h):=
B_{j,1}
\sqrt{\frac{\Lambda_{\mathcal H}(h,\delta)}{Mh^d}}
+
B_{j,2}
\frac{\Lambda_{\mathcal H}(h,\delta)}{Mh^d},
\qquad j\in\{1,2\}
\end{equation}
and, for calibration constants $\kappa_j,\kappa_j'>0$ with $\kappa_j'\ge \kappa_j$,
\begin{equation}
\mathrm{pen}_j(h',h):=\kappa_j\ [\ v_j(h)+v_j(h')\ ],
\qquad
\mathrm{pen}_j(h):=\kappa_j' v_j(h),
\qquad j\in\{1,2\}.
\end{equation}

\paragraph{Bias proxies and decoupled selectors.}
For each $j\in\{1,2\}$, define the GL bias proxy
\begin{equation}
\label{eq:Bj-app}
B_j(h):=\sup_{h'\le h}\ \big\{\mathrm{Dev}_j(h',h)-\mathrm{pen}_j(h',h)\big\}_+,
\end{equation}
and select
\begin{equation}
\label{eq:selector-app}
\hat h_j \in \arg\min_{h\in\mathcal H_j}\ \Big\{\ B_j(h)+\mathrm{pen}_j(h)\ \Big\},
\qquad j\in\{1,2\}.
\end{equation}
Selection is \emph{decoupled}: we use $\hat h_1$ for $(\hat g_1,\hat f_1)$ and $\hat h_2$ for
$(\hat g_2,\hat f_2)$. Ties are broken by choosing the largest bandwidth in the argmin.

\subsection{Uniform deviations and bias inputs}
\label{app:GL-inputs}

Let $C_{\rm bias}:=(C_B L_p/\ell_\beta!)\mu_\beta(K)$, where $\ell_\beta$ is as in Assumption~\ref{ass:kernel}.
By the proofs of \Cref{lem:kde,lem:g1,lem:g2}, there exist deterministic constants
$A_j,B_{j,1},B_{j,2}>0$, depending only on
\begin{equation}
(\beta,L_K,\mu_\beta(K),\mu_\infty(K),|K|_2,C_f,d,L_p,f_{\min},D_{\min},R,\eta,\Delta,\operatorname{diam}(B)),
\end{equation}
such that the following holds. For every $\delta\in(0,1)$, with probability at least $1-\delta$,
simultaneously for all $h\in\mathcal H_j$ and $j\in\{1,2\}$,
\begin{equation}
\label{eq:block-dev}
E_j(h)\le b_j(h)+v_j(h).
\end{equation}
Moreover, uniformly in $h\in\mathcal H_j$ and $j\in\{1,2\}$,
\begin{equation}
\label{eq:block-bias}
\max\!\left\{
\sup_{(t,x,\xi)\in Q_{R,\eta,\rho}}\big|\E\hat g_j^{\ h}-g_j\big|,
\ 
\sup_{\xi\in B_\rho}\big|\E\hat f_j^{\ h}-f\big|
\right\}
\le b_j(h).
\end{equation}

\begin{theorem}[Unequal-bandwidth oracle inequality]
\label{thm:oracle-unequal}
Under \Cref{ass:kernel,ass:pair}, there exist constants $M_0\ge 1$ and
$C_1,C_2,C_3>0$ such that, for all $M\ge M_0$, the decoupled selector
$(\hat h_1,\hat h_2)$ defined by \eqref{eq:selector-app} satisfies
\begin{align}
\sup_{\nu\in\mathfrak P(\beta,L_p,f_{\min},B)}
\mathcal R_\nu\bigl(\hat a_{\hat h_1,\hat h_2}\bigr)
&\le
C_1\inf_{h_1\in\mathcal H_1}
\left\{
h_1^\beta+\sqrt{\frac{\log M}{M h_1^d}}+\frac{\log M}{M h_1^d}
\right\}\\
&\qquad\qquad+
C_2\inf_{h_2\in\mathcal H_2}
\left\{
h_2^\beta+\sqrt{\frac{\log M}{M h_2^d}}+\frac{\log M}{M h_2^d}
\right\}
+
C_3\ M^{-1}.
\end{align}
\end{theorem}

\begin{remark}[Diagonal specialization]
\label{rem:oracle-diagonal}
If $\mathcal H_1=\mathcal H_2=\mathcal H$ and one restricts the selector to the diagonal class
$\{(h,h):h\in\mathcal H\}$, then \Cref{thm:oracle-unequal} reduces, up to a change of constants, to
the main-text \Cref{thm:oracle}. Likewise, the diagonal specialization of
\Cref{cor:adaptive-unequal} below yields the upper-bound part of the main-text
\Cref{cor:adaptive}; the matching lower bound is supplied separately in Appendix~\ref{app:lower}.
\end{remark}

\begin{proof}[Proof of \Cref{thm:oracle-unequal}]
Let $\Omega$ be the event on which \eqref{eq:block-dev} holds simultaneously for all
$h\in\mathcal H_j$ and $j\in\{1,2\}$. By a union bound over the two geometric grids, after
enlarging the constants in $v_j$ if necessary, we may assume
$\Pr(\Omega)\ge 1-\delta$.

\medskip\noindent
Fix $j\in\{1,2\}$ and $h'\le h$ in $\mathcal H_j$. On $\Omega$, the triangle inequality and
\eqref{eq:block-dev} give
\begin{equation}
\mathrm{Dev}_j(h',h)\le E_j(h)+E_j(h')
\le b_j(h)+b_j(h')+v_j(h)+v_j(h').
\end{equation}
Since $b_j(\cdot)$ is increasing and $\kappa_j\ge 1$, it follows that
$\mathrm{Dev}_j(h',h)-\mathrm{pen}_j(h',h)\le b_j(h)+b_j(h')$,
and therefore
$B_j(h)\le \sup_{h'\le h}\{b_j(h)+b_j(h')\}=2b_j(h)$.

Now fix $u\in\mathcal H_j$. We claim that on $\Omega$,
\begin{equation}
\label{eq:Ehat-u}
E_j(\hat h_j)\le C\ \bigl(b_j(u)+v_j(u)\bigr),
\end{equation}
for a constant $C$ depending only on $\kappa_j,\kappa_j'$.

If $u\le \hat h_j$, then by the triangle inequality,
\begin{equation}
E_j(\hat h_j)\le \mathrm{Dev}_j(u,\hat h_j)+E_j(u)
\le B_j(\hat h_j)+\mathrm{pen}_j(u,\hat h_j)+E_j(u).
\end{equation}
Using $\kappa_j'\ge \kappa_j$, $B_j(\hat h_j)\ge 0$, and the selector optimality
$B_j(\hat h_j)+\mathrm{pen}_j(\hat h_j)\le B_j(u)+\mathrm{pen}_j(u)$,
we obtain
$E_j(\hat h_j)\le 2\bigl(B_j(u)+\mathrm{pen}_j(u)\bigr)+E_j(u)$.
If instead $u>\hat h_j$, then
\begin{equation}
E_j(\hat h_j)\le \mathrm{Dev}_j(\hat h_j,u)+E_j(u)
\le B_j(u)+\mathrm{pen}_j(\hat h_j,u)+E_j(u),
\end{equation}
and again $\kappa_j'\ge \kappa_j$ together with
$\mathrm{pen}_j(\hat h_j)\le B_j(\hat h_j)+\mathrm{pen}_j(\hat h_j)\le B_j(u)+\mathrm{pen}_j(u)$
gives
\begin{equation}
E_j(\hat h_j)\le 2\bigl(B_j(u)+\mathrm{pen}_j(u)\bigr)+E_j(u).
\end{equation}
In both cases, using $B_j(u)\le 2b_j(u)$ and \eqref{eq:block-dev},
\begin{equation}
E_j(\hat h_j)\le C\ \bigl(b_j(u)+v_j(u)\bigr).
\end{equation}
Since $u\in\mathcal H_j$ was arbitrary, we conclude that on $\Omega$,
\begin{equation}
\label{eq:GL-final-block}
E_j(\hat h_j)
\le
C\inf_{u\in\mathcal H_j}\left\{
u^\beta
+\sqrt{\frac{\Lambda_{\mathcal H}(u,\delta)}{M u^d}}
+\frac{\Lambda_{\mathcal H}(u,\delta)}{M u^d}
\right\},
\qquad j\in\{1,2\}.
\end{equation}

\medskip\noindent
We now set $\delta=M^{-2}$. Since the grids are geometric and range from
order $1$ down to order $M^{-1/d}$, we have
$|\mathcal H_1|+|\mathcal H_2|=O(\log M)$. Therefore, uniformly over
$h\in\mathcal H_1\cup\mathcal H_2$,
\begin{equation}
\Lambda_{\mathcal H}(h,M^{-2})
=
d\log\!\left(\frac{C_K}{h}\right)
+
\log\!\left(4M^2(|\mathcal H_1|+|\mathcal H_2|)\right)
\le
C\left\{
d\log\!\left(\frac{C_K}{h}\right)+\log M
\right\}
\end{equation}
for all sufficiently large $M$. Thus the grid-union factor is absorbed
into the displayed $\log M$ rate.

Because the grids span the oracle scale, there exist $\bar h_j\in\mathcal H_j$ such that
$\bar h_j\asymp (\log M/M)^{1/(2\beta+d)}$, for $j\in\{1,2\}$.
Hence
\begin{equation}
\bar h_j^\beta
+\sqrt{\frac{\Lambda_{\mathcal H}(\bar h_j,\delta)}{M \bar h_j^d}}
\asymp
\Big(\frac{\log M}{M}\Big)^{\beta/(2\beta+d)},
\qquad
\frac{\Lambda_{\mathcal H}(\bar h_j,\delta)}{M \bar h_j^d}
\lesssim
\Big(\frac{\log M}{M}\Big)^{2\beta/(2\beta+d)},
\end{equation}
so for all sufficiently large $M$,
\begin{equation}
\inf_{u\in\mathcal H_j}\left\{
u^\beta
+\sqrt{\frac{\Lambda_{\mathcal H}(u,\delta)}{M u^d}}
+\frac{\Lambda_{\mathcal H}(u,\delta)}{M u^d}
\right\}
\lesssim
\Big(\frac{\log M}{M}\Big)^{\beta/(2\beta+d)}
\to 0.
\end{equation}
By \eqref{eq:GL-final-block}, there exists $M_0\ge 1$ such that for all $M\ge M_0$, on $\Omega$,
\begin{equation}
E_1(\hat h_1)\le \frac12\min\!\left\{f_{\min},\ \frac{D_{\min}f_{\min}^2}{4(B_f+B_{g_1})}\right\},
\qquad
E_2(\hat h_2)\le \frac{f_{\min}}{2},
\end{equation}
where $B_f:=\sup_{\xi\in B_\rho}f(\xi)$ and
$B_{g_1}:=\sup_{(t,x,\xi)\in Q_{R,\eta,\rho}}|g_1(t,x;\xi)|$.
Therefore, for all $M\ge M_0$, on $\Omega$ we have
\begin{equation}
\inf_{\xi\in B_\rho}\hat f_1^{\ \hat h_1}(\xi)\ge \frac{f_{\min}}{2},
\qquad
\inf_{\xi\in B_\rho}\hat f_2^{\ \hat h_2}(\xi)\ge \frac{f_{\min}}{2},
\end{equation}
and, by \Cref{lem:denom-stability},
\begin{equation}
\inf_{(t,x,\xi)\in Q_{R,\eta,\rho}}\hat D_{\hat h_1,\hat h_2}(t,x;\xi)\ge \frac{D_{\min}}{2}.
\end{equation}
Thus, for all $M\ge M_0$, the denominator-floor event
$\mathcal E(\hat h_1,\hat h_2)$
holds on $\Omega$.

\medskip\noindent
For $M\ge M_0$, on $\Omega$ the denominator-floor convention is inactive, and
\Cref{prop:ratio-transfer}, together with the uniform boundedness of $f$, $g_1$, $g_2$, and $Q^*$
on $ Q_{R,\eta,\rho}$, yields
\begin{equation}
\sup_{(t,x,\xi)\in Q_{R,\eta,\rho}}
\big|\hat a_{\hat h_1,\hat h_2}(t,x;\xi)-a^*(t,x;\xi)\big|
\le
C\bigl(E_1(\hat h_1)+E_2(\hat h_2)\bigr).
\end{equation}
Using \eqref{eq:GL-final-block}, we conclude that on $\Omega$,
\begin{align}
\sup_{(t,x,\xi)\in Q_{R,\eta,\rho}}
\big|\hat a_{\hat h_1,\hat h_2}(t,x;\xi)-a^*(t,x;\xi)\big|
&\le
C_1\inf_{h_1\in\mathcal H_1}
\left\{
h_1^\beta+\sqrt{\frac{\Lambda_{\mathcal H}(h_1,\delta)}{M h_1^d}}+\frac{\Lambda_{\mathcal H}(h_1,\delta)}{M h_1^d}
\right\}\\
&\qquad\qquad+
C_2\inf_{h_2\in\mathcal H_2}
\left\{
h_2^\beta+\sqrt{\frac{\Lambda_{\mathcal H}(h_2,\delta)}{M h_2^d}}+\frac{\Lambda_{\mathcal H}(h_2,\delta)}{M h_2^d}
\right\}.
\end{align}

\medskip\noindent
With this choice, $\mathbb{P}(\Omega^c)\le cM^{-2}$.
On $\Omega$ we have the previous oracle bound. It remains to control the
contribution of $\Omega^c$ to the expected risk.

The denominator-floor convention sets the estimator equal to zero only on
the complement of $\mathcal E(\widehat h_1,\widehat h_2)$, not necessarily
on all of $\Omega^c$. Nevertheless, a deterministic envelope suffices. If
$\mathcal E(\widehat h_1,\widehat h_2)$ fails, then
$\widehat a_{\widehat h_1,\widehat h_2}=0$, and the error is bounded by
$\sup_{Q_{R,\eta,\rho}}|a^*|<\infty$. If
$\mathcal E(\widehat h_1,\widehat h_2)$ holds, then
\begin{equation}
\widehat f^{\,\widehat h_2}_2(\xi)\ge f_{\min}/2,
\qquad
\widehat D^{\widehat h_1,\widehat h_2}(t,x;\xi)\ge D_{\min}/2 .
\end{equation}
Since the grids satisfy $h_{j,K_j}\gtrsim M^{-1/d}$ and $K$, $F$, and
$y$ are bounded on the relevant support/query sets,
\begin{equation}
\sup_{Q_{R,\eta,\rho}}
|\widehat g^{\,\widehat h_2}_2(t,x;\xi)|
\le C\widehat h_2^{-d}
\le CM.
\end{equation}
Therefore, on $\mathcal E(\widehat h_1,\widehat h_2)$,
$\sup_{Q_{R,\eta,\rho}}
|\widehat N^{\,\widehat h_2}(t,x;\xi)|
=
\sup_{Q_{R,\eta,\rho}}
\left|
\frac{\widehat g^{\,\widehat h_2}_2(t,x;\xi)}
{\widehat f^{\,\widehat h_2}_2(\xi)}
\right|
\le CM$,
and the denominator floor for $\widehat D$ gives
\begin{equation}
\sup_{Q_{R,\eta,\rho}}
|\widehat a_{\widehat h_1,\widehat h_2}(t,x;\xi)|
\le C(1+M).
\end{equation}
Since $a^*$ is uniformly bounded on $Q_{R,\eta,\rho}$,
\begin{equation}
\mathbb{E}\!\left[
\sup_{Q_{R,\eta,\rho}}
|\widehat a_{\widehat h_1,\widehat h_2}-a^*|
\mathbf 1_{\Omega^c}
\right]
\le C(1+M)\mathbb{P}(\Omega^c)
\le CM^{-1}.
\end{equation}
Combining this with the oracle bound on $\Omega$ proves the result.
\end{proof}

\begin{corollary}[Unequal-bandwidth adaptive upper bound]
\label{cor:adaptive-unequal}
If $\mathcal H_j=\{h_{j,0}2^{-k}:k=0,1,\dots,K_j\}$ with $h_{j,0}\asymp 1$ and
$h_{j,K_j}\asymp M^{-1/d}$ for $j\in\{1,2\}$, then, for all sufficiently large $M$,
\begin{equation}
\sup_{\nu\in\mathfrak P(\beta,L_p,f_{\min},B)}
\mathcal R_\nu\bigl(\hat a_{\hat h_1,\hat h_2}\bigr)
\lesssim
\Big(\tfrac{\log M}{M}\Big)^{\beta/(2\beta+d)}.
\end{equation}
\end{corollary}

\begin{proof}
By \Cref{thm:oracle-unequal}, it suffices to optimize each blockwise term
\begin{equation}
h^\beta+\sqrt{\frac{\log M}{M h^d}}+\frac{\log M}{M h^d}
\end{equation}
over the geometric grid. Since the grid spans from order $1$ down to order $M^{-1/d}$, it contains a
bandwidth of order
$h_j^\star\asymp (\log M/M)^{1/(2\beta+d)}$, $j\in\{1,2\}$,
up to a multiplicative constant depending only on the grid ratio. At this scale, the linear term
$\log M/(M (h_j^\star)^d)$ is of strictly smaller order than
$(h_j^\star)^\beta\asymp \sqrt{\log M/(M (h_j^\star)^d)}$.
Therefore both infima in
\Cref{thm:oracle-unequal} are of order
$(\log M/M)^{\beta/(2\beta+d)}$.
Absorbing constants proves the claim.
\end{proof}

\section{Local and global minimax lower bounds}
\label{app:lower}

Here and throughout Appendix~\ref{app:lower}, we use the convention
$\ell_\beta:=\max\{m\in\mathbb N_0:m<\beta\}$ and
$\theta_\beta:=\beta-\ell_\beta\in(0,1]$. The notation
$|\cdot|_{C^\beta(B)}$ denotes the Hölder norm from
Appendix~\ref{Appendix A}, restricted to the set $B$.

In this Appendix we first prove a pivot-local lower-bound statement and then derive a global minimax lower bound by instantiating it at an explicit uniform pivot. The proof uses a standard two-point Le Cam construction
\cite{lecam1986asymptotic,yu1997assouad,tsybakov2009nonparametric} embedded into a shrinking neighborhood of the pivot, intersected with the
original global class.

\begin{proposition}[Local minimax lower bound around an interior pivot]
\label{prop:lower}
Fix
$(t^\circ,x^\circ,\xi^\circ)\in Q_{R,\eta,\rho}$.
Let $\nu_0\in\mathfrak P(\beta,L_p,f_{\min},B)$ be a pivot law with joint density $p_0$.
Assume that there exist constants $\delta_f,\delta_D,\delta_H>0$ such that
\begin{equation}
\inf_{\xi\in B}f_0(\xi)\ge f_{\min}+2\delta_f,\qquad
\inf_{(t,x,\xi)\in Q_{R,\eta,\rho}}D_0^*(t,x;\xi)\ge D_{\min}+2\delta_D,
\end{equation}
and $\sup_{y\in B}|\bar p_0(\cdot,y)|_{C^\beta(B)}\le L_p-2\delta_H$.

For $Q>0$ and $\varepsilon>0$, define
\begin{equation}
\mathcal U_M(\nu_0;Q,\varepsilon)
:=
\left\{
\nu:p_\nu=p_0+g,\ 
\sup_{y\in B}|g(\cdot,y)|_{C^\beta(B)}\le Q,\ 
|g|_{L^\infty(B\times B)}\le \varepsilon
\right\},
\end{equation}
and the local class inside the global model by
\begin{equation}
\mathfrak P_{\mathrm{loc},M}^{\cap}(\nu_0;Q,\varepsilon)
:=
\mathfrak P(\beta,L_p,f_{\min},B)\cap
\mathcal U_M(\nu_0;Q,\varepsilon).
\end{equation}

Then there exist constants $Q_0>0$, $c_{\mathrm{loc}}>0$, and $c>0$, depending on $\nu_0$ and the
fixed model parameters, such that for all sufficiently large $M$,
\begin{equation}
\inf_{\tilde a}\ \sup_{\nu\in\mathfrak P_{\mathrm{loc},M}^{\cap}\!\bigl(\nu_0;Q_0,c_{\mathrm{loc}}M^{-\beta/(2\beta+d)}\bigr)}\ \mathcal R_\nu(\tilde a)
\ \ge\
c\,M^{-\beta/(2\beta+d)}.
\end{equation}
In particular, combined with Corollary~\ref{cor:adaptive-unequal}, the selected estimator is locally minimax-optimal around $\nu_0$ for the uniform risk $\mathcal R_\nu$, up to logarithmic factors; the lower bound is witnessed by the fixed interior evaluation point $(t^\circ,x^\circ,\xi^\circ)$.
\end{proposition}

\begin{proof}
Write
$F^\circ(y):=F(t^\circ,\xi^\circ,x^\circ,y), \ \mu_0(dy):=p_0(\xi^\circ,y)\,dy$.
Since $f_0(\xi^\circ)>0$, the measure $\mu_0$ is a nonzero finite measure on $B$, absolutely
continuous with respect to Lebesgue measure.

\paragraph{1.}
Let
$S:=\operatorname{span}\{1,F^\circ\}\subset L^2(\mu_0)$.
For each $r\in\{1,\ldots,d\}$ define
$g_r(y):=y_rF^\circ(y)$.
We claim that there exists $r$ such that
\begin{equation}
g_r\notin S
\qquad\text{in }L^2(\mu_0).
\end{equation}
Suppose otherwise. Then for every $r$ there exist $a_r,b_r\in\R$ such that
$y_rF^\circ(y)=a_r+b_rF^\circ(y)
\ \mu_0$-a.e. on $B$.
Equivalently,
$h_r(y):=(y_r-b_r)F^\circ(y)-a_r=0
\  \mu_0$-a.e. on $B$.
Since $\mu_0$ is absolutely continuous and nonzero, $h_r$ vanishes on a set of positive Lebesgue
measure. The function $h_r$ is real analytic in $y$ because $F^\circ$ is real analytic. A nontrivial
real-analytic function on $\R^d$ cannot vanish on a set of positive Lebesgue measure, so $h_r$ must
vanish identically. This is impossible: if $a_r=0$, then $y_r\equiv b_r$, contradiction; if
$a_r\neq 0$, then
$F^\circ(y)=a_r/(y_r-b_r)$,
but the right-hand side is singular on the hyperplane $y_r=b_r$, whereas $F^\circ$ is smooth and
strictly positive. This contradiction proves the claim.

Fix such a coordinate $r$. Let $\Pi_S$ denote the orthogonal projection in $L^2(\mu_0)$ onto $S$,
and define
$\varphi_0:=g_r-\Pi_S g_r$.
Then $\varphi_0\neq 0$, $\varphi_0\in S^\perp$, and $\varphi_0\in L^\infty(B)$ because $g_r$ and
$F^\circ$ are bounded on the bounded set $B$. Set
$\varphi:=\varphi_0/(2|\varphi_0|_{L^\infty(B)})$.
Then $|\varphi|_{L^\infty(B)}\le 1/2$, and since $\varphi\in S^\perp$,
\begin{equation}
\label{eq:phi-constraints}
\int_B \varphi(y)\ p_0(\xi^\circ,y)\ dy=0,
\qquad
\int_B F^\circ(y)\ \varphi(y)\ p_0(\xi^\circ,y)\ dy=0.
\end{equation}
Moreover,
\begin{equation}
\label{eq:m-star}
\int_B y_r F^\circ(y)\ \varphi(y)\ p_0(\xi^\circ,y)\ dy
=
\frac{|\varphi_0|_{L^2(\mu_0)}^2}{2|\varphi_0|_{L^\infty(B)}}>0.
\end{equation}
Hence the vector
$m^\circ:=\int_B y\ F^\circ(y)\ \varphi(y)\ p_0(\xi^\circ,y)\ dy$
is nonzero.

\paragraph{2.}
Choose any nonzero $\psi\in C_c^\beta(\mathbb{R}^d)$ supported in
$B(0,1)$ with $\psi(0)=1$, and set
$C_\psi:=|\psi|_{C^\beta(\mathbb{R}^d)}<\infty $.
Since $\xi^\circ\in\operatorname{int}(B)$, there exists $h_0>0$ such that
$\operatorname{supp}(\psi_h)\subset B$ for all $0<h\le h_0$, where
$\psi_h(\xi):=\psi\!\left((\xi-\xi^\circ)/h\right)$.

Set
$h:=M^{-1/(2\beta+d)},
\ \alpha:=a_0 h^\beta$,
with $a_0>0$ to be chosen small later. For all sufficiently large $M$, we have $h\le h_0$. Define
$u_h(\xi,y):=\alpha\,\psi_h(\xi)\varphi(y)$,
and the unnormalized perturbations
$\tilde p_\pm(\xi,y):=p_0(\xi,y)\bigl(1\pm u_h(\xi,y)\bigr)$.
Since $|\varphi|_{L^\infty(B)}\le 1/2$, if $a_0$ is sufficiently small then
$|u_h(\xi,y)|\le \frac14
$ on $B\times B$,
so $\tilde p_\pm\ge \frac34\,p_0\ge 0$.

Let
$m_h:=\iint_{B\times B} p_0(\xi,y)\,u_h(\xi,y)\ d\xi\,dy$.
Because $p_0$ is bounded on $B\times B$, $\varphi$ is bounded, and $\psi_h$ is supported on a set of
volume of order $h^d$, we have
$|m_h|\lesssim \alpha h^d$.
Hence
$c_\pm:=\iint_{B\times B}\tilde p_\pm(\xi,y)\ d\xi\,dy = 1\pm m_h$
satisfy $c_\pm\to 1$. Define the normalized densities
\begin{equation}
p_\pm(\xi,y):=\frac{\tilde p_\pm(\xi,y)}{c_\pm}
=
p_0(\xi,y)\,\frac{1\pm u_h(\xi,y)}{1\pm m_h},
\end{equation}
and let $\nu_\pm$ be the corresponding laws.

\paragraph{3.}
Define
$g_\pm:=p_\pm-p_0$.
Then
$g_\pm(\xi,y) = p_0(\xi,y)\,\frac{\pm(u_h(\xi,y)-m_h)}{1\pm m_h}$,
and therefore
\begin{equation}
|g_\pm|_{L^\infty(B\times B)}
\lesssim |u_h|_{L^\infty(B\times B)}+|m_h|
\lesssim \alpha
\asymp
M^{-\beta/(2\beta+d)}.
\end{equation}

For the $C^\beta$ control in the first coordinate, write
$g_\pm(\cdot,y)
= \frac{p_0(\cdot,y)}{1\pm m_h}
\Bigl(\pm \alpha \varphi(y)\psi_h(\cdot)\mp m_h\Bigr)$.
Using the standard product estimate for H\"older spaces on bounded sets,
$|fg|_{C^\beta(B)}
\le C |f|_{C^\beta(B)} |g|_{C^\beta(B)}$,
and the scaling bound
$|\psi_h|_{C^\beta(B)}\le C C_\psi h^{-\beta}$,
we get, since $\alpha=a_0h^\beta$,
$\alpha |\psi_h|_{C^\beta(B)}
\le C a_0 C_\psi$.
The $m_h$ term contributes only $O(\alpha h^d)$ in $C^\beta(B)$.
Consequently,
$\sup_{y\in B}|g_\pm(\cdot,y)|_{C^\beta(B)}
\le C a_0 C_\psi$.
Choose $Q_0>0$ larger than this constant.

\paragraph{4.}
Because
$|f_\pm(\xi)-f_0(\xi)|
\le \int_B |g_\pm(\xi,y)|\,dy
\le \lambda(B)|g_\pm|_{L^\infty(B\times B)}
\lesssim \alpha$,
and the assumption of the pivot law
$\inf_{\xi\in B} f_0(\xi)\ge f_{\min}+2\delta_f$,
for all sufficiently large $M$ we obtain
\begin{equation}
\inf_{\xi\in B} f_\pm(\xi)\ge f_{\min}.
\end{equation}

Similarly, define
\begin{equation}
g_{1,\pm}(t,x;\xi):=\int_B F(t,\xi,x,y)\ p_\pm(\xi,y)\ dy,
\qquad
g_{1,0}(t,x;\xi):=\int_B F(t,\xi,x,y)\ p_0(\xi,y)\ dy.
\end{equation}
Because $F$ is continuous on the compact set
$[s,u-\eta]\times B_R\times B\times B$, there exists
\begin{equation}
C_{F,Q}:=\sup_{t\in[s,u-\eta],\,x\in B_R,\,\xi\in B,\,y\in B} F(t,\xi,x,y)<\infty.
\end{equation}
Hence
$|g_{1,\pm}(t,x;\xi)-g_{1,0}(t,x;\xi)|
\le C_{F,Q}\lambda(B)|g_\pm|_{L^\infty(B\times B)}
\lesssim \alpha$
uniformly on $ Q_{R,\eta,\rho}$.
Since $f_0$ and $f_\pm$ are bounded below away from zero for all sufficiently large $M$, the
deterministic ratio inequality gives
$\sup_{(t,x,\xi)\in Q_{R,\eta,\rho}}
|D_\pm^*(t,x;\xi)-D_0^*(t,x;\xi)|
\lesssim
\alpha$.
As
$\inf_{(t,x,\xi)\in Q_{R,\eta,\rho}} D_0^*(t,x;\xi)\ge D_{\min}+2\delta_D$,
for all sufficiently large $M$ we obtain
\begin{equation}
\inf_{(t,x,\xi)\in Q_{R,\eta,\rho}} D_\pm^*(t,x;\xi)\ge D_{\min}.
\end{equation}

Finally, since
$\sup_{y\in B}|p_0(\cdot,y)|_{C^\beta(B)}\le L_p-2\delta_H
\ \text{and}\ 
\sup_{y\in B}|g_\pm(\cdot,y)|_{C^\beta(B)}
\le C a_0 C_\psi$,
choosing $a_0$ sufficiently small yields
$\sup_{y\in B}|p_\pm(\cdot,y)|_{C^\beta(B)}
\le L_p-\delta_H<L_p $.
Thus $\nu_\pm\in \mathfrak P(\beta,L_p,f_{\min},B)$ for all sufficiently large $M$.

Consequently, for $c_{\mathrm{loc}}$ sufficiently large and all sufficiently large $M$,
\begin{equation}
\label{eq:two-point-local}
\{\nu_+,\nu_-\}\subset \mathfrak P_{\mathrm{loc},M}^{\cap}\!\bigl(\nu_0;Q_0,c_{\mathrm{loc}}M^{-\beta/(2\beta+d)}\bigr).
\end{equation}

\paragraph{5.}
At $(t^\circ,x^\circ,\xi^\circ)$, write
$g_{1,\pm}^\circ:=\int_B F^\circ(y)\ p_\pm(\xi^\circ,y)\ dy,
\ g_{2,\pm}^\circ:=\int_B y\ F^\circ(y)\ p_\pm(\xi^\circ,y)\ dy$.
Since $\psi_h(\xi^\circ)=1$, using \eqref{eq:phi-constraints} and \eqref{eq:m-star} we obtain
\begin{equation}
g_{1,\pm}^\circ
=
\frac{1}{1\pm m_h}\int_B F^\circ(y)\ p_0(\xi^\circ,y)\bigl(1\pm \alpha\varphi(y)\bigr)\ dy
=
\frac{g_{1,0}^\circ}{1\pm m_h},
\end{equation}
and
\begin{equation}
g_{2,\pm}^\circ
=
\frac{1}{1\pm m_h}\int_B y\ F^\circ(y)\ p_0(\xi^\circ,y)\bigl(1\pm \alpha\varphi(y)\bigr)\ dy
=
\frac{g_{2,0}^\circ\pm \alpha m^\circ}{1\pm m_h},
\end{equation}
where
$g_{1,0}^\circ:=\int_B F^\circ(y)\ p_0(\xi^\circ,y)\ dy,
\ g_{2,0}^\circ:=\int_B y\ F^\circ(y)\ p_0(\xi^\circ,y)\ dy$.
Therefore the common normalization cancels:
\begin{equation}
Q_\pm^\circ
:=
\frac{g_{2,\pm}^\circ}{g_{1,\pm}^\circ}
=
\frac{g_{2,0}^\circ\pm \alpha m^\circ}{g_{1,0}^\circ}.
\quad \text{Hence} \quad
Q_+^\circ-Q_-^\circ
= \frac{2\alpha\ m^\circ}{g_{1,0}^\circ},
\end{equation}
and therefore
\begin{equation}
\label{eq:drift-sep}
\big|a_+^*(t^\circ,x^\circ;\xi^\circ)-a_-^*(t^\circ,x^\circ;\xi^\circ)\big|
=
\frac{1}{\Delta(t^\circ)}\ |Q_+^\circ-Q_-^\circ|
\ge c_1\ \alpha
=
c_1 a_0 h^\beta,
\end{equation}
where $c_1>0$ depends only on the fixed model parameters, since $\Delta(t^\circ)\ge \eta$ and
\begin{equation}
g_{1,0}^\circ=f_0(\xi^\circ)D_0^*(t^\circ,x^\circ;\xi^\circ)>0.
\end{equation}

\paragraph{6.}
Using the bound
$\mathrm{KL}(\nu_+\mid\nu_-)\le \chi^2(\nu_+,\nu_-)
:= \iint_{B\times B}\frac{(p_+-p_-)^2}{p_-}$,
we estimate the one-sample divergence. Since
\begin{equation}
p_+-p_-
=
p_0\left(\frac{1+u_h}{1+m_h}-\frac{1-u_h}{1-m_h}\right)
=
\frac{2p_0(u_h-m_h)}{1-m_h^2},
\end{equation}
and since $|u_h|\le 1/4$ and $|m_h|\lesssim \alpha h^d$ for all sufficiently large $M$, we have
\begin{equation}
p_-(\xi,y)
=
p_0(\xi,y)\frac{1-u_h(\xi,y)}{1-m_h}
\ge c\,p_0(\xi,y)
\end{equation}
for some absolute constant $c>0$. Therefore
$\chi^2(\nu_+,\nu_-)
\le C\iint_{B\times B} p_0(\xi,y)\,(u_h(\xi,y)-m_h)^2\ d\xi dy$.
Since $m_h=\iint p_0u_h$ and $\iint p_0=1$, Jensen's inequality gives
$m_h^2\le \iint p_0u_h^2$,
and hence
\begin{equation}
\iint p_0(u_h-m_h)^2
\le
2\iint p_0u_h^2 + 2m_h^2
\lesssim
\iint p_0u_h^2.
\end{equation}
Now $p_0$ is bounded on $B\times B$, $\varphi$ is bounded, and $\psi_h$ is supported on a set of
$\xi$-volume of order $h^d$, so
\begin{equation}
\iint_{B\times B} p_0(\xi,y)u_h(\xi,y)^2\ d\xi dy
\lesssim
\alpha^2 h^d.
\quad\text{Thus}\quad
\mathrm{KL}(\nu_+\mid\nu_-)\lesssim \alpha^2 h^d.
\end{equation}
Hence for $M$ i.i.d.\ observations,
\begin{equation}
\label{eq:KL-total}
\mathrm{KL}(\nu_+^{\otimes M}|\nu_-^{\otimes M})
\lesssim M\ \alpha^2 h^d
= M\ a_0^2\ h^{d+2\beta}.
\end{equation}

With $h=M^{-1/(2\beta+d)}$, the right-hand side of \eqref{eq:KL-total} is bounded by a universal
constant if $a_0$ is chosen small enough. By Le Cam’s two-point lemma,
\begin{equation}
\inf_{\tilde a}\ \sup_{\nu\in\{\nu_+,\nu_-\}}
\E_\nu\!\left[
\big|\tilde a(t^\circ,x^\circ;\xi^\circ)-a_\nu^*(t^\circ,x^\circ;\xi^\circ)\big|
\right]
\ge
c_2\ \big|a_+^*(t^\circ,x^\circ;\xi^\circ)-a_-^*(t^\circ,x^\circ;\xi^\circ)\big|.
\end{equation}
Since the risk $\mathcal R_\nu$ dominates the pointwise risk at $(t^\circ,x^\circ,\xi^\circ)$,
\begin{equation}
\mathcal R_\nu(\tilde a)
\ge
\E_\nu\!\left[
\big|\tilde a(t^\circ,x^\circ;\xi^\circ)-a_\nu^*(t^\circ,x^\circ;\xi^\circ)\big|
\right],
\end{equation}
and since \eqref{eq:two-point-local} holds, we conclude from \eqref{eq:drift-sep} that
\begin{equation}
\inf_{\tilde a}\ \sup_{\nu\in\mathfrak P_{\mathrm{loc},M}^{\cap}\!\bigl(\nu_0;Q_0,c_{\mathrm{loc}}M^{-\beta/(2\beta+d)}\bigr)}\ \mathcal R_\nu(\tilde a)
\ge
\inf_{\tilde a}\ \sup_{\nu\in\{\nu_+,\nu_-\}}\ \mathcal R_\nu(\tilde a)
\gtrsim
h^\beta
\asymp
M^{-\beta/(2\beta+d)},
\end{equation}
which proves the claim.
\end{proof}

\Cref{prop:lower} is local around a pivot law satisfying the stated slack conditions. The next corollary shows that, under explicit compatibility conditions, a uniform pivot law satisfies those conditions and therefore upgrades the local lower bound to a global minimax lower bound on $\mathfrak P(\beta,L_p,f_{\min},B)$.

Write $|B|:=\lambda(B)$. The compatibility conditions $f_{\min}<|B|^{-1}$ and $L_p>|B|^{-2}$ are strict feasibility
conditions ensuring that the uniform law lies in the interior of the marginal-density and smoothness
constraints of the global class. In particular, $f_{\min}<|B|^{-1}$ is necessary for the existence of
any density on $B$ bounded below by $f_{\min}$, since
$1=\int_B f\ge f_{\min}|B|$. The denominator slack is automatic from the deterministic lower bound
following \Cref{ass:pair}.

\begin{corollary}[Global minimax lower bound via an explicit uniform pivot]
\label{cor:global-minimax-explicit-pivot}
Let $R_B:=\sup_{z\in B}|z|<\infty$ and let
$\mathcal P:=\mathfrak P(\beta,L_p,f_{\min},B)$ be the global model class from
\Cref{ass:pair}. Assume further that $f_{\min}<|B|^{-1}$ and $L_p>|B|^{-2}$.
Then there exists a constant $c>0$, depending only on the fixed model parameters,
such that for all sufficiently large $M$,
\begin{equation}
\inf_{\widetilde a}\sup_{\nu\in\mathcal P} \mathcal R_\nu(\widetilde a)
\ge
c\,M^{-\beta/(2\beta+d)}.
\end{equation}
Consequently, together with \Cref{thm:oracle},
\begin{equation}
\sup_{\nu\in\mathcal P} \mathcal R_\nu(\hat a_{\widehat h})
\lesssim
\left(\frac{\log M}{M}\right)^{\beta/(2\beta+d)},
\end{equation}
so the adaptive estimator is globally minimax-rate optimal over $\mathcal P$ up to logarithmic factors.
\end{corollary}

\begin{proof}
Fix any $(t^\circ,x^\circ,\xi^\circ)\in Q_{R,\eta,\rho}$. Define the uniform pivot law $\nu_0$ on
$B\times B$ by the joint density
\begin{equation}
p_0(\xi,y)=\frac{1}{|B|^2}\mathbf 1_B(\xi)\mathbf 1_B(y).
\end{equation}
Then its $X_s$-marginal density is
$f_0(\xi)=\int_B p_0(\xi,y)\,dy=\frac1{|B|},
\  \xi\in B$.
Since $f_{\min}<|B|^{-1}$, we may choose
$\delta_f:=\frac14\bigl(|B|^{-1}-f_{\min}\bigr)>0$,
so that
$\inf_{\xi\in B} f_0(\xi)=|B|^{-1}\ge f_{\min}+2\delta_f$.

Next, for each $y\in B$, the map $\xi\mapsto p_0(\xi,y)$ is constant on $B$. A valid extension to a neighborhood of $B$ is
$\bar p_0(z,y)=\frac1{|B|^2}$.
With the $C^\beta(B)$ norm of \Cref{prop:lower},
$\sup_{y\in B} |\bar p_0(\cdot,y)|_{C^\beta(B)} = |B|^{-2}$.

Since $L_p>|B|^{-2}$, we choose
$\delta_H:=\frac14\bigl(L_p-|B|^{-2}\bigr)>0$, then
$\sup_{y\in B} |\bar p_0(\cdot,y)|_{C^\beta(B)}
\le L_p-2\delta_H$.

Under the uniform pivot,
\begin{equation}
D_0^*(t,x;\xi)
=
\mathbb E_{\nu_0}\!\left[F(t,\xi,x,X_u)\mid X_s=\xi\right]
=
\frac1{|B|}\int_B F(t,\xi,x,y)\,dy.
\end{equation}
Define $D_{\mathrm{unif}}:=\inf_{(t,x,\xi)\in Q_{R,\eta,\rho}}D_0^*(t,x;\xi)$. Moreover, by the automatic denominator bound following \Cref{ass:pair},
$D_{\mathrm{unif}}\ge 2D_{\min}$, so the uniform pivot satisfies the required denominator slack.

It remains to note that $D_{\mathrm{unif}}>0$. Indeed, on $Q_{R,\eta,\rho}$ we have $\Delta(t)\ge \eta$, and for $x\in B_R$, $y\in B$,
$|y-x|\le |y|+|x|\le R_B+R$.
Since the expanding term in the exponent is nonnegative,
\begin{equation}
F(t,\xi,x,y)
=
\exp\!\left(
-\frac{|y-x|^2}{2\Delta(t)}
+
\frac{|y-\xi|^2}{2\Delta}
\right)
\ge
\exp\!\left(
-\frac{(R_B+R)^2}{2\eta}
\right).
\end{equation}
Hence
\begin{equation}
D_{\mathrm{unif}}
\ge \frac1{|B|}\int_B
\exp\!\left( -\frac{(R_B+R)^2}{2\eta} \right)\,dy
= \exp\!\left( -\frac{(R_B+R)^2}{2\eta}
\right) >0.
\end{equation}

Thus the explicit uniform pivot $\nu_0$ satisfies the interiority/slack conditions of \Cref{prop:lower}. Therefore there exist constants $Q_0>0$, $c_{\mathrm{loc}}>0$, and $c>0$ such that, for all sufficiently large $M$,
\begin{equation}
\inf_{\widetilde a} \sup_{\nu\in \mathfrak P_{\mathrm{loc},M}^{\cap}
(\nu_0;Q_0,c_{\mathrm{loc}}M^{-\beta/(2\beta+d)})} \mathcal R_\nu(\widetilde a)
\ge c\,M^{-\beta/(2\beta+d)}.
\end{equation}
By construction,
\begin{equation}
\mathfrak P_{\mathrm{loc},M}^{\cap}
(\nu_0;Q_0,c_{\mathrm{loc}}M^{-\beta/(2\beta+d)})
\subset \mathcal P.
\end{equation}
Hence
\begin{equation}
\inf_{\widetilde a}\sup_{\nu\in\mathcal P} \mathcal R_\nu(\widetilde a)
\ge
\inf_{\widetilde a}
\sup_{\nu\in
\mathfrak P_{\mathrm{loc},M}^{\cap}
(\nu_0;Q_0,c_{\mathrm{loc}}M^{-\beta/(2\beta+d)})}
\mathcal R_\nu(\widetilde a)
\ge
c\,M^{-\beta/(2\beta+d)}.
\end{equation}
This proves the global minimax lower bound.
\end{proof}

\section{Experimental Details}
\label{app:exp-details}

This Appendix records the synthetic models, truth computation, bandwidth rules, diagnostics, and reproducibility details used in \Cref{sec:experiments}. The experiments are theorem-facing: they probe the finite-sample scaling law, the pointwise Gaussian approximation, the terminal singularity, and the finite-sample oracle-inequality behavior of a practical selector. For adaptivity we follow the GL-oracle-inequality viewpoint, so the relevant empirical quantity is a non-unit oracle ratio rather than a target close to one \cite{goldenshluger2011,lacour2016minimal}.

\paragraph{Benchmark scope.}
We do not re-benchmark SBTS as a time-series generator. The original SBTS paper of \cite{hamdouche2023sbts} and the
comprehensive evaluation of \cite{alouadi2025robust} already study generator-level performance against
standard baselines and datasets. Our experiments instead use synthetic laws with cached
ground-truth drifts in order to isolate statistical drift-estimation error from optimization,
simulation, and potential-to-drift recovery errors.

\subsection{Synthetic bridge families}
\label{app:exp-models}

All experiments are run on the single interval $[s,u]=[0.2,1.0]$, so $\Delta=u-s=0.8$, with support box $B=[-3,3]^d$. We use $\mathcal X_{\rm eval}=[-2,2]$ for $d=1$ and $\mathcal X_{\rm eval}=[-1.5,1.5]^2$ for $d=2$; the latter was frozen after a pre-flight denominator-stability check.

We consider two bounded-support synthetic pair laws. In the \emph{Gaussian-to-Gaussian} family, $X_s\sim \TruncNormal(m_s,\Sigma_s;B)$ and $X_u\mid X_s=\xi\sim \TruncNormal(A\xi+b,\Sigma_\varepsilon;B)$. In the \emph{Mixture-to-Mixture} family, $X_s\sim \frac12\TruncNormal(m_1,S;B)+\frac12\TruncNormal(m_2,S;B)$ and
\begin{equation}
X_u\mid X_s=\xi\sim \pi(\xi)\TruncNormal(A_1\xi+b_1,\Sigma_1;B)+(1-\pi(\xi))\TruncNormal(A_2\xi+b_2,\Sigma_2;B),
\end{equation}
with logistic gate $\pi(\xi)=\sigma(\alpha_0+\alpha^\top \xi)$.

The default parameters are as follows. For GG1 we use $m_s=0$, $\Sigma_s=(1.0)^2$, $A=0.7$, $b=0.3$, and $\Sigma_\varepsilon=(0.35)^2$. For GG2 we use $m_s=(0,0)$, $\Sigma_s=\begin{psmallmatrix}1.0&0\\0&0.8\end{psmallmatrix}$, $A=\begin{psmallmatrix}0.75&0.15\\-0.10&0.65\end{psmallmatrix}$, $b=\begin{psmallmatrix}0.25\\-0.20\end{psmallmatrix}$, and $\Sigma_\varepsilon=\begin{psmallmatrix}0.14&0.03\\0.03&0.12\end{psmallmatrix}$. For MM1 we use $m_1=-1.2$, $m_2=1.2$, $S=(0.45)^2$, $A_1=0.8$, $b_1=0.4$, $\Sigma_1=(0.25)^2$, $A_2=-0.5$, $b_2=-0.3$, $\Sigma_2=(0.30)^2$, and $\pi(\xi)=\sigma(1.5\xi)$. For MM2 we use $m_1=(-0.9,0.9)$, $m_2=(0.9,-0.9)$, $S=\begin{psmallmatrix}0.16&0\\0&0.16\end{psmallmatrix}$, $A_1=\begin{psmallmatrix}0.8&0.1\\0&0.7\end{psmallmatrix}$, $b_1=\begin{psmallmatrix}0.3\\-0.2\end{psmallmatrix}$, $\Sigma_1=\begin{psmallmatrix}0.0484&0\\0&0.0324\end{psmallmatrix}$, $A_2=\begin{psmallmatrix}-0.4&0.2\\0.15&-0.6\end{psmallmatrix}$, $b_2=\begin{psmallmatrix}-0.35\\0.25\end{psmallmatrix}$, $\Sigma_2=\begin{psmallmatrix}0.08&0.02\\0.02&0.07\end{psmallmatrix}$, and $\pi(\xi)=\sigma(1.2\xi_1-1.0\xi_2)$.

For the rate and adaptivity experiments we use the fixed interior queries GG1: $(t_0,\xi_0)=(0.6,0)$, GG2: $(0.6,(0,0))$, MM1: $(0.6,0.8)$, and MM2: $(0.6,(0.8,-0.8))$. For the one-dimensional CLT experiment we use GG1: $(t_0,x_0,\xi_0)=(0.6,0.2,0)$ and MM1: $(0.6,0.3,0.8)$. The final four-point rate/adaptivity summaries use $R_M=50$ for GG1 and MM1 and $R_M=20$ for GG2 and MM2. The denser adaptivity-figure reruns use the same repetition counts on the seven-point grid $M\in\{10^3,1500,2\times10^3,3000,4\times10^3,6000,8\times10^3\}$. The final one-dimensional CLT runs use $R_M=300$, while the undersmoothing screen uses $R_M=150$ at $M\in\{4\times10^3,8\times10^3\}$.

\subsection{Truth computation and pre-flight checks}
\label{app:exp-truth}

The target drift is computed numerically, not by Monte Carlo. For fixed $(t,x,\xi)$ we evaluate
\begin{equation}
D^\ast(t,x;\xi)=\int_B F(t,\xi,x,y)\ \frac{p(\xi,y)}{f(\xi)}\ dy,\qquad
N^\ast(t,x;\xi)=\int_B y\ F(t,\xi,x,y)\ \frac{p(\xi,y)}{f(\xi)}\ dy,
\end{equation}
and then
\begin{equation}
a^\ast(t,x;\xi)=\Delta(t)^{-1}\left(\frac{N^\ast(t,x;\xi)}{D^\ast(t,x;\xi)}-x\right).
\end{equation}
In $d=1$ we use adaptive Gauss--Legendre quadrature; in $d=2$ we use a dense tensor-product grid. The truth cache is accepted only after a one-step refinement check showing negligible change at the recorded query points.

Before the main runs we verify constrained sampling, positivity of $f(\xi_0)$ at the chosen interior points, numerical positivity of $\inf_{(t,x,\xi)\in\mathcal Q_{\mathrm{eval}}}D^\ast(t,x;\xi)$ on the actual evaluation domain, and negligible truth-grid refinement error. The pre-flight report is used only to freeze the configurations and evaluation domains.

\begin{table}[t]
    \centering
    \caption{Pre-flight numerical diagnostics: source density at the chosen conditioning point, minimum source density on the evaluation conditioning grid, minimum population denominator on the evaluation domain, and truth-grid convergence error.}
    \label{tab:preflight}
    \begin{tabular}{lcccc}
        \toprule
        Testbed & $f(\xi_0)$ & $\min_{\xi\in \Xi_{\mathrm{grid}}} f(\xi)$ & $\min D^\ast$ & truth error \\
        \midrule
        GG1 & $4.000223\times 10^{-1}$ & $5.413713\times 10^{-2}$ & $6.140077\times 10^{-3}$ & $0$ \\
        GG2 & $1.785645\times 10^{-1}$ & $1.420651\times 10^{-2}$ & $2.285052\times 10^{-3}$ & $1.110223\times 10^{-16}$ \\
        MM1 & $2.986354\times 10^{-1}$ & $2.532522\times 10^{-2}$ & $2.997735\times 10^{-1}$ & $0$ \\
        MM2 & $4.672258\times 10^{-1}$ & $4.918337\times 10^{-9}$ & $8.261432\times 10^{-3}$ & $2.220446\times 10^{-16}$ \\
        \bottomrule
    \end{tabular}
\end{table}

\subsection{Grids, losses, and finite-range slope benchmark}
\label{app:exp-metrics}

For $d=1$ the evaluation grid uses $|\mathcal G_x|=200$ points on $[-2,2]$; for $d=2$ the final runs use a $21\times21$ grid on $[-1.5,1.5]^2$. The primary loss is
\begin{equation}
E_{\infty,\mathcal G}(t_0,\xi_0;h):=\max_{x\in\mathcal G_x}|\hat a_h(t_0,x;\xi_0)-a^\ast(t_0,x;\xi_0)|,
\end{equation}
with repetition average
\begin{equation}
\overline E_{\infty,\mathcal G}(M;h):=R_M^{-1}\sum_{r=1}^{R_M}E_{\infty,\mathcal G}^{(r)}(t_0,\xi_0;h).
\end{equation}
As a secondary descriptive metric we also record
\begin{equation}
\ISE(t_0,\xi_0;h):=\left(\int_{\mathcal X_{\rm eval}}|\hat a_h(t_0,x;\xi_0)-a^\ast(t_0,x;\xi_0)|^2\ dx\right)^{1/2}.
\end{equation}

The finite-sample theorem predicts $(\log M/M)^{\beta/(2\beta+d)}$. With the product Epanechnikov kernel we use the effective benchmark $\beta=2$, so $p(d)=2/(4+d)$ and the finite-range theoretical secant slope on $[M_1,M_2]$ is
\begin{equation}
s_{\mathrm{theory,sec}}(M_1,M_2;d)
=
-p(d)+p(d)\frac{\log(\log M_2/\log M_1)}{\log(M_2/M_1)}.
\end{equation}

For adaptivity we define
\begin{equation}
\Gamma_{\nu,M}^{(r)}:=\frac{E_{\infty,\mathcal G}^{(r)}(t_0,\xi_0;\hat h)}{E_{\infty,\mathcal G}^{(r)}(t_0,\xi_0;h_{\rm or})}\ge1,
\end{equation}
where $h_{\rm or}\in\arg\min_{h\in\mathcal H_M}E_{\infty,\mathcal G}(t_0,\xi_0;h)$. We then report
\begin{equation}
\bar\Gamma_\nu(M):=R_M^{-1}\sum_{r=1}^{R_M}\Gamma_{\nu,M}^{(r)},\qquad
\widehat C_\nu^{\max}:=\max_{M\in\mathcal M}\bar\Gamma_\nu(M),\qquad
\widehat C_\nu^{\avg}:=|\mathcal M|^{-1}\sum_{M\in\mathcal M}\bar\Gamma_\nu(M),
\end{equation}
with $\mathcal M=\{10^3,2\times10^3,4\times10^3,8\times10^3\}$. This is the empirical oracle-inequality diagnostic. The benchmark is not $\Gamma\approx1$: GL-type selectors naturally come with non-unit oracle constants, and in the classical kernel setting Goldenshluger--Lepski obtain the factor $1+3|K|_1=4$ for a nonnegative unit-mass kernel \cite{goldenshluger2011}$;$ minimal-penalty results further show that under-penalization may trigger instability \cite{lacour2016minimal}.

\subsection{Bandwidth grid, oracle rule, and practical selector}
\label{app:exp-bandwidths}

All rate and adaptive experiments are run on a geometric grid $\mathcal H_M=\{h_0q^k:k=0,\dots,K_M\}$ with $q=2^{-1/2}$ and $h_0=1.2$. The oracle bandwidth is
\begin{equation}
h_{\rm or}\in\arg\min_{h\in\mathcal H_M}E_{\infty,\mathcal G}(t_0,\xi_0;h),
\end{equation}
with ties broken in favor of the larger bandwidth.

The practical selector used for all quantitative claims in the paper is the raw-max, one-sided GL surrogate
\begin{equation}
|\hat a_{h'}-\hat a_h|_{\infty,\mathcal G}:=\max_{x\in\mathcal G_x}|\hat a_{h'}(t_0,x;\xi_0)-\hat a_h(t_0,x;\xi_0)|,
\end{equation}
\begin{equation}
B(h):=\sup_{h'\le h}\left\{|\hat a_{h'}-\hat a_h|_{\infty,\mathcal G}-\kappa_{\mathrm{pair}}\sqrt{\frac{\log M}{M(h')^d}}\right\}_+,\qquad
\hat h:=\arg\min_{h\in\mathcal H_M}\left\{B(h)+\kappa_{\mathrm{final}}\sqrt{\frac{\log M}{Mh^d}}\right\},
\end{equation}
with $\kappa_{\mathrm{pair}}=\kappa_{\mathrm{final}}=2$.

\paragraph{Why this selector.}
Write
\begin{equation}
\hat a_h:=\hat a_{h,h},
\qquad
v_0(h):=\sqrt{\frac{\log M}{Mh^d}},
\qquad
b(h):=Ah^\beta,
\end{equation}
and
$E(h):=|\hat a_h-a^*|_{\infty,\mathcal G}$.
The implemented diagonal selector uses
\begin{equation}
B(h)
=
\sup_{h'\le h}
\left\{
|\hat a_{h'}-\hat a_h|_{\infty,\mathcal G}
-
2v_0(h')
\right\}_+ .
\end{equation}
On the good event, suppose
$E(h)\le b(h)+v_0(h)$.
For $h'\le h$,
\begin{equation}
v_0(h')\ge v_0(h),
\qquad
b(h')\le b(h),
\end{equation}
and therefore
$|\hat a_{h'}-\hat a_h|_{\infty,\mathcal G}
\le
E(h')+E(h)
\le
b(h')+b(h)+v_0(h')+v_0(h)$.

Since $2v_0(h')\ge v_0(h')+v_0(h)$,
\begin{equation}
|\hat a_{h'}-\hat a_h|_{\infty,\mathcal G}
-
2v_0(h')
\le
b(h')+b(h)
\le
2b(h).
\end{equation}
Hence
$B(h)\le 2b(h)$.
By the minimizing property of $\hat h$, for every $u\in\mathcal H_M$,
\begin{equation}
B(\hat h)+2v_0(\hat h)
\le
B(u)+2v_0(u).
\end{equation}
Using the standard one-sided GL comparison,
\begin{equation}
E(\hat h)
\le
2\{B(u)+2v_0(u)\}+E(u).
\end{equation}
Consequently,
\begin{equation}
E(\hat h)
\le
2\{2b(u)+2v_0(u)\}+b(u)+v_0(u)
=
5\{b(u)+v_0(u)\}.
\end{equation}
Taking the infimum over $u\in\mathcal H_M$, the practical selector tracks
the diagonal oracle objective up to the multiplicative constant $5$.
The appendix oracle inequality uses the full proof-side penalty with
constants, grid-union logarithms, and lower-order linear terms; the
experiments use this leading-order diagonal proxy.

\paragraph{Why the lower bandwidth floor.}
The crucial experimental issue is not the selector form itself but the small-$h$ tail of the admissible grid. The stochastic scale entering both the selector penalty and the local denominator estimation is of order $(M h^d)^{-1/2}$ up to logarithms. If $h$ is too small, then $M h^d$ is too small, so pairwise discrepancies become excessively noisy and the ratio estimator becomes more fragile because its local denominator is supported by too little effective mass. To prevent the selector from entering this regime, we truncate the grid below at
\begin{equation}
h_{\min}(M,d)=c_{\min}M^{-1/d},
\end{equation}
with
\begin{equation}
c_{\min}=81\quad\text{in }d=1,\qquad c_{\min}=9\quad\text{in }d=2.
\end{equation}
These choices are dimensionally matched because in both cases
\begin{equation}
M h_{\min}(M,d)^d=c_{\min}^d=81.
\end{equation}
Hence the entire adaptive search is restricted to
\begin{equation}
M h^d\ge81.
\end{equation}
This gives a uniform finite-sample lower bound on the effective local sample size across dimensions. In exploratory runs, allowing smaller bandwidths was the dominant source of instability in the hard MM cases; after imposing $M h^d\ge81$, the selector became stable on all four testbeds.

For diagnostics we track the average selected oracle and adaptive bandwidths, the empirical oracle-ratio summaries, and the frequency with which the selected bandwidth lies on the search boundary. In the table below the selected-bandwidth averages are first taken over repetitions at each $M$ and then averaged over $M\in\mathcal M$.

\begin{table}[t]
    \centering
    \caption{Bandwidth and oracle-ratio diagnostics for the final adaptive protocol. Here $\mathbb E_M[\mathbb E(h_{\rm or}\mid M)]$ and $\mathbb E_M[\mathbb E(\hat h\mid M)]$ are the $M$-averaged oracle and adaptive selected bandwidths, \texttt{BoundaryAvg} is the mean boundary-hit rate, and $\widehat C_\nu^{\avg},\widehat C_\nu^{\max}$ are the empirical oracle-ratio summaries.}
    \label{tab:bandwidth-diagnostics}
    \begin{tabular}{lccccc}
        \toprule
        Testbed & $\mathbb E_M[\mathbb E(h_{\rm or}\mid M)]$ & $\mathbb E_M[\mathbb E(\hat h\mid M)]$ & BoundaryAvg & $\widehat C_\nu^{\avg}$ & $\widehat C_\nu^{\max}$ \\
        \midrule
        GG1 & $0.242760$ & $0.381922$ & $0.0000$ & $1.773613$ & $2.109796$ \\
        GG2 & $0.376363$ & $0.383079$ & $0.0000$ & $1.314465$ & $1.431312$ \\
        MM1 & $0.238186$ & $0.335428$ & $0.0100$ & $1.960175$ & $2.175424$ \\
        MM2 & $0.398902$ & $0.435579$ & $0.0125$ & $1.485740$ & $1.686308$ \\
        \bottomrule
    \end{tabular}
\end{table}

\subsection{Rate and adaptivity protocol}
\label{app:exp-rate}

For each testbed and each $M\in\mathcal M$, we simulate $R_M$ independent datasets and compute $E_{\infty,\mathcal G}^{(r)}(t_0,\xi_0;h)$ for every $h\in\mathcal H_M$. The oracle-selected bandwidth minimizes this quantity; the adaptive bandwidth is obtained from the canonical selector above. The fitted rate slope is then obtained by ordinary least squares on $\log \overline E_{\infty,\mathcal G}(M)$ versus $\log M$. We also report the slope of the secondary ISE curve as a descriptive diagnostic.

The main adaptivity summaries are $\bar\Gamma_\nu(M)$, $\widehat C_\nu^{\max}$, and $\widehat C_\nu^{\avg}$. Since the theorem is an oracle inequality rather than an oracle-equality statement, these are the right practical diagnostics. The protocol has two distinct theorem-facing ingredients: the selector form is chosen because it is the closest direct implementation of the proof-side GL construction, while the lower truncation $M h^d\ge81$ is imposed because the experimental instability came from the too-small-$h$ regime. Once this stabilized admissibility rule is imposed, the canonical raw-max one-sided selector is oracle-competitive on all four testbeds.

The oracle-slope summaries are already reported in \Cref{tab:slopes} while the four-point empirical oracle-ratio summaries are reported in \Cref{tab:bandwidth-diagnostics}.

\begin{table}[t]
    \centering
    \caption{Secondary rate results based on the integrated squared error. These are reported only as descriptive diagnostics.}
    \label{tab:rate-ise}
    \begin{tabular}{lccc}
        \toprule
        Testbed & $d$ & Finite-range theory & Oracle-$h_{\rm or}$ \\
        \midrule
        GG & 1 & $-0.349379$ & $-0.337109$ \\
        GG & 2 & $-0.291150$ & $-0.286309$ \\
        MM & 1 & $-0.349379$ & $-0.383080$ \\
        MM & 2 & $-0.291150$ & $-0.378696$ \\
        \bottomrule
    \end{tabular}
\end{table}

\paragraph{Primary sup-grid plots.}
The primary sup-grid rate plots are shown in the main text in \Cref{fig:rates}. To avoid duplication, we do not repeat them here. This Appendix instead retains the numerical slope \Cref{tab:slopes}, the oracle-ratio summary and the bandwidth diagnostics \Cref{tab:bandwidth-diagnostics}, and the secondary ISE plots in \Cref{fig:rate-ise-appendix}.

\begin{figure}[t]
    \centering
    \includegraphics[width=.48\linewidth]{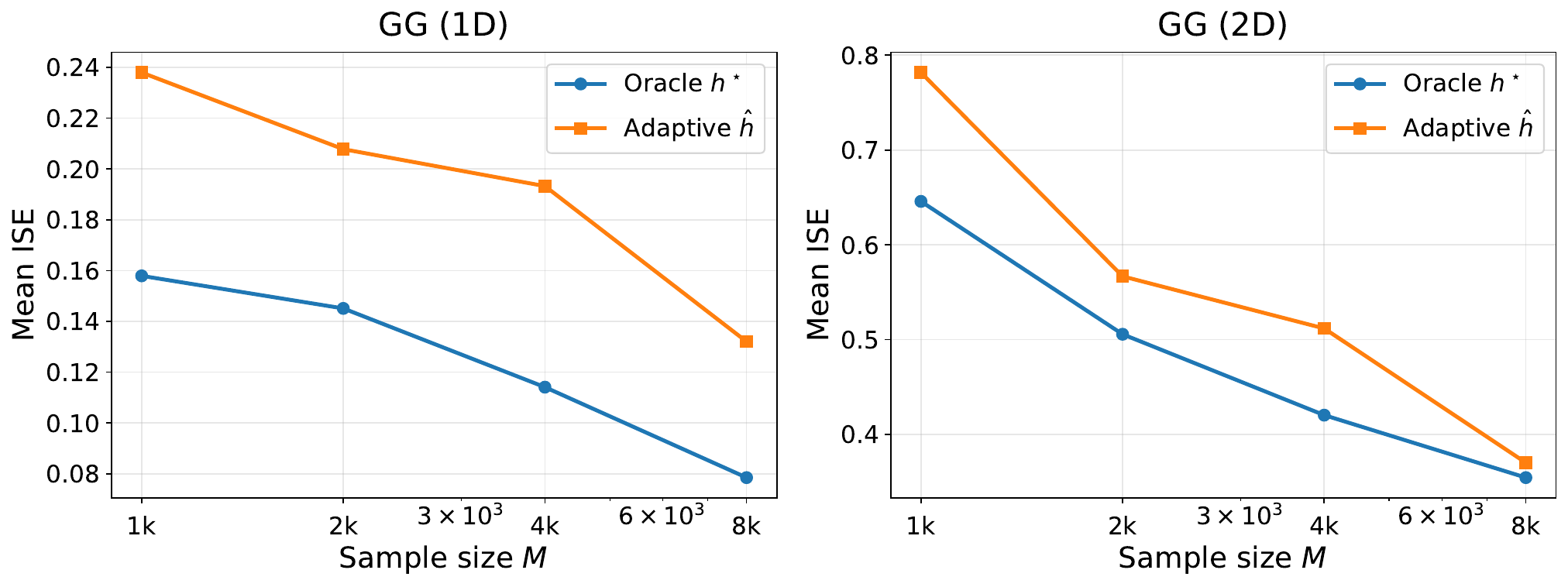}
    \hfill
    \includegraphics[width=.48\linewidth]{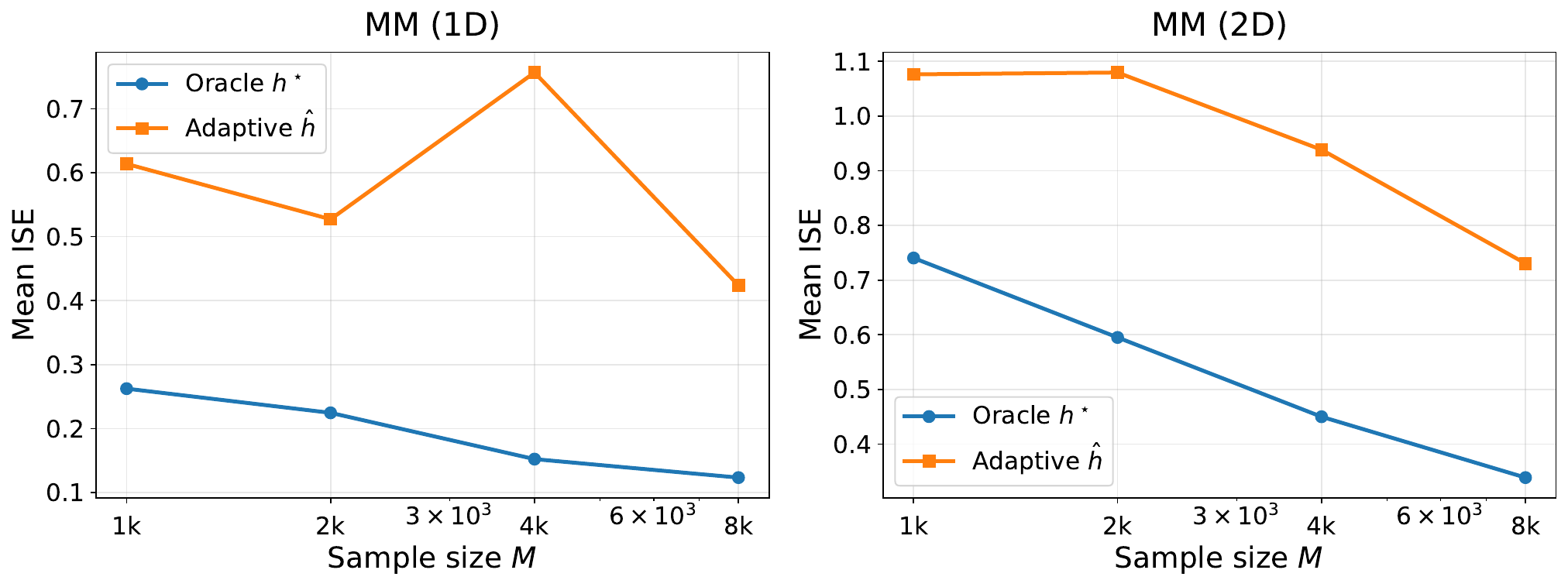}
    \caption{Secondary rate plots based on the integrated squared error.}
    \label{fig:rate-ise-appendix}
\end{figure}

\subsection{Pointwise CLT and confidence intervals}
\label{app:exp-clt}

The main-text CLT experiment is one-dimensional. For a fixed interior query $(t_0,x_0,\xi_0)$, let
\begin{equation}
\psi_{t_0,x_0,\xi_0}(y)=\bigl(y-x_0-\Delta(t_0)a^\ast(t_0,x_0;\xi_0)\bigr)F(t_0,\xi_0,x_0,y).
\end{equation}
In the diagonal bandwidth case, write $\hat f_h:=\hat f_1=\hat f_2$ and $\hat D_h(t,x;\xi):=\hat g_1(t,x;\xi)/\hat f_h(\xi)$. The plug-in variance estimator is
\begin{equation}
\hat\Sigma(t_0,x_0;\xi_0)
=
\frac{R(K)}{\hat f_h(\xi_0)\Delta(t_0)^2\hat D_h(t_0,x_0;\xi_0)^2}
\left(
\widehat{\mathbb E}[\hat\psi(X_u)^2\mid X_s=\xi_0]
-
\widehat{\mathbb E}[\hat\psi(X_u)\mid X_s=\xi_0]^2
\right),
\end{equation}
where
\begin{align}
&\hat\psi(y)=\bigl(y-x_0-\Delta(t_0)\hat a_h(t_0,x_0;\xi_0)\bigr)F(t_0,\xi_0,x_0,y),
\\
&\widehat{\mathbb E}[\phi(X_u)\mid X_s=\xi_0]
=
\frac{M^{-1}\sum_{m=1}^M \phi(X_u^{(m)})K_h(X_s^{(m)}-\xi_0)}{\hat f_h(\xi_0)}.
\end{align}

The CLT bandwidth must satisfy $M h_M^d\to\infty$ and $h_M^\beta\sqrt{M h_M^d}\to0$. We therefore use genuinely undersmoothed sequences of the form $h_M=cM^{-\alpha}$ with $(2\beta+d)^{-1}<\alpha<d^{-1}$. In the final reported one-dimensional runs we first screen the candidate exponents $\alpha\in\{0.22,0.24,0.26,0.28\}$ on $M\in\{4\times10^3,8\times10^3\}$ with $R_M=150$, then fix the final rule at $\alpha=0.22$ for GG1 and $\alpha=0.28$ for MM1, both with $c=1.0$. For repetition $r$, the standardized statistic is
\begin{equation}
Z_M^{(r)}
=
\frac{\sqrt{M h_M^d}\bigl(\hat a_{h_M}^{(r)}(t_0,x_0;\xi_0)-a^\ast(t_0,x_0;\xi_0)\bigr)}
{\sqrt{\hat\Sigma^{(r)}(t_0,x_0;\xi_0)}}.
\end{equation}

The final CLT runs use $R_M=300$ repetitions at each $M\in\{10^3,2\times10^3,4\times10^3,8\times10^3\}$. For each sample size we record the empirical mean $\bar Z_M$, empirical variance $\widehat{\Var}(Z_M)$, the empirical coverage of
\begin{equation}
\hat a_{h_M}^{(r)}(t_0,x_0;\xi_0)\pm1.96\sqrt{\hat\Sigma^{(r)}(t_0,x_0;\xi_0)/(M h_M^d)},
\end{equation}
and formal Shapiro--Wilk / Anderson--Darling diagnostics of the standardized replicates. QQ-plots against $N(0,1)$ are saved for the two largest sample sizes and are used as the main visual CLT diagnostic in the paper.

\begin{table}[t]
    \centering
    \caption{Final pointwise CLT diagnostics at the fixed interior query point: empirical mean and variance of the standardized statistic and empirical coverage of nominal $95\%$ confidence intervals in the $R_M=300$ runs.}
    \label{tab:clt-diagnostics}
    \begin{tabular}{lcccc}
        \toprule
        Testbed & $M$ & $\bar Z_M$ & $\widehat{\mathrm{Var}}(Z_M)$ & Coverage (\%) \\
        \midrule
        GG & $10^3$ & $-0.029120$ & $1.160508$ & $92.67$ \\
        GG & $2\times 10^3$ & $-0.083007$ & $0.964855$ & $96.33$ \\
        GG & $4\times 10^3$ & $0.056018$ & $0.828702$ & $96.67$ \\
        GG & $8\times 10^3$ & $-0.015311$ & $0.974488$ & $96.67$ \\
        \midrule
        MM & $10^3$ & $0.011673$ & $1.180449$ & $92.33$ \\
        MM & $2\times 10^3$ & $-0.054668$ & $0.964488$ & $96.67$ \\
        MM & $4\times 10^3$ & $0.133135$ & $0.879781$ & $96.33$ \\
        MM & $8\times 10^3$ & $0.041255$ & $0.894031$ & $96.00$ \\
        \bottomrule
    \end{tabular}
\end{table}

\begin{table}[t]
    \centering
    \caption{Formal normality diagnostics for the final pointwise CLT runs. We report the Shapiro--Wilk $p$-value and the Anderson--Darling statistic together with its $5\%$ critical value. The indicator \texttt{reject@5\%} equals $1$ if the Anderson--Darling statistic exceeds the $5\%$ critical value.}
    \label{tab:clt-normality}
    \begin{tabular}{lccccc}
        \toprule
        Testbed & $M$ & Shapiro $p$ & Anderson stat & Anderson $5\%$ crit & reject@5\% \\
        \midrule
        GG & $10^3$ & $0.139036$ & $0.689650$ & $0.750000$ & $0$ \\
        GG & $2\times 10^3$ & $0.929805$ & $0.253814$ & $0.750000$ & $0$ \\
        GG & $4\times 10^3$ & $0.180917$ & $0.536864$ & $0.750000$ & $0$ \\
        GG & $8\times 10^3$ & $0.584810$ & $0.483753$ & $0.750000$ & $0$ \\
        \midrule
        MM & $10^3$ & $0.094434$ & $0.487901$ & $0.750000$ & $0$ \\
        MM & $2\times 10^3$ & $0.139434$ & $0.500433$ & $0.750000$ & $0$ \\
        MM & $4\times 10^3$ & $0.905583$ & $0.173670$ & $0.750000$ & $0$ \\
        MM & $8\times 10^3$ & $0.466130$ & $0.414661$ & $0.750000$ & $0$ \\
        \bottomrule
    \end{tabular}
\end{table}

\paragraph{CLT plots and saved artifacts.}
The manuscript-level CLT visual is \Cref{fig:clt-adapt-edge}. To avoid duplication, we do not repeat it here. This Appendix instead records the numerical diagnostics in \Cref{tab:clt-diagnostics,tab:clt-normality}. The raw replicate-level statistics are saved as \texttt{pointwise\_clt.csv} files, and the QQ-plots generated during the final runs are released with the repository under the standard figure directories.

\subsection{Terminal-edge experiment}
\label{app:exp-edge}

To isolate the terminal singularity, we fix one moderate-to-large sample size and evaluate the estimator at $t\in\{u-0.40,u-0.25,u-0.15,u-0.10,u-0.05\}$. The bandwidth is selected once at the interior time $t_0=0.6$ and then held fixed. For this terminal-edge experiment only, we abbreviate \begin{equation} E_{\infty,\mathcal G}(t) :=E_{\infty,\mathcal G}(t,\xi_0;\hat h) =\max_{x\in\mathcal G_x} |\hat a_{\hat h}(t,x;\xi_0)-a^\ast(t,x;\xi_0)|. \end{equation} The primary diagnostic is this quantity together with its rescaled version $\Delta(t)E_{\infty,\mathcal G}(t)$. We also record the corresponding ISE curves as secondary diagnostics.

\paragraph{Terminal-edge plots.}
The terminal-edge curves are shown in the main text in \Cref{fig:clt-adapt-edge}. To avoid duplication, we do not repeat them here. This Appendix records only the protocol and the interpretation: the raw sup-grid error increases as $t\uparrow u$, while the rescaled quantity $\Delta(t)E_{\infty,\mathcal G}(t)$ remains substantially flatter.

\subsection{Bounded-support stress test}
\label{app:exp-stress}

To probe the proof role of bounded support, we run a one-dimensional stress test at the fixed interior queries GG1: $(t_0,x_0,\xi_0)=(0.6,0.2,0)$ and MM1: $(0.6,0.3,0.8)$, with the same canonical selector as in the final adaptive runs (raw-max, one-sided, $\kappa_{\mathrm{pair}}=\kappa_{\mathrm{final}}=2$, and the same lower-bandwidth floor), at the fixed sample size $M_{\mathrm{stress}}=4000$ with $R_{\mathrm{stress}}=100$ repetitions. For GG1, the wide-support variant changes only the truncation box from $[-3,3]$ to $[-5,5]$. For MM1, the wide-support variant also inflates the source and conditional variances from $(0.45)^2,(0.25)^2,(0.30)^2$ to $(0.8)^2,(0.9)^2,(1.0)^2$, while keeping the same means, linear maps, and logistic gate.

For each repetition, with selected bandwidth $\hat h$, we record the empirical variances of the raw weighted numerator and denominator summands
$W_m^N:=X_u^{(m)}F(t_0,\xi_0,x_0,X_u^{(m)})K_{\hat h}(X_s^{(m)}-\xi_0)$ and
$W_m^D:=F(t_0,\xi_0,x_0,X_u^{(m)})K_{\hat h}(X_s^{(m)}-\xi_0)$,
the pointwise empirical denominator $\hat D:=\hat D_{\hat h}(t_0,x_0;\xi_0)$, and the selected-estimator errors. To quantify denominator fragility, we set
\begin{equation}
\tau_D:=\tfrac14\ \operatorname{median}_{r,\mathrm{compact}}(\hat D^{(r)})
\end{equation}
separately for each testbed and report $\Pr(\hat D\le \tau_D)$.

\begin{table}[t]
    \centering
    \caption{Bounded-support stress test at $M_{\mathrm{stress}}=4000$ and $R_{\mathrm{stress}}=100$.}
    \label{tab:stress}
    \small
    \begin{tabular}{llcccc}
        \toprule
        Testbed & Setting & $\mathbb E[\widehat{\Var}(W^N)]$ & $\mathbb E[\widehat{\Var}(W^D)]$ & $\Pr(\hat D\le\tau_D)$ & med.\ $E_{\infty,\mathcal G}$ \\
        \midrule
        GG1 & compact & $0.1579$ & $0.7101$ & $0.00$ & $0.1839$ \\
        GG1 & wide    & $0.1732$ & $0.7843$ & $0.00$ & $0.1570$ \\
        \midrule
        MM1 & compact & $0.2638$ & $0.4047$ & $0.00$ & $0.2387$ \\
        MM1 & wide    & $0.7708$ & $2.1689$ & $0.10$ & $2.4260$ \\
        \bottomrule
    \end{tabular}
\end{table}

\paragraph{Interpretation.}
The stress effect is mild in GG1: the wide-support perturbation raises $\mathbb E[\widehat{\Var}(W^N)]$ and $\mathbb E[\widehat{\Var}(W^D)]$ only by factors $1.10$ and $1.10$, while the median sup-grid error and median ISE slightly decrease from $0.1839$ to $0.1570$ and from $0.1712$ to $0.1460$. Thus weakening compact support is not automatically catastrophic in an easier Gaussian testbed.

The harder MM1 testbed is qualitatively different. The wide-support perturbation multiplies $\mathbb E[\widehat{\Var}(W^N)]$ by $2.92$ and $\mathbb E[\widehat{\Var}(W^D)]$ by $5.36$, raises $\Pr(\hat D\le\tau_D)$ from $0$ to $0.10$, inflates the median sup-grid error from $0.2387$ to $2.4260$ (a factor $10.16$), and inflates the median ISE from $0.2499$ to $1.6083$ (a factor $6.44$). At the same time, the mean oracle ratio rises from $2.09$ to $7.12$, the mean selected bandwidth drops from $0.3221$ to $0.0655$, and the boundary-hit rate rises from $0$ to $0.10$. This is the expected failure mode: the deterministic mean structure is unchanged, but once the support control is weakened, the weighted empirical terms in the ratio estimator become substantially more variable, denominator stability deteriorates, and the selector is pulled toward pathological small-bandwidth behavior. This is precisely the role bounded support plays in the proof.

\subsection{Implementation notes and optional selector variants}
\label{app:exp-secondary}

The implementation is organized around a small set of core modules and driver scripts. The synthetic pair laws are defined in \texttt{src/sbdrift/models.py}; deterministic truth evaluation is implemented in \texttt{src/sbdrift/truth\_engine.py}; the kernel estimator is implemented in \texttt{src/sbdrift/estimator.py}; and the kernel definitions are collected in \texttt{src/sbdrift/kernels.py}. The experiment drivers are \texttt{scripts/00\_preflight.py}, \texttt{scripts/01\_rate.py}, and \texttt{scripts/02\_clt.py}. The pre-flight script writes numerical checks of the marginal density, denominator floor, truth-convergence probe, and sampling sanity checks. The rate script writes raw per-bandwidth and selected-bandwidth CSV files, processed summaries, and rate plots. The CLT script writes per-repetition pointwise statistics, processed summaries, and per-$M$ QQ-plot files, and also records the undersmoothing-screen runs together with the formal normality diagnostics used to choose the final one-dimensional CLT bandwidths. All experiments are implemented in \texttt{Python} using \texttt{NumPy}, \texttt{SciPy}, \texttt{pandas}, and \texttt{matplotlib}.

Each run is driven by a YAML configuration file specifying the synthetic family, support and evaluation boxes, query points, bandwidth grid parameters, repetition counts, and random seeds. The rate pipeline caches deterministic truth values on the evaluation grid to avoid recomputation across repeated runs. Outputs are written to the standard experiment directories \texttt{results/raw/}, \texttt{results/processed/}, and \texttt{results/figures/}. The repository state used for the paper contains the processed pre-flight summaries, the final theorem-facing rate and CLT summaries, the undersmoothing-screen and final CLT runs, the final stabilized adaptive-selector summaries, and the figure files corresponding to the saved experiment outputs.

The codebase also retains optional alternative selector implementations, including trimmed-max, two-sided, and ISE-based discrepancy rules. These remain available for users interested in separate comparison studies, but they are not used for the paper's quantitative claims. Accordingly, the reproducibility claim supported by the current codebase is the following: the repository contains the exact synthetic-model definitions, the deterministic truth engine, the kernel estimator implementation, the pre-flight/rate/CLT experiment drivers, the YAML configuration files, and the saved raw, processed, and figure outputs used to assemble the final tables and theorem-facing plots. Exact command lines, random seeds, repetition counts, the selected undersmoothing exponents, the stabilized bandwidth-floor rule, and run-specific parameter choices are recorded in the experiment logs and processed summaries.

\end{document}